\documentclass[reqno,12pt,a4paper]{amsart}
\UseRawInputEncoding

\usepackage{tikz}
\usetikzlibrary{matrix,arrows,calc,snakes,patterns,decorations.markings}
\newtheorem{algorithm}[subsection]{Algorithm}
\newtheorem{observation}[subsection]{Observation}

\usepackage{makecell}
\usepackage[mathscr]{eucal}
\usepackage{graphics,epic}
\usepackage{amsfonts}
\usepackage{amscd}
\usepackage{latexsym}
\usepackage{amsmath,amssymb, amsthm, stmaryrd, bm, bbm}
\usepackage[all,2cell]{xy}
\usepackage{mathrsfs}

\newtheorem{theorem}[subsection]{Theorem}
\newtheorem*{theorem*}{Theorem}

\newtheorem{lemma}[subsection]{Lemma}

\newtheorem{corollary}[subsection]{Corollary}

\newtheorem*{conjecture*}{Conjecture}

\newtheorem*{question*}{Question}
\theoremstyle{remark}

\newtheorem{example}[subsection]{Example}
\theoremstyle{definition}
\newtheorem{definition}[subsection]{Definition}

\newcommand{\prdim}{\opname{pr.dim}\nolimits}
\newcommand{\Aus}{\opname{Aus}\nolimits}
\newcommand{\MCM}{\opname{MCM}\nolimits}
\newcommand{\krdim}{\opname{kr.dim}\nolimits}
\newcommand{\brutegeq}[1]{\beta_{\geq #1}}
\newcommand{\ie}{{\em i.e.~}\ }
\newcommand{\confer}{{\em cf.~}\ }
\newcommand{\eg}{{\em e.g.~}\ }
\newcommand{\ko}{\: , \;}
\newcommand{\ul}[1]{\underline{#1}}
\newcommand{\ol}[1]{\overline{#1}}

\newcommand{\opname}[1]{\operatorname{\mathsf{#1}}}
\newcommand{\opnamestar}[1]{\operatorname*{\mathsf{#1}}}
\renewcommand{\mod}{\opname{mod}\nolimits}
\newcommand{\fdmod}{\opname{fdmod}\nolimits}
\newcommand{\grmod}{\opname{grmod}\nolimits}
\newcommand{\proj}{\opname{proj}\nolimits}
\newcommand{\Proj}{\opname{Proj}\nolimits}
\newcommand{\Inj}{\opname{Inj}\nolimits}
\newcommand{\Mod}{\opname{Mod}\nolimits}
\newcommand{\Pcm}{\opname{Pcm}\nolimits}
\newcommand{\Lex}{\opname{Lex}\nolimits}
\newcommand{\Grmod}{\opname{Grmod}\nolimits}
\newcommand{\add}{\opname{add}\nolimits}
\newcommand{\op}{^{op}}
\newcommand{\der}{\cd}
\newcommand{\hfd}[1]{\ch_{fd}(#1)/\ac_{fd}(#1)}
\newcommand{\Gr}{\opname{Gr}\nolimits}
\newcommand{\dimv}{\underline{\dim}\,}
\renewcommand{\Im}{\opname{Im}\nolimits}
\newcommand{\Ker}{\opname{Ker}\nolimits}
\newcommand{\inprod}[1]{\langle #1 \rangle}
\newcommand{\rank}{\opname{rank}\nolimits}
\newcommand{\fpr}{\opname{fpr}\nolimits}
\newcommand{\rad}{\opname{rad}\nolimits}

\newcommand{\grproj}{\opname{grproj}}
\newcommand{\GrProj}{\opname{GrProj}}
\newcommand{\ind}{\opname{ind}}
\newcommand{\coind}{\opname{coind}}

\newcommand{\bp}{\mathbf{p}}
\newcommand{\bi}{\mathbf{i}}

\newcommand{\dgcat}{\opname{dgcat}\nolimits}
\newcommand{\Ho}{\opname{Ho}\nolimits}
\newcommand{\rep}{\opname{rep}\nolimits}
\newcommand{\Rep}{\opname{Rep}\nolimits}
\newcommand{\enh}{\opname{enh}\nolimits}
\newcommand{\eff}{\opname{eff}\nolimits}
\newcommand{\cyd}[1]{#1\mbox{-}\opname{cyd}\nolimits}
\newcommand{\pretr}{\opname{pretr}\nolimits}
\newcommand{\ex}{\opname{ex}\nolimits}
\newcommand{\ac}{\mathcal{A}c}

\newcommand{\colim}{\opname{colim}}
\newcommand{\cok}{\opname{cok}\nolimits}
\newcommand{\im}{\opname{im}\nolimits}
\renewcommand{\ker}{\opname{ker}\nolimits}
\newcommand{\obj}{\opname{obj}\nolimits}
\newcommand{\Trs}{\opname{Trs}\nolimits}
\newcommand{\tria}{\opname{tria}\nolimits}
\newcommand{\thick}{\opname{thick}\nolimits}
\newcommand{\Tria}{\opname{Tria}\nolimits}
\newcommand{\per}{\opname{per}\nolimits}
\newcommand{\StPer}{\opname{StPer}\nolimits}

\DeclareMathOperator{\Coh}{\mathsf{Coh}}

\newcommand{\Z}{\mathbb{Z}}
\newcommand{\N}{\mathbb{N}}
\newcommand{\Q}{\mathbb{Q}}
\newcommand{\C}{\mathbb{C}}
\newcommand{\fq}{{\mathbb F}_q}
\renewcommand{\P}{\mathbb{P}}
\newcommand{\T}{\mathbb{T}}
\newcommand{\F}{\mathbb{F}}
\newcommand{\G}{\mathbb{G}}
\newcommand{\I}{\mathbb{I}}
\newcommand{\ra}{\rightarrow}
\newcommand{\iso}{\stackrel{_\sim}{\rightarrow}}
\newcommand{\id}{\mathbf{1}}
\newcommand{\sgn}{\opname{sgn}}

\newcommand{\Hom}{\opname{Hom}}
\newcommand{\End}{\opname{End}}
\newcommand{\Aut}{\opname{Aut}}
\newcommand{\go}{\opname{G_0}}
\newcommand{\RHom}{\opname{RHom}}
\newcommand{\REnd}{\opname{REnd}}
\newcommand{\cHom}{\mathcal{H}\it{om}}
\newcommand{\cEnd}{\mathcal{E}\it{nd}}
\newcommand{\cGrEnd}{\mathcal{G}\it{r}\mathcal{E}nd}
\newcommand{\HOM}{\opname{Hom^\bullet}}
\newcommand{\Ext}{\opname{Ext}}
\newcommand{\Hominf}{\raisebox{0ex}[2ex][0ex]{$\overset{\,\infty}{
                              \raisebox{0ex}[1ex][0ex]{$\mathsf{Hom}$}
                                                                 }$}}
\newcommand{\sHominf}{\raisebox{0ex}[2ex][0ex]{$\scriptsize
\overset{\scriptsize \,\infty}{\raisebox{0ex}[1ex][0ex]{$\scriptsize
\mathsf{Hom}$}}$}}
\newcommand{\GL}{\opname{GL}}

\newcommand{\ten}{\otimes}
\newcommand{\lten}{\overset{\opname{L}}{\ten}}
\newcommand{\tensinf}{\overset{\infty}{\ten}}
\newcommand{\Tor}{\opname{Tor}}
\newcommand{\supp}{\opname{supp}}
\newcommand{\Cone}{\opname{Cone}}
\newcommand{\acyc}{\opname{acyc}}

\newcommand{\Tot}{\opname{Tot}}

\newcommand{\gldim}{\opname{gldim}\nolimits}

\newcommand{\ca}{{\mathcal A}}
\newcommand{\cb}{{\mathcal B}}
\newcommand{\cc}{{\mathcal C}}
\newcommand{\cd}{{\mathcal D}}
\newcommand{\ce}{{\mathcal E}}
\newcommand{\cf}{{\mathcal F}}
\newcommand{\cg}{{\mathcal G}}
\newcommand{\ch}{{\mathcal H}}
\newcommand{\ci}{{\mathcal I}}
\newcommand{\cj}{{\mathcal J}}
\newcommand{\ck}{{\mathcal K}}
\newcommand{\cl}{{\mathcal L}}
\newcommand{\cm}{{\mathcal M}}
\newcommand{\cn}{{\mathcal N}}
\newcommand{\co}{{\mathcal O}}
\newcommand{\cp}{{\mathcal P}}
\newcommand{\cq}{{\mathcal Q}}
\newcommand{\cs}{{\mathcal S}}
\newcommand{\ct}{{\mathcal T}}
\newcommand{\cu}{{\mathcal U}}
\newcommand{\cv}{{\mathcal V}}
\newcommand{\cw}{{\mathcal W}}
\newcommand{\cx}{{\mathcal X}}
\newcommand{\cy}{{\mathcal Y}}
\newcommand{\cz}{{\mathcal Z}}

\newcommand{\wmny}{W_{M,N}^Y}\newcommand{\wnmy}{W_{N,M}^Y}
\newcommand{\eps}{\varepsilon}
\newcommand{\DQ}{{\mathcal D}_Q}
\newcommand{\kos}{K_0^{split}(\cc)}

\newcommand{\tilting}{tilting }
\newcommand{\exceptional}{exceptional }
\renewcommand{\hat}[1]{\widehat{#1}}

\newcommand{\m}{\mathfrak{m}}
\newcommand{\del}{\partial}

\newcommand{\silt}{\opname{silt}}

\renewcommand{\baselinestretch}{1.1}

\numberwithin{equation}{section}

\begin{document}

\title[]{Some examples of silted algebras of Dynkin type}

\author{ Ruo-Yun Xing}
\thanks{MSC2020: 16E35, 16G20}
\thanks{Key words: silted algebra; tilted algebra; 2-term silting complexes; strictly shod algebra}
\address{
}

\email{}

\begin{abstract}

This paper  studies silted algebras, namely, endomorphism algebras of 2-term silting complexes, over path algebras of Dynkin quivers.
We will describe an algorithm to produce all basic 2-term silting complexes over the path algebra of a Dynkin quiver, and use this algorithm to compute some examples.
%Based on the algorithm of Happel and Ringel which produces tilting modules, we obtain an algorithm to produce all basic 2-term silting complexes. We then apply this  algorithm to the paths algebras of certain Dynkin quivers and classify all silted algebras. We notice that in the type $\mathbb{A}$ examples all silted algebras are tilted algebras, while in type $\mathbb{D }$ there are a few strictly shod algebras, although almost all silted algebras are tilted.
%
%In Section 2 we introduce quivers and path algebras, tilting modules, tilted algebras, shod algebras and derived  categories, and also introduce some notation.
%
%In Section 3 we recall the definition and basic properties of 2-term silting complexes and silted algebras.

%In Section 4 we will describe an algorithm to produce all basic 2-term silting complexes over the path algebra of a Dynkin quiver, and use this algorithm to compute some examples.

\end{abstract}

\maketitle

\section{Introduction}\label{s:introduction}
\medskip

%%\bigskip
%%\renewcommand{\chaptermark}[1]{\markboth{\small Chapter~\thechapter\,\quad #1}{}}
%%\pagenumbering{arabic}\setcounter{page}{1}
%%\fancyhf{} \fancyhead[EC]{\kai\leftmark}
%%\fancyhead[OC]{\kai\leftmark}\fancyfoot[EC,OC]{\thepage}

Silting theory  plays an important role in representation theory   because it is closely related to many research fields, such as t-structures (\cite{kv}\cite{kY}\cite{kN}), cluster-tilting theory (\cite{BRT}), and tilting theory (\cite{AI}\cite{AIR}\cite{BZ_2016}).
Silted algebras are defined as endomorphism algebras  of 2-term silting complexes  over hereditary algebras (\cite{BZ}). According to a remarkable result of Buan and Zhou (\cite{BZ}),  if $A$ is a silted algebra, then $A$  either is a tilted algebra, or $A$ is a strictly shod algebra, that is, $A$ has global dimension  3 and any $A$-module has projective or injective dimension no greater than 1.

In this paper, a complete list of  silted algebras is given for path algebras of certain Dynkin quivers. The main result is as follows:

(1) there is 1 silted algebra of the quiver  $\circ$: $\circ$ (tilted algebra of type $\mathbb{A}_1$)

(2) there are 2 silted algebras of the quiver $\circ\rightarrow\circ$:
$\circ\rightarrow\circ$ (tilted algebra of type $\mathbb{A}_2$) and
$\circ ~~~ \circ$ (tilted algebra of type
$\mathbb{A}_1\amalg\mathbb{A}_1$)

(3) there are 5 silted algebras  of the quiver $\circ\longrightarrow \circ \longrightarrow \circ$, forming two families:

\begin{itemize}
\item[(\romannumeral1)]tilted algebras of type $\mathbb{A}_3$:

$\circ\longrightarrow \circ \longrightarrow \circ,$
~~~\text{}~~~
$\circ\longrightarrow \circ \longleftarrow \circ,$
~~~\text{}~~~
$\circ\longleftarrow \circ \longrightarrow \circ,$
~~~\text{}~~~
$\xymatrix@C=5pt{\circ\ar[rr]\ar@/^/@{.}[rrrr]_{}
&&\circ\ar[rr]&&\circ}$

\item[(\romannumeral2)]tilted algebra of type
 $\mathbb{A}_2\amalg\mathbb{A}_1$:
 $\circ\longrightarrow \circ ~~~ \circ$
\end{itemize}

(4) there are 15 silted algebras  of the quiver $\circ\longrightarrow \circ \longrightarrow \circ\longrightarrow\circ$, forming 3 families:

\begin{itemize}
\item[(\romannumeral1)]%tilted algebras of type $\mathbb{A}_4$:
tilted algebras of type $\mathbb{A}_4$, for details see Example \ref{A4} (I);
%$\xymatrix@C=5pt{\circ\ar[rr]&&\circ\ar[rr]&&
%\circ\ar[rr]&&\circ}$
%~~~\text{}~~~
%$\xymatrix@C=5pt{\circ\ar[rr]&&\circ\ar[rr]&&
%\circ&&\circ\ar[ll]}$
%~~~\text{}~~~
%$\xymatrix@C=5pt{\circ\ar[rr]&&\circ&&\circ
%\ar[ll]\ar[rr]&&\circ}$
%
%
%
%$\xymatrix@C=5pt{\circ&&\circ\ar[ll]\ar[rr]&&
%\circ\ar[rr]&&\circ}$
%~~~\text{}~~~
%$\xymatrix@C=5pt{\circ\ar[rr]\ar@/^/@{.}[rrrr]_{}&&
%\circ\ar[rr]&&\circ\ar[rr]&&\circ}$
%~~~\text{}~~~
%$\xymatrix@C=5pt{\circ\ar[rr]&&
%\circ\ar[rr]\ar@/^/@{.}[rrrr]_{}&&\circ\ar[rr]&&\circ}$
%
%$\xymatrix@C=5pt{\circ\ar[rr]&&\circ&&\circ\ar[ll]
%&&\circ\ar[ll]\ar@/_/@{.}[llll]_{}}$
%~~~\text{}~~~
%$\xymatrix@C=5pt{\circ&&\circ\ar[ll]&&\circ
%\ar[ll]\ar[rr]\ar@/_/@{.}[llll]_{}&&\circ}$
%
%$\xymatrix@C=5pt{\circ\ar[rr]&&\circ\ar[rr]&&\circ\\
%&&\circ\ar[u]\ar@/_/@{.}[rru]_{}
%}$
%~~~\text{}~~~
%$\xymatrix@C=5pt{\circ\ar[rr]\ar@/_/@{.}[rrd]_{}
%&&\circ\ar[d]\ar[rr]&&
%\circ\\
%&&\circ
%}$

\item[(\romannumeral2)]tilted algebras of type $\mathbb{A}_3\amalg\mathbb{A}_1$:

$\xymatrix@C=5pt{\circ\ar[rr]&&\circ\ar[rr]
&&\circ~~~\circ},$
~~~\text{}~~~
$\xymatrix@C=5pt{\circ&&\circ\ar[rr]\ar[ll]
&&\circ~~~\circ},$
~~~\text{}~~~
$\xymatrix@C=5pt{\circ\ar[rr]&&\circ
&&\circ\ar[ll]~~~\circ},$
~~~\text{}~~~
$\xymatrix@C=5pt{\circ\ar[rr]\ar@/^/@{.}[rrrr]_{}&&
\circ\ar[rr]&&\circ&~~~\circ}$

\item[(\romannumeral3)]tilted algebras of type $\mathbb{A}_2\amalg\mathbb{A}_2$:
$\xymatrix@C=5pt{\circ\ar[rr]&&\circ\ar[rr]~~~
\circ\ar[rr]&&\circ}$
\end{itemize}

(5) there are 13 silted algebras  of the quiver

 $$
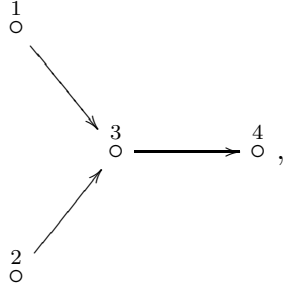
\xymatrix@C=8pt{\stackrel{1}\circ\ar[rrd]&&\\
&&\stackrel{3}\circ\ar[rrr]&&&
\stackrel{4}\circ~~~,\\
\stackrel{2}\circ\ar[rru]
}$$
%forming 4 families:
forming 4 families (see Example \ref{D4} for details on (i)(ii)(iii)):

\begin{itemize}
\item[(\romannumeral1)]tilted algebras of type $\mathbb{D}_4$;

%$\xymatrix@C=5pt{\circ\ar[rr]&&\circ\ar[rr]&&\circ\\
%&&\circ\ar[u]
%}$
%~~~\text{}~~~
%$\xymatrix@C=5pt{\circ\ar[rr]&&\circ\ar[rr]\ar[d]&&\circ\\
%&&\circ
%}$
%~~~\text{}~~~
%$\xymatrix@C=5pt{\circ&&\circ\ar[ll]\ar[d]\ar[rr]&&\circ\\&&\circ}$
%~~~\text{}~~~
%$\xymatrix@C=5pt{\circ\ar[rr]&&\circ&&\circ\ar[ll]\\&&\circ\ar[u]}$
%
%$\xymatrix@C=5pt{\circ\ar[rr]&&\circ\\\circ\ar[u]\ar[rr]\ar@//@{.}[rru]_{}&&\circ\ar[u]}$
%~~~\text{}~~~
%$\xymatrix@C=5pt{\circ\ar[rr]\ar@/^/@{.}[rrrr]_{}
%\ar@/_/@{.}[rrd]_{}&&\circ\ar[d]\ar[rr]&&\circ\\
%&&\circ
%}$
%~~~\text{}~~~
%$\xymatrix@C=5pt{\circ\ar[rr]\ar@/^/@{.}[rrrr]_{}
%&&\circ\ar[rr]&&\circ\\
%&&\circ\ar[u]\ar@/_/@{.}[rru]_{}
%}$
%~~~\text{}~~~
%$\xymatrix@C=5pt{\circ\ar[rr]\ar@/^/@{.}[rrrrrr]_{}
%&&\circ\ar[rr]&&\circ\ar[rr]&&\circ
%}$

\item[(\romannumeral2)]tilted algebras of type $\mathbb{A}_3\amalg\mathbb{A}_1$;

%$\circ\longrightarrow\circ\longleftarrow\circ~~~\circ~~~\text{}~~~ \circ\longrightarrow\circ\longrightarrow\circ~~~\circ~~~\text{}~~~ \circ\longleftarrow\circ\longrightarrow\circ~~~\circ$

\item[(\romannumeral3)]tilted algebras of type $\mathbb{A}_2\amalg\mathbb{A}_1\amalg\mathbb{A}_1$;

%$\circ\longrightarrow\circ~~~\circ~~~\circ$

\item[(\romannumeral4)]strictly shod algebra:

($s1$) $\xymatrix@C=5pt{\circ\ar[rr]\ar@/^/@{.}[rrrr]_{}
&&\circ\ar[rr]\ar@/^/@{.}[rrrr]_{}&&\circ\ar[rr]&&\circ
}$
\end{itemize}

(6) there are 62 silted algebras of the quiver

$$
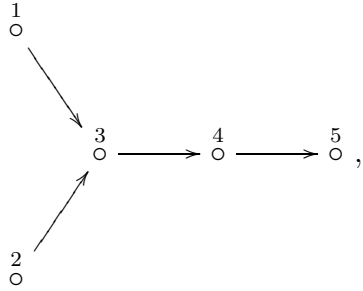
\xymatrix@C=5pt{\stackrel{1}\circ\ar[rrd]&&\\
&&\stackrel{3}\circ\ar[rrr]&&&
\stackrel{4}\circ\ar[rrr]&&&\stackrel{5}\circ~~~,\\
\stackrel{2}\circ\ar[rru]
}$$
forming 6 families (see Section 3.3.2 for details on (i)-(v)):
\begin{itemize}
\item[(\romannumeral1)]tilted algebras of type $\mathbb{D}_5$;

\item[(\romannumeral2)]tilted algebras of type $\mathbb{D}_4\amalg\mathbb{A}_1$;

%$\xymatrix@C=5pt{\circ\ar[rr]&&\circ\ar[rr]
%\ar[d]&&\circ~~~\circ\\
%&&\circ
%}$
%~~~\text{}~~~
%$\xymatrix@C=5pt{\circ\ar[rr]&&\circ\ar[rr]
%&&\circ~~~\circ\\
%&&\circ\ar[u]
%}$
%~~~\text{}~~~
%$\xymatrix@C=5pt{\circ\ar[rr]&&\circ&&\circ\ar[ll]
%~~~\circ\\
%&&\circ\ar[u]
%}$
%~~~\text{}~~~
%$\xymatrix@C=5pt{\circ&&\circ\ar[ll]\ar[rr]\ar[d]&&
%\circ~~~\circ\\
%&&\circ
%}$
%
%$\xymatrix@C=5pt{\circ\ar[rr]&&\circ&~~~\circ\\
%\circ\ar[u]\ar[rr]\ar@//@{.}[rru]_{}&&\circ\ar[u]
%}$
%~~~\text{}~~~
%$\xymatrix@C=5pt{\circ\ar[rr]&&\circ\ar[rr]
%&&\circ&~~~\circ\\
%&&\circ\ar[u]\ar@/_/@{.}[rru]_{}
%}$
%~~~\text{}~~~
%$\xymatrix@C=5pt{\circ\ar[rr]\ar@/^/@{.}[rrrr]_{}
%&&\circ\ar[rr]
%&&\circ&~~~\circ\\
%&&\circ\ar[u]\ar@/_/@{.}[rru]_{}
%}$
%
%
%
%$\xymatrix@C=5pt{\circ\ar[rr]\ar@/^/@{.}[rrrrrr]_{}&&
%\circ\ar[rr]&&\circ\ar[rr]&&
%\circ~~~\circ
%}$

\item[(\romannumeral3)]tilted algebras of type $\mathbb{A}_4\amalg\mathbb{A}_1$;

%$\xymatrix@C=5pt{\circ\ar[rr]&&\circ\ar[rr]&&
%\circ\ar[rr]&&\circ~~~\circ
%}$
%~~~\text{}~~~
%$\xymatrix@C=5pt{\circ\ar[rr]&&\circ\ar[rr]&&
%\circ&&\circ\ar[ll]~~~\circ
%}$
%~~~\text{}~~~
%$\xymatrix@C=5pt{\circ\ar[rr]&&\circ&&\circ\ar[ll]
%\ar[rr]&&\circ~~~\circ
%}$

\item[(\romannumeral4)]tilted algebras of type $\mathbb{A}_3\amalg\mathbb{A}_2$;

%$\xymatrix@C=5pt{\circ\ar[rr]&&\circ~~~
%\circ\ar[rr]&&\circ\ar[rr]&&\circ
%}$
%~~~\text{}~~~
%$\xymatrix@C=5pt{\circ\ar[rr]&&\circ~~~
%\circ\ar[rr]&&\circ&&\circ\ar[ll]
%}$
%~~~\text{}~~~
%$\xymatrix@C=5pt{\circ\ar[rr]&&\circ~~~
%\circ&&\circ\ar[ll]\ar[rr]&&\circ
%}$

\item[(\romannumeral5)]tilted algebras of type $\mathbb{A}_3\amalg\mathbb{A}_1\amalg\mathbb{A}_1$;

%$\xymatrix@C=5pt{\circ\ar[rr]&&\circ\ar[rr]
%&&\circ~~~\circ~~~\circ
%}$
%~~~\text{}~~~
%$\xymatrix@C=5pt{\circ&&\circ\ar[ll]\ar[rr]&&\circ
%~~~\circ~~~\circ
%}$
%~~~\text{}~~~
%$\xymatrix@C=5pt{\circ\ar[rr]&&\circ&&\circ\ar[ll]
%~~~\circ~~~\circ
%}$
%
%$\xymatrix@C=5pt{\circ\ar[rr]\ar@/^/@{.}[rrrr]_{}&&
%\circ\ar[rr]&&\circ&~~~\circ~~~\circ
%}$

\item[(\romannumeral6)]strictly shod algebras:

($s2$) $\xymatrix@C=5pt{\circ\ar[rr]\ar@/^/@{.}[rrrrrr]_{}&&
\circ\ar[rr]&&\circ\ar[rr]\ar@/^/@{.}[rrrr]_{}&&
\circ\ar[rr]&&\circ
}$

($s3$) $\xymatrix@C=5pt{\circ\ar[rr]\ar@/^/@{.}[rrrr]&&
\circ\ar[rr]\ar@/^/@{.}[rrrr]&&\circ\ar[rr]&&
\circ&&\circ\ar[ll]
}$

($s4$) $\xymatrix@C=5pt{\circ\ar[rr]\ar@/^/@{.}[rrrr]&&
\circ\ar[rr]\ar@/^/@{.}[rrrr]&&\circ\ar[rr]&&
\circ\ar[rr]&&\circ
}$

($s5$) $
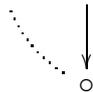
\xymatrix@C=5pt{\circ\ar[rr]\ar@/^/@{.}[rrrr]_{}
\ar@/_/@{.}[rrd]_{}&&\circ\ar[rr]
\ar[d]\ar@/^/@{.}[rrrr]_{}&&
\circ\ar[rr]&&\circ\\
&&\circ
}$

\end{itemize}

 For a brief summary,   all of these type $\mathbb{A}$ silted algebras are tilted algebras, all but one of these type $\mathbb{D}_4$  silted algebras are tilted algebras, and all but four of these type $\mathbb{D}_5$ silted algebras are tilted algebras.  Therefore, from these examples we obtain  five strictly shod algebras $(s1)-(s5)$.

In order to classify these silted algabras, we first classify the 2-term silting complexes. For this purpose, we develop an algorithm  (Algorithm \ref{al2}) based on the algorithm of Happel and Ringel  classifying  tilting modules.  Given a Dynkin quiver $Q$, we can produce all the 2-term silting complexes by using this algorithm. Then  we calculate the endomorphism algebra of each 2-term silting complex. For types $\mathbb{A}_3$, $\mathbb{A}_4$  and $\mathbb{D}_4$, we actually classify silted algebras for all orientations,
but we do not find strictly shod algebras except in (5) and (6) above.

The structure of this paper is as follows.
%In Section 2, we mainly introduce quivers and path algebras, tilting modules, tilted algebras, shod algebras and derived categories, and also introduce some notation.
%In Section 3, we recall the definition and basic properties of  2-term silting complexes and silted algebras.
In Section 2, we recall the definitions of tilting modules, tilted algebras, 2-term silting complexes and silted algebras.
In Section 3, we describe the algorithm producing all 2-term silting complexes and calculate the concrete examples.

Throughout this paper, $K$ denotes an algebraically closed field and $D=\text{Hom}_K(-,K)$ denotes the $K$-dual. All algebras will be finite-dimensional $K$-algebras, and all modules will be finite-dimensional right modules.

\noindent{\it Acknowledgement:} This paper is based on the master thesis of the author. She is deeply grateful to her supervisor Qunhua Liu and Dong Yang for their kind supervision, and she thanks Zongzhen Xie and Houjun Zhang for carefully reading the manuscript and pointing out an error. She acknowledges support by the National Natural Science Foundation of China No. 11671207.

%\section{Preliminaries}
\section{Silted algebras}
%In this section we introduce  tilting modules, tilted algebras,  and also introduce some notation.
%we recall the definition and basic properties of 2-term silting complexes and silted algebras.
In this section, we recall the definitions of tilting modules, tilted algebras, 2-term silting complexes and silted algebras.
\medskip

\subsection*{2.1 Tilted algebras}

\begin{definition}\label{N8}\cite[Chapter 6 Definition 2.1 and Chapter 6 Corollary 4.4]{ASS}
 Let $A$ be an algebra. An $A$-module $T$ is called a \textbf{tilting module} if the following three conditions are satisfied:
\begin{itemize}
\item[(T1)]$\text{pd} T_A\leq 1$.
\item[(T2)]$\text{Ext}_A^1(T,T)=0$.

%A partial tilting module T is called a \textbf{tilting module} if it  satisfies the following additional condition:
\item[(T3)]%There exists a short exact sequence $0\longrightarrow A_A\longrightarrow T_A^{'}\longrightarrow T_A^{''}\longrightarrow 0$ with $T^{'}$, $T^{''}$ in $\text{add} T$.
    $|T|=|A|$.
\end{itemize}
\end{definition}

%\begin{corollary}\cite[Chapter6 Corollary 4.4.]{ASS}\label{c1} Let $T_A$ be a partial tilting module. Then $T_A$ is a tilting module if and only if $|T|=|A|$.
%\end{corollary}

\begin{definition}\cite[Chapter 8 Definition 3.1]{ASS}
Let $Q$ be an acyclic quiver. An algebra $B$ is said to be \textbf{tilted} of type $Q$ if there exists a tilting module $T$ over the path algebra $A=KQ$ of $Q$ such that $B=\text{End}(T_A)$.
\end{definition}

The global dimension of a tilted algebra is at most 2 \cite[Proposition 3.2]{C}.

Let $Q$ be a Dynkin quiver and $A=KQ$. In this case the condition (T1) is automatic and both $A_A$ and $D({A_A})$ are tilting modules with endomorphism algebra $A$.
According to the proof of \cite[Proposition 2.1]{HR} by Happel and Ringel, we obtain the following algorithm to produce all basic tilting modules over $A$.

\begin{algorithm}\label{al1}
 Perform the following 3 steps for all non-empty subsets $I$ of $Q_0$.

(1) Let $e=\bigoplus\limits_{i\in I}e(i)$, $P(I)=eA$, $A(I)=A/\langle e\rangle$.

(2)  For each basic tilting $A(I)$-module $N$ which, considered as an $A$-module, has no non-trivial injective direct summands, form the $A$-module $M=P(I)\oplus\tau_A^{-1}N_A$.

(3) For each $A$-module M obtained in (2), let $m\in \mathbb{N}\cup \{0\}$ be such that $|\tau_A^{-m}M|=|A|$, and $\tau_A^{-m}M$ has non-trivial injective direct summand. Form the $A$-modules $\tau_A^{-p}M, 0\leq p\leq m$.
\end{algorithm}

\begin{example}\label{A2}
Let $A$ the path algebra  of the quiver
$\stackrel{1}\circ \longrightarrow \stackrel{2}\circ $.
Then  the AR-quiver $\Gamma (\text{mod}A)$ of $\text{mod}A$  is of the form
$$
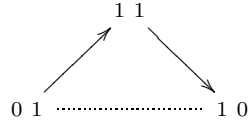
\xymatrix@C=15pt{&{\begin{smallmatrix}  1&1 \end{smallmatrix}}\ar[rd]\\
{\begin{smallmatrix}  0&1 \end{smallmatrix}}\ar@{..}[rr]\ar[ru]&&
{\begin{smallmatrix}  1&0 \end{smallmatrix}}
}$$

We apply Algorithm \ref{al1}.

($1$) $I=\{1\}$. The quiver of $A(I)=A/\langle e\rangle$ has only one vertex $2$. So $\text{mod} A(I)$ has only one tilting module $M_A={\begin{smallmatrix} 0&1 \end{smallmatrix}}$. This yields the tilting module
$$T_A=P(I)\oplus\tau_A^{-1}(M_A)=P(1)\oplus\tau_A^{-1}({\begin{smallmatrix} 0&1 \end{smallmatrix}})=D({A_A})={\begin{smallmatrix} 1&1\end{smallmatrix}}\oplus{\begin{smallmatrix} 1&0 \end{smallmatrix}}.$$
\vspace{-10pt}
\newcommand{\scirc}{{\scriptstyle \circ}}
\newcommand{\sbullet}{{\scriptstyle\bullet}}
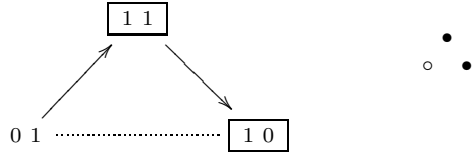
\[{\setlength{\unitlength}{0.3pt}
\begin{picture}(650,230)
\put(-20,200){
$\xymatrix@C=15pt{&\boxed{{\begin{smallmatrix}  1&1 \end{smallmatrix}}}\ar[rd]\\
{\begin{smallmatrix}  0&1 \end{smallmatrix}}\ar@{..}[rr]\ar[ru]&&
\boxed{{\begin{smallmatrix}  1&0 \end{smallmatrix}}}
}$}
\put(500,170){$
\begin{array}{*{20}{c@{\hspace{3pt}}}}
&&\sbullet&&\\[-5pt]
&\scirc&&\sbullet\\[-5pt]
\end{array}$}
\end{picture}}\]

($2$) $I=\{2\}$. The quiver of $A(I)=A/\langle e\rangle$ has only one vertex $1$. So $\text{mod} A(I)$ has only one tilting module $M_A={\begin{smallmatrix} 1&0 \end{smallmatrix}}$,  which is injective.

($3$) $I=\{1,2\}$. In this case the tilting module is $$T_A=P(I)=P(1)\oplus P(2)=A_A={\begin{smallmatrix} 1&1 \end{smallmatrix}}\oplus{\begin{smallmatrix} 0&1 \end{smallmatrix}}.$$
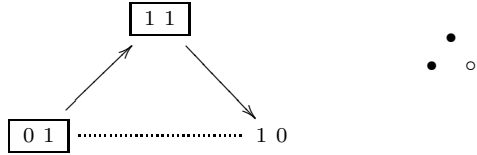
\[{\setlength{\unitlength}{0.3pt}
\begin{picture}(650,230)
\put(-20,200){
$\xymatrix@C=15pt{&\boxed{{\begin{smallmatrix}  1&1 \end{smallmatrix}}}\ar[rd]\\
\boxed{{\begin{smallmatrix}  0&1 \end{smallmatrix}}}\ar@{..}[rr]\ar[ru]&&
{\begin{smallmatrix}  1&0 \end{smallmatrix}}
}$}
\put(500,170){$
\begin{array}{*{20}{c@{\hspace{3pt}}}}
&&\sbullet&&\\[-5pt]
&\sbullet&&\scirc\\[-5pt]
\end{array}$}
\end{picture}}\]

To summarise,  $A$ has two basic tilting modules : ${\begin{smallmatrix} 1&1\end{smallmatrix}}\oplus{\begin{smallmatrix} 1&0 \end{smallmatrix}}$ and ${\begin{smallmatrix} 1&1 \end{smallmatrix}}\oplus{\begin{smallmatrix} 0&1 \end{smallmatrix}}$.
Both endomorphism algebras are isomorphism to $A$.
\end{example}

\subsection*{2.2 2-term silting complexes}
\begin{definition}\cite[Page 1]{BZ}
Let $A$ be an algebra.
Let $\textbf{P}$ be a complex in the bounded homotopy category of finitely generated projective $A$-modules $\text{K}^b(\text{proj}A)$. Then $\textbf{P}$ is called \textbf{silting} if $\Hom_{\text{K}^b(\text{proj}A)}(\textbf{P},\textbf{P}[i])=0$ for $i>0$, and if $\textbf{P}$ generates $\text{K}^b(\text{proj}A)$ as a triangulated category. Furthermore, we say that $\textbf{P}$ is \textbf{$2$-term} if $\textbf{P}$ only has non-zero terms in degrees 0 and -1.
\end{definition}

The following result is a corollary of \cite[Theorem 3.2]{AIR}.

\begin{corollary}\label{coro}
 Assume that $A$ is hereditary. Then any basic 2-term silting complex over $A$ is of the form $M\oplus P[1]$, where $P=eA$ for some idempotent e of $A$, and $M$ is a basic tilting module over $A/\langle e\rangle$. Conversely, every complex of this form is a 2-term silting complex.
\end{corollary}

%Conversely, every complex of this form is a 2-term silting complex.

\subsection*{2.3 Silted algebras}

\begin{definition}\cite[Definition 0.1]{BZ}
Let $Q$ be an acyclic quiver. We call an algebra $B$ \textbf{silted} of type $Q$ if there exists a 2-term silting complex $M$ over $KQ$ such that $B\cong \text{End}_{\text{K}^b(\mathrm{proj}KQ)}(M)$.
\end{definition}
Tilted algebras  are silted algebras, because (projective resolutions of)
tilting modules are 2-term silting complexes.

\begin{theorem}\label{s}\cite[Theorem 2.13]{BZ}
Let A be a connected  algebra. Then the following are equivalent:
\begin{itemize}
\item[($a$)] A is a silted algebra;

\item[($b$)] A is a tilted algebra or a strictly shod algebra.
\end{itemize}
\end{theorem}

Recall from \cite[page 2]{CL} that an algebra $A$ is called \textbf{shod} (for small homological dimension) provided for each indecomposable $A$-module $X$, either $\text{pd} X_A\leq 1$ or $\text{id} X_A\leq 1$. It is known that $\text{gl.dim}A\leq3$ \cite[Proposition 2.2]{C}. We call $A$ \textbf{strictly shod} if it is shod and $\text{gl.dim} A=3$. It is known that tilted algebras are shod \cite[Proposition 3.2]{C}.

\medskip
The following lemma will be useful.

\begin{lemma}\label{s}
 An algebra $A$ is silted of type $Q$ if and only if $\text{A}^{op}$ is silted of type $\text{Q}^{op}$.
\end{lemma}

\begin{proof}
%\noindent{\textbf{Proof}}
This is because $\RHom_{A}(?,A)[1]:\text{K}^b(\text{proj}A)\longrightarrow\text{K}^b(\text{proj}A^{op})$  is a triangle anti-equivalence and induces a bijection between the set of 2-term silting complexes over $KQ$ and that over $KQ^{op}$.
\end{proof}

\section{Examples of silted algebras of Dynkin type}

In this section, we will describe an algorithm to produce all basic 2-term silting complexes over the path algebra of a Dynkin quiver, and use this algorithm to compute some examples.

Let $Q$ be a Dynkin quiver and $A=KQ$. Let $\text{K}^{[-1.0]}(\text{proj}A)$ be the full subcategory of $\text{K}^b(\text{proj}A)$ consisting of complexes concentrated in degrees -1 and 0. We will call the full subquiver of the AR quiver of $\text{K}^b(\text{proj}A)$ whose vertices belong to $\text{K}^{[-1.0]}(\text{proj}A)$ the AR quiver of $\text{K}^{[-1.0]}(\text{proj}A)$. It is obtained from the AR quiver of  $\text{mod}A$  by properly gluing a copy of $Q$ from the right.

\vspace{27pt}
\subsection*{3.1 The algorithm}

\paragraph{\indent Let $Q$ be a Dynkin quiver, and $A=KQ$. Due to Corollary \ref{coro} we have the following algorithm to produce all basic 2-term silting complexes over $A$.}

\begin{algorithm}\label{al2}

We perform the following two steps for any subset $I$ of $Q_0$:

(1) Let $e=\bigoplus\limits_{i\in I}e(i)$  and  $A(I)=A/\langle e\rangle$.

(2) For each basic tilting $A(I)$-module  $M$ produced by Algorithm \ref{al1}, form $T=M\oplus P[1]$ where $P=eA$.

\end{algorithm}

\begin{observation}\label{Ob}
Let $T=M\oplus P[1]$ be a 2-term silting complex over $A$, where $P\in \mathrm{proj}A$ and  $M\in \mathrm{mod}A$. If $P=0$ or $M$ has no non-trivial projective direct summands, then $\mathrm{End}(T)$ is  a tilted algebra of type $Q$.

Indeed, if $P=0$, then $T=M$ is a tilting $A$-module; if $M$ has no no-trivial projective direct summands, then $\tau_A(T)$ belongs to $\mathrm{mod}A$, and  hence is a tilting $A$-module, so $\mathrm{End}(T)\cong \mathrm{End}(\tau_A(T))$ is a tilted algebra.
\end{observation}

 By Observation \ref{Ob}, we will divide silted algebras of type $Q$ into two classes:
\begin{itemize}
\item[(I)] tilted algebras of type $Q$,
\item[(II)]$\text{End}(T)$, where $T=M\oplus P[1]$ is a 2-term silting complex such that $P\neq 0$ and $M$ has a non-zero projective direct summand over $A$. In other words, $T$ has direct summands both on the left border and on the right border of the AR quiver of $\text{K}^{[-1.0]}(\text{proj}A)$.
\end{itemize}

We remark that (I) and (II) may have overlaps. We are mainly interested in the silted algebras which are not tilted of type $Q$,  especially the strictly shod algebras.

%Below in the computation of silted algebras
%we will use the following formulas for $T, T{'}\in \text{modA}$  and
%$P, P{'}\in \text{projA}$
%\begin{align*}
%\text{Hom}_{\text{K}^b(\text{proj}A)}(T^{'},T)&=\text{Hom}_A(T^{'},T)\\
%\text{Hom}_{\text{K}^b(\text{proj}A)}(P^{'}[1],P[1])&=\text{Hom}_{\text{K}^b(\text{proj}A)}(P^{'},P)
%=\text{Hom}_A(P^{'},P)\\
%\text{Hom}_{\text{K}^b(\text{proj}A)}(P^{'}[1],T)&=\text{Hom}_{\text{K}^b(\text{proj}A)}(P^{'},T[-1])\\
%&=\Ext^{-1}(P^{'},T)=0\\
%\text{Hom}_{\text{K}^b(\text{proj}A)}(T^{'},P[1])&=\Ext^1(T^{'},P)
%=\text{DHom}(P,\tau(T^{'}))
%\end{align*}

\subsection*{3.2 Examples of type $\mathbb{A}$}

\subsection*{3.2.1 Type $\mathbb{A}_1$}\label{1}
\begin{example}
Let $A$ be the path algebra of  the quiver
$\stackrel{1}\circ$. For a 2-term silting complex $M\oplus P[1]$, either $P=0$, or $M=0$.
So $\text{End}_{\text{K}^b(\text{proj}A)}(M\oplus P[1])$ is isomorphic to $A$.
\end{example}

\subsection*{3.2.2 Type  $\mathbb{A}_2$}\label{2}

%\paragraph{\indent There are 2 silted algebras of $\circ\longrightarrow \circ $, forming two families:}
%
%\begin{itemize}
%\item[(\romannumeral1)]tilted algebra of type $\mathbb{A}_2$£º
%$\circ\rightarrow\circ$
%
%\item[(\romannumeral2)]tilted algebra of type
%$\mathbb{A}_1\amalg\mathbb{A}_1$:
%$\circ ~~~ \circ$
%\end{itemize}

\begin{example}
Let $A$ be the path algebra of  the quiver $$\stackrel{1}\circ\longrightarrow \stackrel{2}\circ $$
Tilted algebras were already computed in Example \ref{A2}. Thus
below we apply Algorithm \ref{al2} to all non-empty subsets $I$ of $Q_0$.

The AR-quiver $\Gamma(\text{K}^{[-1.0]}(\text{proj}A))$ is
$$\xymatrix@C=15pt{&{\begin{smallmatrix}  1&1 \end{smallmatrix}}\ar[rd]\ar@{..}[rr]&&{\begin{smallmatrix}  0&1[1] \end{smallmatrix}}\ar[rd]\\
{\begin{smallmatrix}  0&1 \end{smallmatrix}}\ar@{..}[rr]\ar[ru]&&{\begin{smallmatrix}  1&0 \end{smallmatrix}}\ar@{..}[rr]\ar[ru]&&
{\begin{smallmatrix}  1&1[1] \end{smallmatrix}}&
}$$

($1$) $I=\{1\}$. $A(I)=A/\langle e\rangle$ is given by the quiver  $\stackrel{2}\circ $, which has only one tilting module $M={\begin{smallmatrix}  0&1 \end{smallmatrix}}$.
The corresponding silting complex is
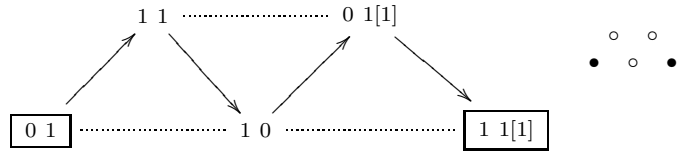
\newcommand{\scirc}{{\scriptstyle \circ}}
\newcommand{\sbullet}{{\scriptstyle\bullet}}
\[{\setlength{\unitlength}{0.3pt}
\begin{picture}(650,230)
\put(-20,200){
$\xymatrix@C=15pt{&{\begin{smallmatrix}  1&1 \end{smallmatrix}}\ar[rd]\ar@{..}[rr]&&{\begin{smallmatrix}  0&1[1] \end{smallmatrix}}\ar[rd]\\
\boxed{{\begin{smallmatrix}  0&1 \end{smallmatrix}}}\ar@{..}[rr]\ar[ru]&&{\begin{smallmatrix}  1&0 \end{smallmatrix}}\ar@{..}[rr]\ar[ru]&&
\boxed{{\begin{smallmatrix}  1&1[1] \end{smallmatrix}}}&
}$}
\put(700,170){$
\begin{array}{*{20}{c@{\hspace{3pt}}}}
&&\scirc&&\scirc\\[-5pt]
&\sbullet&&\scirc&&\sbullet\\[-5pt]
\end{array}$}
\end{picture}}\]
\noindent Its endomorphism algebra is given by the quiver
$\circ ~~~ \circ$.

($2$) $I=\{2\}$. $A(I)=A/\langle e\rangle$ is given by the quiver  $\stackrel{1}\circ $, which has only one tilting module $M={\begin{smallmatrix}  1&0 \end{smallmatrix}}$.
The corresponding silting complex is
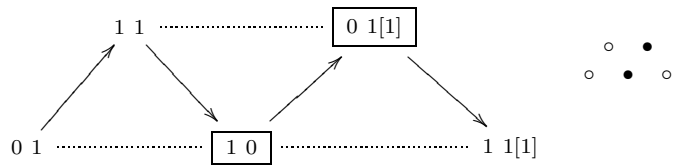
\[{\setlength{\unitlength}{0.3pt}
\begin{picture}(650,230)
\put(-20,200){
$\xymatrix@C=15pt{&{\begin{smallmatrix}  1&1 \end{smallmatrix}}\ar[rd]\ar@{..}[rr]&&\boxed{{\begin{smallmatrix}  0&1[1] \end{smallmatrix}}}\ar[rd]\\
{\begin{smallmatrix}  0&1 \end{smallmatrix}}\ar@{..}[rr]\ar[ru]&&\boxed{{\begin{smallmatrix}  1&0 \end{smallmatrix}}}\ar@{..}[rr]\ar[ru]&&
{\begin{smallmatrix}  1&1[1] \end{smallmatrix}}&
}$}
\put(700,170){$
\begin{array}{*{20}{c@{\hspace{3pt}}}}
&&\scirc&&\sbullet\\[-5pt]
&\scirc&&\sbullet&&\scirc\\[-5pt]
\end{array}$}
\end{picture}}\]
\noindent It is clear that $\tau T=A_A$.

($3$) $I=\{1,2\}$. Then $T=A_A[1]$, $\tau T=D({}_AA)$ and $\text{End}(T)\cong A$.
\smallskip
\paragraph{\indent To summarise, there are 2 silted algebras of $\circ\longrightarrow \circ $, forming two families:}

\begin{itemize}
\item[(\romannumeral1)]tilted algebra of type $\mathbb{A}_2$:
$\circ\rightarrow\circ$

\item[(\romannumeral2)]tilted algebra of type
$\mathbb{A}_1\amalg\mathbb{A}_1$:
$\circ ~~~ \circ$
\end{itemize}

More precisely, we have the following table:
\[\begin{tabular}{|c|c|c|}
 \hline
silted algebras & 2-term silting complexes & tilted type\\
\hline
\xymatrix{\circ \ar[r] & \circ} & $\begin{array}{*{20}{c@{\hspace{3pt}}}}
&&\sbullet&&\scirc\\[-5pt]
&\scirc&&\sbullet&&\scirc
\end{array} \hspace{5mm}   \begin{array}{*{20}{c@{\hspace{3pt}}}}
&&\sbullet&&\scirc\\[-5pt]
&\sbullet&&\scirc&&\scirc
\end{array}  \hspace{5mm}  \begin{array}{*{20}{c@{\hspace{3pt}}}}
&&\scirc&&\sbullet\\[-5pt]
&\scirc&&\sbullet&&\scirc
\end{array}  \hspace{5mm}  \begin{array}{*{20}{c@{\hspace{3pt}}}}
&&\scirc&&\sbullet\\[-5pt]
&\scirc&&\scirc&&\sbullet
\end{array} $ & $\mathbb{A}_2$ \\
\hline
\xymatrix{\circ & \circ} & $\begin{array}{*{20}{c@{\hspace{3pt}}}}
&&\scirc&&\scirc\\[-5pt]
&\sbullet&&\scirc&&\sbullet
\end{array}$ & $\mathbb{A}_1\amalg\mathbb{A}_1$ \\
\hline
\end{tabular}
\]

\end{example}

\vspace{-10pt}
\bigskip
\subsection*{3.2.3 Type  $\mathbb{A}_3$}\label{3}

\paragraph{\indent According to \cite[Theorem 1]{ONFR}, there are 14 basic 2-term silting complexes, 5 of which are tilting modules.}

 Up to isomorphism there are three quivers of type $\mathbb{A}_3$. Due to Lemma \ref{s}, we classify silted algebras for two of them.

\begin{example}\label{A3} Let $A$ be the path algebra of the quiver $$\stackrel{1}\circ\longrightarrow \stackrel{2}\circ \longrightarrow \stackrel{3}\circ$$

(I) The AR-quiver $\Gamma (\text{mod}A)$ of $\text{mod}A$ is of the form
$$
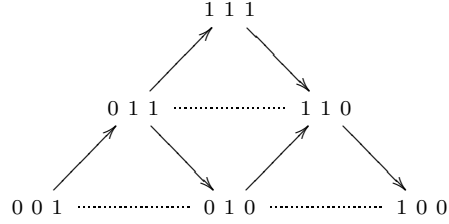
\xymatrix@C=5pt{&&{\begin{smallmatrix} 1&1&1 \end{smallmatrix}}\ar[rd]\\
&{\begin{smallmatrix} 0&1&1 \end{smallmatrix}}\ar@{..}[rr]\ar[ru]\ar[rd]&&
{\begin{smallmatrix} 1&1&0 \end{smallmatrix}}\ar[rd]\\
{\begin{smallmatrix} 0&0&1 \end{smallmatrix}}\ar@{..}[rr]\ar[ru]&&
{\begin{smallmatrix} 0&1&0 \end{smallmatrix}}\ar@{..}[rr]\ar[ru]&&
{\begin{smallmatrix} 1&0&0 \end{smallmatrix}}
}$$
 We apply Algorithm \ref{al1} to obtain the following table of tilted algebras

\newcommand{\scirc}{{\scriptstyle \circ}}
\newcommand{\sbullet}{{\scriptstyle\bullet}}
%\[
%\begin{array}{*{20}{c@{\hspace{3pt}}}}
%&&\sbullet&&\\[-5pt]
%&\scirc&&\sbullet\\[-5pt]
%\scirc&&\scirc&&\sbullet\\[-5pt]
%\end{array}
%~~~\text{}~~~
%\begin{array}{*{20}{c@{\hspace{3pt}}}}
%&&\sbullet&&\\[-5pt]
%&\scirc&&\sbullet\\[-5pt]
%\scirc&&\sbullet&&\scirc\\[-5pt]
%\end{array}
%~~~\text{}~~~
%\begin{array}{*{20}{c@{\hspace{3pt}}}}
%&&\sbullet&&\\[-5pt]
%&\scirc&&\scirc\\[-5pt]
%\sbullet&&\scirc&&\sbullet\\[-5pt]
%\end{array}
%~~~\text{}~~~
%\begin{array}{*{20}{c@{\hspace{3pt}}}}
%&&\sbullet&&\\[-5pt]
%&\sbullet&&\scirc\\[-5pt]
%\scirc&&\sbullet&&\scirc\\[-5pt]
%\end{array}
%~~~\text{}~~~
%\begin{array}{*{20}{c@{\hspace{3pt}}}}
%&&\sbullet&&\\[-5pt]
%&\sbullet&&\scirc\\[-5pt]
%\sbullet&&\scirc&&\scirc\\[-5pt]
%\end{array}
%\]

%\end{example}
%
%
%\begin{example}
%Let A be be the path algebra of the quiver $$\stackrel{1}\circ\longrightarrow \stackrel{2}\circ \longrightarrow \stackrel{3}\circ$$
%Tilted algebras were already computed in Example\ref{t A_3}.

\[\begin{tabular}{|c|c|}
 \hline
tilted algebras & tilting modules\\
\hline
\xymatrix{\circ \ar[r] & \circ \ar[r] & \circ} & $\begin{array}{*{20}{c@{\hspace{3pt}}}}
&&\sbullet&&\\[-5pt]
&\scirc&&\sbullet\\[-5pt]
\scirc&&\scirc&&\sbullet
\end{array} \hspace{5mm}
\begin{array}{*{20}{c@{\hspace{3pt}}}}
&&\sbullet&&\\[-5pt]
&\sbullet&&\scirc\\[-5pt]
\sbullet&&\scirc&&\scirc
\end{array}$\\
\hline
\xymatrix{\circ  & \circ \ar[l]\ar[r] & \circ} & $\begin{array}{*{20}{c@{\hspace{3pt}}}}
&&\sbullet&&\\[-5pt]
&\scirc&&\sbullet\\[-5pt]
\scirc&&\sbullet&&\scirc
\end{array}$\\
\hline
\xymatrix@C=8pt{\circ\ar[rr]\ar@/^/@{.}[rrrr]_{}
&&\circ\ar[rr]&&\circ} & $\begin{array}{*{20}{c@{\hspace{3pt}}}}
&&\sbullet&&\\[-5pt]
&\scirc&&\scirc\\[-5pt]
\sbullet&&\scirc&&\sbullet
\end{array}$\\
\hline
\xymatrix{\circ \ar[r] & \circ & \circ \ar[l]} &
$\begin{array}{*{20}{c@{\hspace{3pt}}}}
&&\sbullet&&\\[-5pt]
&\sbullet&&\scirc\\[-5pt]
\scirc&&\sbullet&&\scirc
\end{array}$\\
\hline
\end{tabular}
\]

(II) The AR-quiver $\Gamma(\text{K}^{[-1.0]}(\text{proj}A))$  is
$$
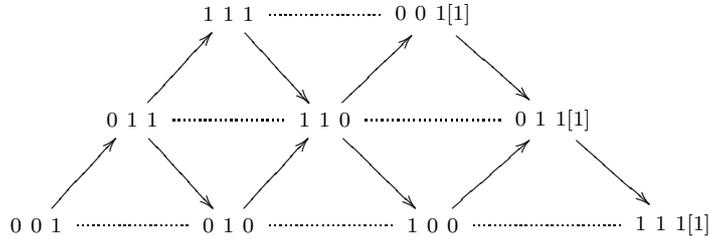
\xymatrix@C=5pt{&&{\begin{smallmatrix}  1&1&1 \end{smallmatrix}}\ar[rd]\ar@{..}[rr]&&{\begin{smallmatrix}  0&0&1[1] \end{smallmatrix}}\ar[rd]\\
&{\begin{smallmatrix}  0&1&1 \end{smallmatrix}}\ar@{..}[rr]\ar[ru]\ar[rd]&&
{\begin{smallmatrix}  1&1&0 \end{smallmatrix}}\ar[rd]\ar@{..}[rr]\ar[ru]&&
{\begin{smallmatrix}  0&1&1[1] \end{smallmatrix}}\ar[rd]&\\
{\begin{smallmatrix}  0&0&1 \end{smallmatrix}}\ar@{..}[rr]\ar[ru]&&
{\begin{smallmatrix}  0&1&0 \end{smallmatrix}}\ar@{..}[rr]\ar[ru]&&
{\begin{smallmatrix}  1&0&0 \end{smallmatrix}}\ar@{..}[rr]\ar[ru]&&
{\begin{smallmatrix}  1&1&1[1] \end{smallmatrix}}
}$$
Below we apply Algorithm \ref{al2} to all non-empty subsets $I$ of $Q_0$.

($1$) $I=\{1\}$. $A(I)=A/\langle e\rangle$ is given by the quiver $\stackrel{2}\circ\longrightarrow\stackrel{3}\circ$.
$A(I)$ has two basic tilting modules: $M_1={\begin{smallmatrix} 0&1&0 \end{smallmatrix}}\oplus{\begin{smallmatrix} 0&1&1 \end{smallmatrix}}$ and $M_2={\begin{smallmatrix} 0&0&1 \end{smallmatrix}}\oplus{\begin{smallmatrix} 0&1&1 \end{smallmatrix}}$.

\begin{itemize}
\item[(\romannumeral1)] For $M_1={\begin{smallmatrix} 0&1&0 \end{smallmatrix}}\oplus{\begin{smallmatrix} 0&1&1 \end{smallmatrix}}$, the corresponding silting complex is
%\newcommand{\scirc}{{\scriptstyle \circ}}
%\newcommand{\sbullet}{{\scriptstyle\bullet}}
%\[{\setlength{\unitlength}{0.5pt}
%\begin{picture}(650,230)
%\put(-50,200){
%$\xymatrix@C=15pt{&&{\begin{smallmatrix}  1&1&1 \end{smallmatrix}}\ar[rd]\ar@{..}[rr]&&{\begin{smallmatrix}  0&0&1[1] \end{smallmatrix}}\ar[rd]\\
%&\boxed{{\begin{smallmatrix}  0&1&1 \end{smallmatrix}}}\ar@{..}[rr]\ar[ru]\ar[rd]&&
%{\begin{smallmatrix}  1&1&0 \end{smallmatrix}}\ar[rd]\ar@{..}[rr]\ar[ru]&&
%{\begin{smallmatrix}  0&1&1[1] \end{smallmatrix}}\ar[rd]&\\
%{\begin{smallmatrix}  0&0&1 \end{smallmatrix}}\ar@{..}[rr]\ar[ru]&&
%\boxed{{\begin{smallmatrix}  0&1&0 \end{smallmatrix}}}\ar@{..}[rr]\ar[ru]&&
%{\begin{smallmatrix}  1&0&0 \end{smallmatrix}}\ar@{..}[rr]\ar[ru]&&
%\boxed{{\begin{smallmatrix}  1&1&1[1] \end{smallmatrix}}}
%}$}
%\put(550,170){$
%\begin{array}{*{20}{c@{\hspace{3pt}}}}
%&&\scirc&&\scirc\\[-5pt]
%&\sbullet&&\scirc&&\scirc\\[-5pt]
%\scirc&&\sbullet&&\scirc&&\sbullet\\[-5pt]
%\end{array}$}
%\end{picture}}\]
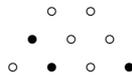
\vspace{-10pt}
\[\begin{array}{*{20}{c@{\hspace{3pt}}}}
&&\scirc&&\scirc\\[-5pt]
&\sbullet&&\scirc&&\scirc\\[-5pt]
\scirc&&\sbullet&&\scirc&&\sbullet\\[-5pt]
\end{array}
\]
Its endomorphism algebra is given by the quiver $\circ\longrightarrow\circ~~~\circ.$

\item[(\romannumeral2)] For $M_2={\begin{smallmatrix} 0&0&1 \end{smallmatrix}}\oplus{\begin{smallmatrix} 0&1&1 \end{smallmatrix}}$, the corresponding silting complex is
%\[{\setlength{\unitlength}{0.5pt}
%\begin{picture}(650,230)
%\put(-50,200){$
%$$\xymatrix@C=15pt{&&{\begin{smallmatrix}  1&1&1 \end{smallmatrix}}\ar[rd]\ar@{..}[rr]&&{\begin{smallmatrix}  0&0&1[1] \end{smallmatrix}}\ar[rd]\\
%&\boxed{{\begin{smallmatrix}  0&1&1 \end{smallmatrix}}}\ar@{..}[rr]\ar[ru]\ar[rd]&&
%{\begin{smallmatrix}  1&1&0 \end{smallmatrix}}\ar[rd]\ar@{..}[rr]\ar[ru]&&
%{\begin{smallmatrix}  0&1&1[1] \end{smallmatrix}}\ar[rd]&\\
%\boxed{{\begin{smallmatrix}  0&0&1 \end{smallmatrix}}}\ar@{..}[rr]\ar[ru]&&
%{\begin{smallmatrix}  0&1&0 \end{smallmatrix}}\ar@{..}[rr]\ar[ru]&&
%{\begin{smallmatrix}  1&0&0 \end{smallmatrix}}\ar@{..}[rr]\ar[ru]&&
%\boxed{{\begin{smallmatrix}  1&1&1[1] \end{smallmatrix}}}
%}$}
%\put(550,170){$
%\begin{array}{*{20}{c@{\hspace{3pt}}}}
%&&\scirc&&\scirc\\[-5pt]
%&\sbullet&&\scirc&&\scirc\\[-5pt]
%\sbullet&&\scirc&&\scirc&&\sbullet\\[-5pt]
%\end{array}$}
%\end{picture}}\]
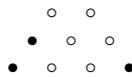
\vspace{-10pt}
\[\begin{array}{*{20}{c@{\hspace{3pt}}}}
&&\scirc&&\scirc\\[-5pt]
&\sbullet&&\scirc&&\scirc\\[-5pt]
\sbullet&&\scirc&&\scirc&&\sbullet\\[-5pt]
\end{array}
\]
Its endomorphism algebra is given by the quiver $\circ\longrightarrow\circ~~~\circ.$
\end{itemize}

($2$) $I=\{2\}$. $A(I)=A/\langle e\rangle$ is given by the quiver $\stackrel{1}\circ~~~\stackrel{3}\circ$.
$A(I)$ has only one basic tilting module: $M_1={\begin{smallmatrix} 1&0&0 \end{smallmatrix}}\oplus{\begin{smallmatrix} 0&0&1 \end{smallmatrix}}$. The corresponding silting complex is
%\newcommand{\scirc}{{\scriptstyle \circ}}
%\newcommand{\sbullet}{{\scriptstyle\bullet}}
%\[{\setlength{\unitlength}{0.5pt}
%\begin{picture}(650,230)
%\put(-50,200){
%$\xymatrix@C=15pt{&&{\begin{smallmatrix}  1&1&1 \end{smallmatrix}}\ar[rd]\ar@{..}[rr]&&{\begin{smallmatrix}  0&0&1[1] \end{smallmatrix}}\ar[rd]\\
%&{\begin{smallmatrix}  0&1&1 \end{smallmatrix}}\ar@{..}[rr]\ar[ru]\ar[rd]&&
%{\begin{smallmatrix}  1&1&0 \end{smallmatrix}}\ar[rd]\ar@{..}[rr]\ar[ru]&&
%\boxed{{\begin{smallmatrix}  0&1&1[1] \end{smallmatrix}}}\ar[rd]&\\
%\boxed{{\begin{smallmatrix}  0&0&1 \end{smallmatrix}}}\ar@{..}[rr]\ar[ru]&&
%{\begin{smallmatrix}  0&1&0 \end{smallmatrix}}\ar@{..}[rr]\ar[ru]&&
%\boxed{{\begin{smallmatrix}  1&0&0 \end{smallmatrix}}}\ar@{..}[rr]\ar[ru]&&
%{\begin{smallmatrix}  1&1&1[1] \end{smallmatrix}}
%}$}
%\put(550,170){$
%\begin{array}{*{20}{c@{\hspace{3pt}}}}
%&&\scirc&&\scirc\\[-5pt]
%&\scirc&&\scirc&&\sbullet\\[-5pt]
%\sbullet&&\scirc&&\sbullet&&\scirc\\[-5pt]
%\end{array}$}
%\end{picture}}\]
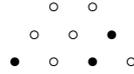
\vspace{-10pt}
\[\begin{array}{*{20}{c@{\hspace{3pt}}}}
&&\scirc&&\scirc\\[-5pt]
&\scirc&&\scirc&&\sbullet\\[-5pt]
\sbullet&&\scirc&&\sbullet&&\scirc\\[-5pt]
\end{array}
\]
Its endomorphism algebra is given by the quiver
$\circ\longrightarrow\circ~~~\circ.$

($3$) $I=\{1,2\}$. $A(I)=A/\langle e\rangle$ is given by the quiver $\stackrel{3}\circ$.
$A(I)$ has only one  basic tilting module: $M_1={\begin{smallmatrix} 0&0&1 \end{smallmatrix}}$. The corresponding silting complex is
%\[{\setlength{\unitlength}{0.5pt}
%\begin{picture}(650,230)
%\put(-50,200){
%$\xymatrix@C=15pt{&&{\begin{smallmatrix}  1&1&1 \end{smallmatrix}}\ar[rd]\ar@{..}[rr]&&{\begin{smallmatrix}  0&0&1[1] \end{smallmatrix}}\ar[rd]\\
%&{\begin{smallmatrix}  0&1&1 \end{smallmatrix}}\ar@{..}[rr]\ar[ru]\ar[rd]&&
%{\begin{smallmatrix}  1&1&0 \end{smallmatrix}}\ar[rd]\ar@{..}[rr]\ar[ru]&&
%\boxed{{\begin{smallmatrix}  0&1&1[1] \end{smallmatrix}}}\ar[rd]&\\
%\boxed{{\begin{smallmatrix}  0&0&1 \end{smallmatrix}}}\ar@{..}[rr]\ar[ru]&&
%{\begin{smallmatrix}  0&1&0 \end{smallmatrix}}\ar@{..}[rr]\ar[ru]&&
%{\begin{smallmatrix}  1&0&0 \end{smallmatrix}}\ar@{..}[rr]\ar[ru]&&
%\boxed{{\begin{smallmatrix}  1&1&1[1] \end{smallmatrix}}}
%}$}
%\put(550,170){$
%\begin{array}{*{20}{c@{\hspace{3pt}}}}
%&&\scirc&&\scirc\\[-5pt]
%&\scirc&&\scirc&&\sbullet\\[-5pt]
%\sbullet&&\scirc&&\scirc&&\sbullet\\[-5pt]
%\end{array}$}
%\end{picture}}\]
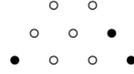
\vspace{-10pt}
\[\begin{array}{*{20}{c@{\hspace{3pt}}}}
&&\scirc&&\scirc\\[-5pt]
&\scirc&&\scirc&&\sbullet\\[-5pt]
\sbullet&&\scirc&&\scirc&&\sbullet\\[-5pt]
\end{array}
\]
Its endomorphism algebra is given by the quiver $\circ\longrightarrow\circ~~~\circ.$

($4$) If $3\in I$, then $\tau T$ is a tilting module. We list all such $T$ below:
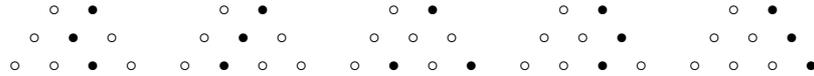
\vspace{-10pt}
\[\begin{array}{*{20}{c@{\hspace{3pt}}}}
&&\scirc&&\sbullet\\[-5pt]
&\scirc&&\sbullet&&\scirc\\[-5pt]
\scirc&&\scirc&&\sbullet&&\scirc\\[-5pt]
\end{array}~~~\text{}~~~
\begin{array}{*{20}{c@{\hspace{3pt}}}}
&&\scirc&&\sbullet\\[-5pt]
&\scirc&&\sbullet&&\scirc\\[-5pt]
\scirc&&\sbullet&&\scirc&&\scirc\\[-5pt]
\end{array}~~~\text{}~~~
\begin{array}{*{20}{c@{\hspace{3pt}}}}
&&\scirc&&\sbullet\\[-5pt]
&\scirc&&\scirc&&\scirc\\[-5pt]
\scirc&&\sbullet&&\scirc&&\sbullet\\[-5pt]
\end{array}~~~\text{}~~~
\begin{array}{*{20}{c@{\hspace{3pt}}}}
&&\scirc&&\sbullet\\[-5pt]
&\scirc&&\scirc&&\sbullet\\[-5pt]
\scirc&&\scirc&&\sbullet&&\scirc\\[-5pt]
\end{array}~~~\text{}~~~
\begin{array}{*{20}{c@{\hspace{3pt}}}}
&&\scirc&&\sbullet\\[-5pt]
&\scirc&&\scirc&&\sbullet\\[-5pt]
\scirc&&\scirc&&\scirc&&\sbullet\\[-5pt]
\end{array}
\]
\end{example}

\paragraph{\indent Note that the four silted algebras in the above (1), (2) and (3) are isomorphic, so there are  5 silted algebras  of type $\circ\longrightarrow \circ \longrightarrow \circ$, forming two families:}

\begin{itemize}
\item[(\romannumeral1)]tilted algebras of type $\mathbb{A}_3$:

$\circ\longrightarrow \circ \longrightarrow \circ$
~~~\text{,}~~~
$\circ\longrightarrow \circ \longleftarrow \circ$
~~~\text{,}~~~
$\circ\longleftarrow \circ \longrightarrow \circ$
~~~\text{,}~~~
$\xymatrix@C=5pt{\circ\ar[rr]\ar@/^/@{.}[rrrr]_{}
&&\circ\ar[rr]&&\circ}$

\item[(\romannumeral2)]tilted algebra of type
 $\mathbb{A}_2\amalg\mathbb{A}_1$:
 $\circ\longrightarrow \circ ~~~ \circ$
\end{itemize}

\begin{example}
Let $A$ be the path algebra of the quiver  $$\xymatrix@C=5pt{Q=\stackrel{1}\circ\longrightarrow
\stackrel{3}\circ \longleftarrow \stackrel{2}\circ}$$

 The AR-quiver of $\text{K}^{[-1.0]}(\text{proj}A)$  is
$$
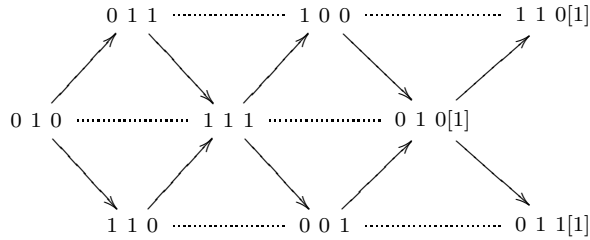
\xymatrix@C=5pt{&{\begin{smallmatrix}  0&1&1 \end{smallmatrix}}\ar@{..}[rr]\ar[rd]&&
{\begin{smallmatrix}  1&0&0 \end{smallmatrix}}\ar@{..}[rr]\ar[rd]&&
{\begin{smallmatrix}  1&1&0[1] \end{smallmatrix}}\\
{\begin{smallmatrix}  0&1&0 \end{smallmatrix}}\ar@{..}[rr]\ar[ru]\ar[rd]&&
{\begin{smallmatrix}  1&1&1 \end{smallmatrix}}\ar@{..}[rr]\ar[ru]\ar[rd]&&
{\begin{smallmatrix}  0&1&0[1] \end{smallmatrix}}\ar[ru]\ar[rd]\\
&{\begin{smallmatrix}  1&1&0 \end{smallmatrix}}\ar@{..}[rr]\ar[ru]&&
{\begin{smallmatrix}  0&0&1 \end{smallmatrix}}\ar@{..}[rr]\ar[ru]&&
{\begin{smallmatrix}  0&1&1[1] \end{smallmatrix}}
}$$

(I) We first apply Algorithm \ref{al1} to produce all tilting modules and compute the corresponding tilted algebras. These are the tilted algebras of type $Q$.
\[\begin{tabular}{|c|c|c|}
 \hline
No. & tilted algebras & tilting modules\\
\hline
(1)& \xymatrix{\circ\ar[r]&\circ\ar[r]&\circ} &
\newcommand{\scirc}{{\scriptstyle \circ}}
\newcommand{\sbullet}{{\scriptstyle\bullet}}
$\begin{array}{*{20}{c@{\hspace{3pt}}}}
&\scirc&&\sbullet\\[-5pt]
\scirc&&\sbullet\\[-5pt]
&\sbullet&&\scirc
\end{array}\hspace{5mm}
\begin{array}{*{20}{c@{\hspace{3pt}}}}
&\sbullet&&\scirc\\[-5pt]
\scirc&&\sbullet\\[-5pt]
&\scirc&&\sbullet\\[-5pt]
\end{array}$\\
\hline
(2)& \xymatrix{\circ&\circ\ar[r]\ar[l]&\circ} &
\newcommand{\scirc}{{\scriptstyle \circ}}
\newcommand{\sbullet}{{\scriptstyle\bullet}}
$\begin{array}{*{20}{c@{\hspace{3pt}}}}
&\sbullet&&\scirc\\[-5pt]
\scirc&&\sbullet\\[-5pt]
&\sbullet&&\scirc
\end{array}$\\
\hline
(3)& \xymatrix{\circ\ar[r]&\circ&\circ\ar[l]} &
\newcommand{\scirc}{{\scriptstyle \circ}}
\newcommand{\sbullet}{{\scriptstyle\bullet}}
$\begin{array}{*{20}{c@{\hspace{3pt}}}}
&\sbullet&&\scirc\\[-5pt]
\sbullet&&\scirc\\[-5pt]
&\sbullet&&\scirc
\end{array} \hspace{5mm}   \begin{array}{*{20}{c@{\hspace{3pt}}}}
&\scirc&&\sbullet\\[-5pt]
\scirc&&\sbullet\\[-5pt]
&\scirc&&\sbullet
\end{array}$\\
\hline
\end{tabular}
\]

(II) We apply Algorithm \ref{al2} to all non-empty subsets $I$ of $Q_0$. Due to Observation \ref{Ob}, we only list below the endomorphism algebras   $\text{End}_{\text{K}^b(\text{proj}A)}(T)$, where $T=M\oplus P[1]$ is a 2-term silting complex with $M\neq 0$ and $P\neq 0$ (\emph{i.e.} T has direct summands both on the left border and the right border of the AR quiver)
\[\begin{tabular}{|c|c|c|c|}
 \hline
No.& silted algebras & 2-term silting complexes &  tilted type\\
%\hline
%\xymatrix{\circ \ar[r] & \circ\ar[r] & \circ} &
%\newcommand{\scirc}{{\scriptstyle \circ}}
%\newcommand{\sbullet}{{\scriptstyle\bullet}} $\begin{array}{*{20}{c@{\hspace{3pt}}}}
%&\scirc&&\scirc&&\sbullet\\[-5pt]
%\scirc&&\scirc&&\sbullet\\[-5pt]
%&\scirc&&\sbullet&&\scirc
%\end{array}  \hspace{5mm} \begin{array}{*{20}{c@{\hspace{3pt}}}}
%&\scirc&&\sbullet&&\scirc\\[-5pt]
%\scirc&&\scirc&&\sbullet\\[-5pt]
%&\scirc&&\scirc&&\sbullet
%\end{array} $ & $\mathbb{A}_3$  \\
%\hline
%\xymatrix{\circ  & \circ\ar[r]\ar[l] & \circ} &
%\newcommand{\scirc}{{\scriptstyle \circ}}
%\newcommand{\sbullet}{{\scriptstyle\bullet}}
%$\begin{array}{*{20}{c@{\hspace{3pt}}}}
%&\scirc&&\sbullet&&\scirc\\[-5pt]
%\scirc&&\scirc&&\sbullet\\[-5pt]
%&\scirc&&\sbullet&&\scirc
%\end{array}$& $\mathbb{A}_3$\\
%\hline
%\xymatrix{\circ\ar[r]  & \circ & \circ\ar[l]} &
%\newcommand{\scirc}{{\scriptstyle \circ}}
%\newcommand{\sbullet}{{\scriptstyle\bullet}}
%$\begin{array}{*{20}{c@{\hspace{3pt}}}}
%&\scirc&&\scirc&&\sbullet\\[-5pt]
%\scirc&&\scirc&&\sbullet\\[-5pt]
%&\scirc&&\scirc&&\sbullet
%\end{array}$& $\mathbb{A}_3$\\
\hline
(4)&\xymatrix{\circ\ar[r]  & \circ ~~~ \circ} &
\newcommand{\scirc}{{\scriptstyle \circ}}
\newcommand{\sbullet}{{\scriptstyle\bullet}}
$\begin{array}{*{20}{c@{\hspace{3pt}}}}
&\scirc&&\scirc&&\scirc\\[-5pt]
\sbullet&&\scirc&&\scirc\\[-5pt]
&\sbullet&&\scirc&&\sbullet
\end{array}\hspace{5mm} \begin{array}{*{20}{c@{\hspace{3pt}}}}
&\sbullet&&\scirc&&\sbullet\\[-5pt]
\sbullet&&\scirc&&\scirc\\[-5pt]
&\scirc&&\scirc&&\scirc
\end{array}$&$\mathbb{A}_2
\amalg\mathbb{A}_1$\\
\hline
(5)&\xymatrix{\circ~~~ \circ ~~~ \circ} &
\newcommand{\scirc}{{\scriptstyle \circ}}
\newcommand{\sbullet}{{\scriptstyle\bullet}}
$\begin{array}{*{20}{c@{\hspace{3pt}}}}
&\scirc&&\scirc&&\sbullet\\[-5pt]
\sbullet&&\scirc&&\scirc\\[-5pt]
&\scirc&&\scirc&&\sbullet
\end{array}$&$\mathbb{A}_1
\amalg\mathbb{A}_1\amalg\mathbb{A}_1$\\
\hline
(6)&\xymatrix{\circ\ar[r]\ar@/^/@{.}[rr]_{}
&\circ\ar[r]&\circ}&
\newcommand{\scirc}{{\scriptstyle \circ}}
\newcommand{\sbullet}{{\scriptstyle\bullet}}
$\begin{array}{*{20}{c@{\hspace{3pt}}}}
&\scirc&&\sbullet&&\scirc\\[-5pt]
\scirc&&\scirc&&\scirc\\[-5pt]
&\sbullet&&\scirc&&\sbullet
\end{array}\hspace{5mm}\begin{array}{*{20}{c@{\hspace{3pt}}}}
&\sbullet&&\scirc&&\sbullet\\[-5pt]
\scirc&&\scirc&&\scirc\\[-5pt]
&\scirc&&\sbullet&&\scirc
\end{array} $& $\mathbb{A}_3$\\
\hline
\end{tabular}
\]

To summarise, there are 6 silted algebras of type $Q$, forming 3 families:
\begin{itemize}
\item[(\romannumeral1)]tilted algebra of type $\mathbb{A}_3$:
$(1)-(3), (6)$;

\item[(\romannumeral2)]tilted algebra of type
$\mathbb{A}_2\amalg\mathbb{A}_1$:
$(4)$;

\item[(\romannumeral3)]tilted algebra of type
$\mathbb{A}_1\amalg\mathbb{A}_1\amalg\mathbb{A}_1$:
$(5)$.
\end{itemize}

\end{example}

\subsection*{3.2.4 Type  $\mathbb{A}_4$}\label{4}

\paragraph{\indent According to \cite[Theorem 1]{ONFR}, there are 42 basic 2-term silting complexes, 14 of which are tilting modules. Up to isomorphism there are four quivers of type $\mathbb{A}_4$. Due to Lemma \ref{s}, we classify silted algebras for three of them.}

\begin{example}\label{A4}
Let $A$ be the path algebra of the quiver %$$Q= \stackrel{1}\circ\longrightarrow \stackrel{2}\circ \longrightarrow \stackrel{3}\circ\longrightarrow \stackrel{4}\circ$$
$$\xymatrix@C=5pt{Q= \stackrel{1}\circ\ar[rr]&&\stackrel{2}\circ\ar[rr]&&
\stackrel{3}\circ\ar[rr]&&\stackrel{4}\circ}$$
The AR-quiver of $\text{K}^{[-1.0]}(\text{proj}A)$  is
$$
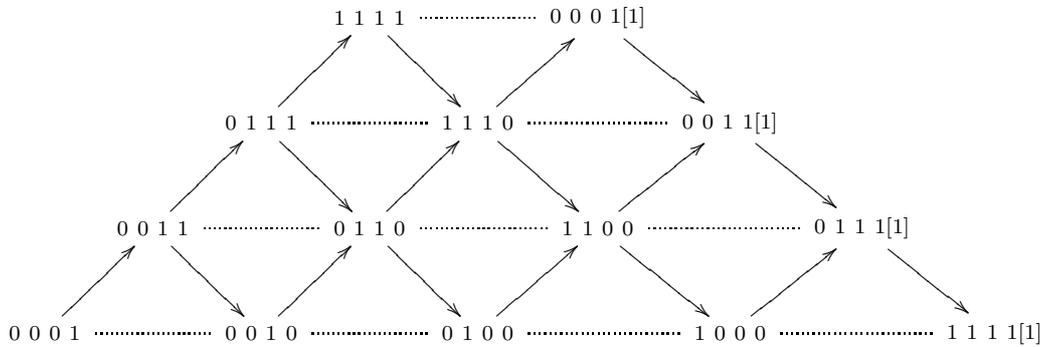
\xymatrix@C=2pt{&&&{\begin{smallmatrix}  1&1&1&1 \end{smallmatrix}}\ar[rd]\ar@{..}[rr] &&{\begin{smallmatrix}  0&0&0&1[1] \end{smallmatrix}}\ar[rd]\\
&& {\begin{smallmatrix}  0&1&1&1 \end{smallmatrix}}\ar@{..}[rr]\ar[ru]\ar[rd]&&
{\begin{smallmatrix}  1&1&1&0 \end{smallmatrix}}\ar[rd]\ar[ru]\ar@{..}[rr]&&{\begin{smallmatrix}  0&0&1&1[1] \end{smallmatrix}}\ar[rd]& \\
&{\begin{smallmatrix}  0&0&1&1 \end{smallmatrix}}\ar@{..}[rr]\ar[ru]\ar[rd]&&
{\begin{smallmatrix}  0&1&1&0 \end{smallmatrix}}\ar@{..}[rr]\ar[ru]\ar[rd]&&
{\begin{smallmatrix}  1&1&0&0 \end{smallmatrix}}\ar[rd]\ar@{..}[rr]\ar[ru]&&{\begin{smallmatrix}  0&1&1&1[1] \end{smallmatrix}}\ar[rd]\\
{\begin{smallmatrix}  0&0&0&1 \end{smallmatrix}}\ar@{..}[rr]\ar[ru]&&
{\begin{smallmatrix}  0&0&1&0 \end{smallmatrix}}\ar@{..}[rr]\ar[ru]&&
{\begin{smallmatrix}  0&1&0&0 \end{smallmatrix}}\ar@{..}[rr]\ar[ru]&&
{\begin{smallmatrix}  1&0&0&0 \end{smallmatrix}}\ar@{..}[rr]\ar[ru]&&{\begin{smallmatrix}  1&1&1&1[1] \end{smallmatrix}}
}$$

%(I) The  AR-quiver $\Gamma (\text{modA})$ of $\text{mod}A$ is of the form
%$$\xymatrix@C=15pt{&&&{\begin{smallmatrix}  1&1&1&1 \end{smallmatrix}}\ar[rd]\\
%&&{\begin{smallmatrix}  0&1&1&1 \end{smallmatrix}}\ar@{..}[rr]\ar[ru]\ar[rd]&&
%{\begin{smallmatrix}  1&1&1&0 \end{smallmatrix}}\ar[rd]&&\\
%&{\begin{smallmatrix}  0&0&1&1 \end{smallmatrix}}\ar@{..}[rr]\ar[ru]\ar[rd]&&
%{\begin{smallmatrix}  0&1&1&0 \end{smallmatrix}}\ar@{..}[rr]\ar[ru]\ar[rd]&&
%{\begin{smallmatrix}  1&1&0&0 \end{smallmatrix}}\ar[rd]&\\
%{\begin{smallmatrix}  0&0&0&1 \end{smallmatrix}}\ar@{..}[rr]\ar[ru]&&
%{\begin{smallmatrix}  0&0&1&0 \end{smallmatrix}}\ar@{..}[rr]\ar[ru]&&
%{\begin{smallmatrix}  0&1&0&0 \end{smallmatrix}}\ar@{..}[rr]\ar[ru]&&
%{\begin{smallmatrix}  1&0&0&0 \end{smallmatrix}}
%}$$

(I) Tilted algebras of type $Q$ are

\[\begin{tabular}{|c|c|c|}
 \hline
No. & tilted algebras & tilting modules\\
\hline
(1)&
\xymatrix{\circ \ar[r] & \circ\ar[r] & \circ\ar[r] & \circ} &\newcommand{\scirc}{{\scriptstyle \circ}}
\newcommand{\sbullet}{{\scriptstyle\bullet}} $\begin{array}{*{10}{c@{\hspace{1pt}}}}
&&&\sbullet\\[-9pt]
&&\sbullet&&\scirc\\[-9pt]
&\sbullet&&\scirc&&\scirc\\[-9pt]
\sbullet&&\scirc&&\scirc&&\scirc
\end{array} \hspace{5mm} \begin{array}{*{10}{c@{\hspace{1pt}}}}
&&&\sbullet\\[-9pt]
&&\scirc&&\sbullet\\[-9pt]
&\scirc&&\scirc&&\sbullet\\[-9pt]
\scirc&&\scirc&&\scirc&&\sbullet
\end{array}$\\
\hline
(2)&
\xymatrix{\circ\ar[r]&\circ\ar[r]&
\circ&\circ\ar[l]} &
\newcommand{\scirc}{{\scriptstyle \circ}}
\newcommand{\sbullet}{{\scriptstyle\bullet}} $\begin{array}{*{10}{c@{\hspace{1pt}}}}
&&&\sbullet\\[-9pt]
&&\sbullet&&\scirc\\[-9pt]
&\scirc&&\sbullet&&\scirc\\[-9pt]
\scirc&&\scirc&&\sbullet&&\scirc
\end{array}\hspace{5mm} \begin{array}{*{10}{c@{\hspace{1pt}}}}
&&&\sbullet\\[-9pt]
&&\sbullet&&\scirc\\[-9pt]
&\sbullet&&\scirc&&\scirc\\[-9pt]
\scirc&&\sbullet&&\scirc&&\scirc
\end{array}$\\
\hline
(3)&
\xymatrix{\circ\ar[r]&\circ&\circ
\ar[l]\ar[r]&\circ} &
\newcommand{\scirc}{{\scriptstyle \circ}}
\newcommand{\sbullet}{{\scriptstyle\bullet}}
$\begin{array}{*{10}{c@{\hspace{1pt}}}}
&&&\sbullet\\[-9pt]
&&\scirc&&\sbullet\\[-9pt]
&\scirc&&\sbullet&&\scirc\\[-9pt]
\scirc&&\scirc&&\sbullet&&\scirc
\end{array} \hspace{5mm} \begin{array}{*{10}{c@{\hspace{1pt}}}}
&&&\sbullet\\[-9pt]
&&\sbullet&&\scirc\\[-9pt]
&\scirc&&\sbullet&&\scirc\\[-9pt]
\scirc&&\sbullet&&\scirc&&\scirc
\end{array}$\\
\hline
(4)&
\xymatrix{\circ&\circ\ar[l]\ar[r]&
\circ\ar[r]&\circ} &\newcommand{\scirc}{{\scriptstyle \circ}}
\newcommand{\sbullet}{{\scriptstyle\bullet}} $\begin{array}{*{10}{c@{\hspace{1pt}}}}
&&&\sbullet\\[-9pt]
&&\scirc&&\sbullet\\[-9pt]
&\scirc&&\scirc&&\sbullet\\[-9pt]
\scirc&&\scirc&&\sbullet&&\scirc
\end{array} \hspace{5mm} \begin{array}{*{10}{c@{\hspace{1pt}}}}
&&&\sbullet\\[-9pt]
&&\scirc&&\sbullet\\[-9pt]
&\scirc&&\sbullet&&\scirc\\[-9pt]
\scirc&&\sbullet&&\scirc&&\scirc
\end{array}$\\
\hline
(5)&
\xymatrix{\circ\ar[r]\ar@/^/@{.}[rr]_{}&
\circ\ar[r]&\circ\ar[r]&\circ} &
\newcommand{\scirc}{{\scriptstyle \circ}}
\newcommand{\sbullet}{{\scriptstyle\bullet}}
$\begin{array}{*{10}{c@{\hspace{1pt}}}}
&&&\sbullet\\[-9pt]
&&\scirc&&\scirc\\[-9pt]
&\sbullet&&\scirc&&\scirc\\[-9pt]
\sbullet&&\scirc&&\scirc&&\sbullet
\end{array}$\\
\hline
(6)&
\xymatrix{\circ\ar[r]&
\circ\ar[r]\ar@/^/@{.}[rr]_{}&\circ\ar[r]&\circ}&
\newcommand{\scirc}{{\scriptstyle \circ}}
\newcommand{\sbullet}{{\scriptstyle\bullet}}
$\begin{array}{*{10}{c@{\hspace{1pt}}}}
&&&\sbullet\\[-9pt]
&&\scirc&&\scirc\\[-9pt]
&\scirc&&\scirc&&\sbullet\\[-9pt]
\sbullet&&\scirc&&\scirc&&\sbullet
\end{array}$\\
\hline
(7)&
\xymatrix{\circ\ar[r]&\circ&\circ\ar[l]
&\circ\ar[l]\ar@/_/@{.}[ll]_{}}&
\newcommand{\scirc}{{\scriptstyle \circ}}
\newcommand{\sbullet}{{\scriptstyle\bullet}}
$\begin{array}{*{10}{c@{\hspace{1pt}}}}
&&&\sbullet\\[-9pt]
&&\scirc&&\scirc\\[-9pt]
&\sbullet&&\scirc&&\scirc\\[-9pt]
\scirc&&\sbullet&&\scirc&&\sbullet
\end{array}$\\
\hline
(8)&
\xymatrix{\circ&\circ\ar[l]&\circ
\ar[l]\ar[r]\ar@/_/@{.}[ll]_{}&\circ}&
\newcommand{\scirc}{{\scriptstyle \circ}}
\newcommand{\sbullet}{{\scriptstyle\bullet}}
$\begin{array}{*{10}{c@{\hspace{1pt}}}}
&&&\sbullet\\[-9pt]
&&\scirc&&\scirc\\[-9pt]
&\scirc&&\scirc&&\sbullet\\[-9pt]
\sbullet&&\scirc&&\sbullet&&\scirc
\end{array}$\\
\hline
(9)&
\makecell[c]{
\xymatrix{\circ\ar[r]&\circ\ar[r]&\circ\\
&\circ\ar[u]\ar@/_/@{.}[ru]_{}
}}& \makecell[c]{
\newcommand{\scirc}{{\scriptstyle \circ}}
\newcommand{\sbullet}{{\scriptstyle\bullet}}
$\begin{array}{*{10}{c@{\hspace{1pt}}}}
&&&\sbullet\\[-9pt]
&&\sbullet&&\scirc\\[-9pt]
&\scirc&&\scirc&&\scirc\\[-9pt]
\sbullet&&\scirc&&\sbullet&&\scirc
\end{array}$}
\\
\hline
(10)&
\makecell[c]{
\xymatrix{\circ\ar[r]\ar@/_/@{.}[rd]_{}
&\circ\ar[d]\ar[r]&
\circ\\
&\circ
}}& \makecell[c]{
\newcommand{\scirc}{{\scriptstyle \circ}}
\newcommand{\sbullet}{{\scriptstyle\bullet}}
$\begin{array}{*{10}{c@{\hspace{1pt}}}}
&&&\sbullet\\[-9pt]
&&\scirc&&\sbullet\\[-9pt]
&\scirc&&\scirc&&\scirc\\[-9pt]
\scirc&&\sbullet&&\scirc&&\sbullet
\end{array}$}
\\
\hline
\end{tabular}
\]

\bigskip
(II) Silted algebras of the form $\text{End}_{\text{K}^b(\text{proj}A)}(M\oplus P[1])$ ($P\neq 0,~M\neq 0$) are

\newcommand{\scirc}{{\scriptstyle \circ}}
\newcommand{\sbullet}{{\scriptstyle\bullet}}

\[\begin{tabular}{|c|c|c|c|}
 \hline
No. &
silted algebras & 2-term silting complexes & tilted type\\
\hline
(11)&
\xymatrix{\circ\ar[r]&\circ\ar[r]
&\circ~~~\circ} & $\begin{array}{*{10}{c@{\hspace{1pt}}}}
&&&\scirc&&\scirc\\[-10pt]
&&\sbullet&&\scirc&&\scirc\\[-10pt]
&\scirc&&\sbullet&&\scirc&&\scirc\\[-10pt]
\scirc&&\scirc&&\sbullet&&\scirc&&\sbullet
\end{array}  \hspace{3mm} \begin{array}{*{10}{c@{\hspace{1pt}}}}
&&&\scirc&&\scirc\\[-10pt]
&&\sbullet&&\scirc&&\scirc\\[-10pt]
&\sbullet&&\scirc&&\scirc&&\scirc\\[-10pt]
\sbullet&&\scirc&&\scirc&&\scirc&&\sbullet
\end{array}   \hspace{3mm}  \begin{array}{*{10}{c@{\hspace{1pt}}}}
&&&\scirc&&\scirc\\[-10pt]
&&\scirc&&\scirc&&\sbullet\\[-10pt]
&\scirc&&\scirc&&\sbullet&&\scirc\\[-10pt]
\sbullet&&\scirc&&\sbullet&&\scirc&&\scirc
\end{array}   \hspace{3mm}  \begin{array}{*{10}{c@{\hspace{1pt}}}}
&&&\scirc&&\scirc\\[-10pt]
&&\scirc&&\scirc&&\sbullet\\[-10pt]
&\scirc&&\scirc&&\scirc&&\sbullet\\[-10pt]
\sbullet&&\scirc&&\scirc&&\scirc&&\sbullet
\end{array} $ & $\mathbb{A}_3\amalg\mathbb{A}_1$ \\
\hline
(12)&
\xymatrix{\circ&\circ\ar[r]\ar[l]
&\circ~~~\circ} & $
\begin{array}{*{10}{c@{\hspace{1pt}}}}
&&&\scirc&&\scirc\\[-10pt]
&&\sbullet&&\scirc&&\scirc\\[-10pt]
&\scirc&&\sbullet&&\scirc&&\scirc\\[-10pt]
\scirc&&\sbullet&&\scirc&&\scirc&&\sbullet
\end{array}   \hspace{3mm}  \begin{array}{*{10}{c@{\hspace{1pt}}}}
&&&\scirc&&\scirc\\[-10pt]
&&\scirc&&\scirc&&\sbullet\\[-10pt]
&\scirc&&\scirc&&\scirc&&\sbullet\\[-10pt]
\sbullet&&\scirc&&\scirc&&\sbullet&&\scirc
\end{array}$ & $\mathbb{A}_3\amalg\mathbb{A}_1$ \\
\hline
(13)&
\xymatrix{\circ\ar[r]&\circ
&\circ\ar[l]~~~\circ} & $
\begin{array}{*{10}{c@{\hspace{1pt}}}}
&&&\scirc&&\scirc\\[-10pt]
&&\sbullet&&\scirc&&\scirc\\[-10pt]
&\sbullet&&\scirc&&\scirc&&\scirc\\[-10pt]
\scirc&&\sbullet&&\scirc&&\scirc&&\sbullet
\end{array} \hspace{3mm}  \begin{array}{*{10}{c@{\hspace{1pt}}}}
&&&\scirc&&\scirc\\[-10pt]
&&\scirc&&\scirc&&\sbullet\\[-10pt]
&\scirc&&\scirc&&\sbullet&&\scirc\\[-10pt]
\sbullet&&\scirc&&\scirc&&\sbullet&&\scirc
\end{array}$ & $\mathbb{A}_3\amalg\mathbb{A}_1$ \\
\hline
(14)&
\xymatrix{\circ\ar[r]&\circ\ar[r]~~~
\circ\ar[r]&\circ} & $
\begin{array}{*{10}{c@{\hspace{1pt}}}}
&&&\scirc&&\scirc\\[-10pt]
&&\scirc&&\scirc&&\scirc\\[-10pt]
&\sbullet&&\scirc&&\scirc&&\sbullet\\[-10pt]
\scirc&&\sbullet&&\scirc&&\sbullet&&\scirc
\end{array}  \hspace{3mm} \begin{array}{*{10}{c@{\hspace{1pt}}}}
&&&\scirc&&\scirc\\[-10pt]
&&\scirc&&\scirc&&\scirc\\[-10pt]
&\sbullet&&\scirc&&\scirc&&\sbullet\\[-10pt]
\sbullet&&\scirc&&\scirc&&\sbullet&&\scirc
\end{array}  \hspace{3mm} \begin{array}{*{10}{c@{\hspace{1pt}}}}
&&&\scirc&&\scirc\\[-10pt]
&&\scirc&&\scirc&&\scirc\\[-10pt]
&\sbullet&&\scirc&&\scirc&&\sbullet\\[-10pt]
\scirc&&\sbullet&&\scirc&&\scirc&&\sbullet
\end{array}  \hspace{3mm}   \begin{array}{*{10}{c@{\hspace{1pt}}}}
&&&\scirc&&\scirc\\[-10pt]
&&\scirc&&\scirc&&\scirc\\[-10pt]
&\sbullet&&\scirc&&\scirc&&\sbullet\\[-10pt]
\sbullet&&\scirc&&\scirc&&\scirc&&\sbullet
\end{array} $ & $\mathbb{A}_2\amalg\mathbb{A}_2$ \\
\hline
(15)&

\xymatrix{\circ\ar[r]\ar@/^/@{.}[rr]_{}&\circ\ar[r]&
\circ~~~\circ} & $
\begin{array}{*{10}{c@{\hspace{1pt}}}}
&&&\scirc&&\scirc\\[-10pt]
&&\sbullet&&\scirc&&\scirc\\[-10pt]
&\scirc&&\scirc&&\scirc&&\scirc\\[-10pt]
\sbullet&&\scirc&&\sbullet&&\scirc&&\sbullet
\end{array}   \hspace{3mm} \begin{array}{*{10}{c@{\hspace{1pt}}}}
&&&\scirc&&\scirc\\[-10pt]
&&\scirc&&\scirc&&\sbullet\\[-10pt]
&\scirc&&\scirc&&\scirc&&\scirc\\[-10pt]
\sbullet&&\scirc&&\sbullet&&\scirc&&\sbullet
\end{array}$ & $\mathbb{A}_3\amalg\mathbb{A}_1$ \\
\hline
%(15)&
%
%\xymatrix{\circ\ar[r]\ar@/^/@{.}[rr]_{}&\circ\ar[r]&
%\circ~~~\circ} & $
%\begin{array}{*{10}{c@{\hspace{1pt}}}}
%&&&\scirc&&\scirc\\[-10pt]
%&&\sbullet&&\scirc&&\scirc\\[-10pt]
%&\scirc&&\scirc&&\scirc&&\scirc\\[-10pt]
%\sbullet&&\scirc&&\sbullet&&\scirc&&\sbullet
%\end{array}   \hspace{3mm} \begin{array}{*{10}{c@{\hspace{1pt}}}}
%&&&\scirc&&\scirc\\[-10pt]
%&&\scirc&&\scirc&&\sbullet\\[-10pt]
%&\scirc&&\scirc&&\scirc&&\scirc\\[-10pt]
%\sbullet&&\scirc&&\sbullet&&\scirc&&\sbullet
%\end{array}$ & $\mathbb{A}_3\amalg\mathbb{A}_1$ \\
%\hline
\end{tabular}
\]

To summarise, there are 15 silted algebras  of type $Q$, forming 3 families:
\begin{itemize}
\item[(\romannumeral1)]tilted algebras of type $\mathbb{A}_4$: $(1)-(10)$;
\item[(\romannumeral2)]tilted algebras of type $\mathbb{A}_3\amalg\mathbb{A}_1$: $(11)-(13)$, $(15)$;
\item[(\romannumeral3)]tilted algebras of type $\mathbb{A}_2\amalg\mathbb{A}_2$: $(14)$.
\end{itemize}

\end{example}

\begin{example}
Let $A$ be the path algebra of the quiver
$$\xymatrix@C=5pt{Q= \stackrel{1}\circ\ar[rr]&&\stackrel{2}\circ&&
\stackrel{3}\circ\ar[rr]\ar[ll]&&\stackrel{4}\circ}$$
The AR-quiver of  $\text{K}^{[-1.0]}(\text{proj}A)$  is

$$
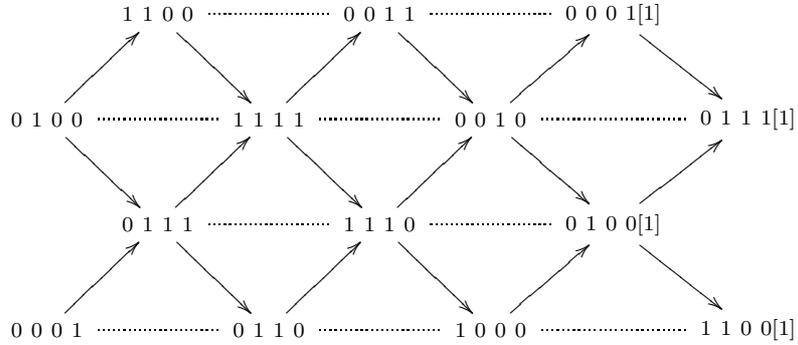
\xymatrix@C=3pt{&{\begin{smallmatrix}  1&1&0&0 \end{smallmatrix}}\ar@{..}[rr]\ar[rd]&&
{\begin{smallmatrix}  0&0&1&1 \end{smallmatrix}}\ar@{..}[rr]\ar[rd]&&
{\begin{smallmatrix}  0&0&0&1[1] \end{smallmatrix}}\ar[rd]\\
{\begin{smallmatrix}  0&1&0&0 \end{smallmatrix}}\ar@{..}[rr]\ar[ru]\ar[rd]&&
{\begin{smallmatrix}  1&1&1&1 \end{smallmatrix}}\ar@{..}[rr]\ar[ru]\ar[rd]&&
{\begin{smallmatrix}  0&0&1&0 \end{smallmatrix}}\ar@{..}[rr]\ar[ru]\ar[rd]&&
{\begin{smallmatrix}  0&1&1&1[1] \end{smallmatrix}}\\
&{\begin{smallmatrix}  0&1&1&1 \end{smallmatrix}}\ar@{..}[rr]\ar[ru]\ar[rd]&&
{\begin{smallmatrix}  1&1&1&0 \end{smallmatrix}}\ar@{..}[rr]\ar[ru]\ar[rd]&&
{\begin{smallmatrix}  0&1&0&0[1] \end{smallmatrix}}\ar[ru]\ar[rd]\\
{\begin{smallmatrix}  0&0&0&1 \end{smallmatrix}}\ar@{..}[rr]\ar[ru]&&
{\begin{smallmatrix}  0&1&1&0 \end{smallmatrix}}\ar@{..}[rr]\ar[ru]&&
{\begin{smallmatrix}  1&0&0&0 \end{smallmatrix}}\ar@{..}[rr]\ar[ru]&&
{\begin{smallmatrix}  1&1&0&0[1] \end{smallmatrix}}
}$$

(I) Tilted algebras of type $Q$ are

\[\begin{tabular}{|c|c|c|}
 \hline
No. & tilted algebras & tilting modules\\
\hline
(1)& \xymatrix{\circ \ar[r] & \circ\ar[r] & \circ\ar[r] & \circ} &
\newcommand{\scirc}{{\scriptstyle \circ}}
\newcommand{\sbullet}{{\scriptstyle\bullet}}
$\begin{array}{*{10}{c@{\hspace{1pt}}}}
&\sbullet&&\scirc\\[-9pt]
\scirc&&\sbullet&&\scirc\\[-9pt]
&\scirc&&\sbullet\\[-9pt]
\scirc&&\scirc&&\sbullet
\end{array}   \hspace{5mm} \begin{array}{*{10}{c@{\hspace{1pt}}}}
&\scirc&&\sbullet\\[-9pt]
\scirc&&\sbullet&&\scirc\\[-9pt]
&\sbullet&&\scirc\\[-9pt]
\sbullet&&\scirc&&\scirc
\end{array}$\\
\hline
(2)& \xymatrix{\circ \ar[r] & \circ\ar[r] & \circ & \circ\ar[l]} &
\newcommand{\scirc}{{\scriptstyle \circ}}
\newcommand{\sbullet}{{\scriptstyle\bullet}}
$\begin{array}{*{10}{c@{\hspace{1pt}}}}
&\scirc&&\sbullet\\[-9pt]
\scirc&&\sbullet&&\scirc\\[-9pt]
&\sbullet&&\scirc\\[-9pt]
\scirc&&\sbullet&&\scirc
\end{array}   \hspace{5mm}   \begin{array}{*{10}{c@{\hspace{1pt}}}}
&\sbullet&&\scirc\\[-9pt]
\sbullet&&\scirc&&\scirc\\[-9pt]
&\sbullet&&\scirc\\[-9pt]
\scirc&&\sbullet&&\scirc
\end{array} \hspace{5mm}   \begin{array}{*{10}{c@{\hspace{1pt}}}}
&\scirc&&\sbullet\\[-9pt]
\scirc&&\sbullet&&\scirc\\[-9pt]
&\scirc&&\sbullet\\[-9pt]
\scirc&&\scirc&&\sbullet
\end{array}$\\
\hline
(3)& \xymatrix{\circ \ar[r] & \circ & \circ\ar[l]\ar[r] & \circ} &
\newcommand{\scirc}{{\scriptstyle \circ}}
\newcommand{\sbullet}{{\scriptstyle\bullet}}
$\begin{array}{*{10}{c@{\hspace{1pt}}}}
&\sbullet&&\scirc\\[-9pt]
\scirc&&\sbullet&&\scirc\\[-9pt]
&\sbullet&&\scirc\\[-9pt]
\scirc&&\sbullet&&\scirc
\end{array} \hspace{5mm}   \begin{array}{*{10}{c@{\hspace{1pt}}}}
&\sbullet&&\scirc\\[-9pt]
\sbullet&&\scirc&&\scirc\\[-9pt]
&\sbullet&&\scirc\\[-9pt]
\sbullet&&\scirc&&\scirc
\end{array}\hspace{5mm} \begin{array}{*{10}{c@{\hspace{1pt}}}}
&\scirc&&\sbullet\\[-9pt]
\scirc&&\scirc&&\sbullet\\[-9pt]
&\scirc&&\sbullet\\[-9pt]
\scirc&&\scirc&&\sbullet
\end{array}
 \hspace{5mm}
 \begin{array}{*{10}{c@{\hspace{1pt}}}}
&\scirc&&\sbullet\\[-9pt]
\scirc&&\sbullet&&\scirc\\[-9pt]
&\scirc&&\sbullet\\[-9pt]
\scirc&&\sbullet&&\scirc
\end{array}$ \\
\hline
(4)& \xymatrix{\circ  & \circ\ar[r]\ar[l] & \circ\ar[r] & \circ} &
\newcommand{\scirc}{{\scriptstyle \circ}}
\newcommand{\sbullet}{{\scriptstyle\bullet}}
$\begin{array}{*{10}{c@{\hspace{1pt}}}}
&\sbullet&&\scirc\\[-9pt]
\scirc&&\sbullet&&\scirc\\[-9pt]
&\scirc&&\sbullet\\[-9pt]
\scirc&&\sbullet&&\scirc
\end{array}   \hspace{5mm}   \begin{array}{*{10}{c@{\hspace{1pt}}}}
&\sbullet&&\scirc\\[-9pt]
\scirc&&\sbullet&&\scirc\\[-9pt]
&\sbullet&&\scirc\\[-9pt]
\sbullet&&\scirc&&\scirc
\end{array} \hspace{5mm}\begin{array}{*{10}{c@{\hspace{1pt}}}}
&\scirc&&\sbullet\\[-9pt]
\scirc&&\scirc&&\sbullet\\[-9pt]
&\scirc&&\sbullet\\[-9pt]
\scirc&&\sbullet&&\scirc
\end{array} $\\
\hline
(5)& \makecell[c]{\xymatrix{\circ\ar[r]&\circ\ar[r]&\circ\\
&\circ\ar[u]\ar@/_/@{.}[ru]_{}}} &
\newcommand{\scirc}{{\scriptstyle \circ}}
\newcommand{\sbullet}{{\scriptstyle\bullet}}
$\begin{array}{*{10}{c@{\hspace{1pt}}}}
&\scirc&&\sbullet\\[-9pt]
\scirc&&\sbullet&&\scirc\\[-9pt]
&\scirc&&\scirc\\[-9pt]
\sbullet&&\scirc&&\sbullet
\end{array}$  \\
\hline
(6) &  \makecell[c]{\xymatrix{\circ\ar[r]\ar@/_/@{.}[rd]_{}&\circ\ar[d]
\ar[r]&\circ\\&\circ}}&
\newcommand{\scirc}{{\scriptstyle \circ}}
\newcommand{\sbullet}{{\scriptstyle\bullet}}
$\begin{array}{*{10}{c@{\hspace{1pt}}}}
&\sbullet&&\scirc\\[-9pt]
\scirc&&\sbullet&&\scirc\\[-9pt]
&\scirc&&\scirc\\[-9pt]
\sbullet&&\scirc&&\sbullet
\end{array}$\\
\hline
\end{tabular}
\]

%(II) Silted algebras of type $Q$ which are not tilted of type $Q$ are
(II) Silted algebras of the form $\text{End}_{\text{K}^b(\text{proj}A)}(M\oplus P[1])$ ($P\neq 0,~M\neq 0$) are
\[\begin{tabular}{|c|c|c|c|}
 \hline
No.& silted algebras & 2-term silting complexes &  tilted type\\
\hline
(7)&\xymatrix{\circ\ar[r]\ar@/^/@{.}[rr]_{}&
\circ\ar[r]&\circ\ar[r]&\circ}&
\newcommand{\scirc}{{\scriptstyle \circ}}
\newcommand{\sbullet}{{\scriptstyle\bullet}}
$\begin{array}{*{10}{c@{\hspace{1pt}}}}
&\scirc&&\sbullet&&\scirc\\[-7pt]
\scirc&&\scirc&&\scirc&&\scirc\\[-7pt]
&\sbullet&&\scirc&&\scirc\\[-7pt]
\sbullet&&\scirc&&\scirc&&\sbullet
\end{array}$&$\mathbb{A}_4$\\
\hline
(8)&\xymatrix{\circ\ar[r]&
\circ\ar[r]\ar@/^/@{.}[rr]_{}&\circ\ar[r]&\circ}&
\newcommand{\scirc}{{\scriptstyle \circ}}
\newcommand{\sbullet}{{\scriptstyle\bullet}}
$\begin{array}{*{10}{c@{\hspace{1pt}}}}
&\scirc&&\sbullet&&\scirc\\[-7pt]
\scirc&&\scirc&&\scirc&&\scirc\\[-7pt]
&\scirc&&\scirc&&\sbullet\\[-7pt]
\sbullet&&\scirc&&\scirc&&\sbullet
\end{array}$&$\mathbb{A}_4$\\
\hline
(9)&\xymatrix{\circ\ar[r]&\circ&\circ\ar[l]
&\circ\ar[l]\ar@/_/@{.}[ll]_{}}&
\newcommand{\scirc}{{\scriptstyle \circ}}
\newcommand{\sbullet}{{\scriptstyle\bullet}}
$\begin{array}{*{10}{c@{\hspace{1pt}}}}
&\scirc&&\sbullet&&\scirc\\[-7pt]
\scirc&&\scirc&&\scirc&&\scirc\\[-7pt]
&\sbullet&&\scirc&&\scirc\\[-7pt]
\scirc&&\sbullet&&\scirc&&\sbullet
\end{array}\hspace{5mm} \begin{array}{*{10}{c@{\hspace{1pt}}}}
&\sbullet&&\scirc&&\sbullet\\[-7pt]
\sbullet&&\scirc&&\scirc&&\scirc\\[-7pt]
&\scirc&&\scirc&&\scirc\\[-7pt]
\scirc&&\sbullet&&\scirc&&\scirc
\end{array}$& $\mathbb{A}_4$\\
\hline

(10)&\xymatrix{\circ&\circ\ar[l]&\circ
\ar[l]\ar[r]\ar@/_/@{.}[ll]_{}&\circ}&
\newcommand{\scirc}{{\scriptstyle \circ}}
\newcommand{\sbullet}{{\scriptstyle\bullet}}
$\begin{array}{*{10}{c@{\hspace{1pt}}}}
&\scirc&&\sbullet&&\scirc\\[-7pt]
\scirc&&\scirc&&\scirc&&\scirc\\[-7pt]
&\scirc&&\scirc&&\sbullet\\[-7pt]
\sbullet&&\scirc&&\sbullet&&\scirc
\end{array}  \hspace{5mm} \begin{array}{*{10}{c@{\hspace{1pt}}}}
&\sbullet&&\scirc&&\sbullet\\[-7pt]
\scirc&&\scirc&&\scirc&&\sbullet\\[-7pt]
&\scirc&&\scirc&&\scirc\\[-7pt]
\scirc&&\scirc&&\sbullet&&\scirc
\end{array}$&$\mathbb{A}_4$\\
\hline

(5)&\makecell[c]{\xymatrix{\circ\ar[r]&\circ\ar[r]&\circ\\
&\circ\ar[u]\ar@/_/@{.}[ru]_{}
}}&
\newcommand{\scirc}{{\scriptstyle \circ}}
\newcommand{\sbullet}{{\scriptstyle\bullet}}
$\begin{array}{*{10}{c@{\hspace{1pt}}}}
&\sbullet&&\scirc&&\sbullet\\[-7pt]
\scirc&&\scirc&&\scirc&&\scirc\\[-7pt]
&\scirc&&\sbullet&&\scirc\\[-7pt]
\scirc&&\scirc&&\sbullet&&\scirc
\end{array}$
%\hspace{5mm}  \begin{array}{*{20}{c@{\hspace{3pt}}}}
%&\scirc&&\scirc&&\sbullet\\[-5pt]
%\scirc&&\scirc&&\sbullet&&\scirc\\[-5pt]
%&\scirc&&\scirc&&\scirc\\[-5pt]
%\scirc&&\sbullet&&\scirc&&\sbullet
%\end{array}$
&$\mathbb{A}_4$\\
\hline
(6)&\makecell[c]{\xymatrix{\circ\ar[r]\ar@/_/@{.}[rd]_{}
&\circ\ar[d]\ar[r]&
\circ\\
&\circ
}}&
\newcommand{\scirc}{{\scriptstyle \circ}}
\newcommand{\sbullet}{{\scriptstyle\bullet}}
$\begin{array}{*{10}{c@{\hspace{1pt}}}}
&\sbullet&&\scirc&&\sbullet\\[-7pt]
\scirc&&\scirc&&\scirc&&\scirc\\[-7pt]
&\scirc&&\sbullet&&\scirc\\[-7pt]
\scirc&&\sbullet&&\scirc&&\scirc
\end{array}$&$\mathbb{A}_4$\\
\hline

(11)&\xymatrix{\circ\ar[r]&\circ\ar[r]
&\circ~~~\circ}&
\newcommand{\scirc}{{\scriptstyle \circ}}
\newcommand{\sbullet}{{\scriptstyle\bullet}}
$\begin{array}{*{10}{c@{\hspace{1pt}}}}
&\scirc&&\scirc&&\scirc\\[-7pt]
\sbullet&&\scirc&&\scirc&&\scirc\\[-7pt]
&\sbullet&&\scirc&&\scirc\\[-7pt]
\scirc&&\sbullet&&\scirc&&\sbullet
\end{array}   \hspace{5mm}  \begin{array}{*{10}{c@{\hspace{1pt}}}}
&\scirc&&\scirc&&\scirc\\[-7pt]
\scirc&&\scirc&&\scirc&&\sbullet\\[-7pt]
&\scirc&&\scirc&&\sbullet\\[-7pt]
\sbullet&&\scirc&&\sbullet&&\scirc
\end{array}$& $\mathbb{A}_3\amalg\mathbb{A}_1$\\
\hline
(12)&\xymatrix{\circ&\circ\ar[r]\ar[l]
&\circ~~~\circ}&
\newcommand{\scirc}{{\scriptstyle \circ}}
\newcommand{\sbullet}{{\scriptstyle\bullet}}
$\begin{array}{*{10}{c@{\hspace{1pt}}}}
&\scirc&&\scirc&&\scirc\\[-7pt]
\sbullet&&\scirc&&\scirc&&\scirc\\[-7pt]
&\sbullet&&\scirc&&\scirc\\[-7pt]
\sbullet&&\scirc&&\scirc&&\sbullet
\end{array}$&$\mathbb{A}_3\amalg\mathbb{A}_1$\\
\hline

(13)&\xymatrix{\circ\ar[r]\ar@/^/@{.}[rr]_{}&
\circ\ar[r]&\circ~~~\circ}&
\newcommand{\scirc}{{\scriptstyle \circ}}
\newcommand{\sbullet}{{\scriptstyle\bullet}}
$\begin{array}{*{10}{c@{\hspace{1pt}}}}
&\sbullet&&\scirc&&\scirc\\[-7pt]
\scirc&&\scirc&&\scirc&&\sbullet\\[-7pt]
&\scirc&&\scirc&&\scirc\\[-7pt]
\sbullet&&\scirc&&\sbullet&&\scirc
\end{array} \hspace{5mm} \begin{array}{*{10}{c@{\hspace{1pt}}}}
&\scirc&&\scirc&&\sbullet\\[-7pt]
\sbullet&&\scirc&&\scirc&&\scirc\\[-7pt]
&\scirc&&\scirc&&\scirc\\[-7pt]
\scirc&&\sbullet&&\scirc&&\sbullet
\end{array}$&$\mathbb{A}_3\amalg\mathbb{A}_1$\\
\hline
(14)&\xymatrix{\circ\ar[r]&\circ
&\circ\ar[l]~~~\circ}&
\newcommand{\scirc}{{\scriptstyle \circ}}
\newcommand{\sbullet}{{\scriptstyle\bullet}}
$\begin{array}{*{10}{c@{\hspace{1pt}}}}
&\scirc&&\scirc&&\scirc\\[-7pt]
\scirc&&\scirc&&\scirc&&\sbullet\\[-7pt]
&\scirc&&\scirc&&\sbullet\\[-7pt]
\sbullet&&\scirc&&\scirc&&\sbullet
\end{array}$&$\mathbb{A}_3\amalg\mathbb{A}_1$\\
\hline

(15)&\xymatrix{\circ\ar[r]&\circ\ar[r]~~~
\circ\ar[r]&\circ}&
\newcommand{\scirc}{{\scriptstyle \circ}}
\newcommand{\sbullet}{{\scriptstyle\bullet}}
$\begin{array}{*{10}{c@{\hspace{1pt}}}}
&\sbullet&&\scirc&&\sbullet\\[-7pt]
\sbullet&&\scirc&&\scirc&&\sbullet\\[-7pt]
&\scirc&&\scirc&&\scirc\\[-7pt]
\scirc&&\scirc&&\scirc&&\scirc
\end{array}$&$\mathbb{A}_2\amalg\mathbb{A}_2$\\
\hline
(16)&\xymatrix{\circ\ar[r]&\circ~~~
\circ~~~\circ}&
\newcommand{\scirc}{{\scriptstyle \circ}}
\newcommand{\sbullet}{{\scriptstyle\bullet}}
$\begin{array}{*{10}{c@{\hspace{1pt}}}}
&\sbullet&&\scirc&&\scirc\\[-7pt]
\sbullet&&\scirc&&\scirc&&\sbullet\\[-7pt]
&\scirc&&\scirc&&\scirc\\[-7pt]
\sbullet&&\scirc&&\scirc&&\scirc
\end{array} \hspace{5mm}   \begin{array}{*{10}{c@{\hspace{1pt}}}}
&\scirc&&\scirc&&\sbullet\\[-7pt]
\sbullet&&\scirc&&\scirc&&\sbullet\\[-7pt]
&\scirc&&\scirc&&\scirc\\[-7pt]
\scirc&&\scirc&&\scirc&&\sbullet
\end{array}$&$\mathbb{A}_2\amalg\mathbb{A}_1\amalg\mathbb{A}_1$\\
\hline
(17)&\xymatrix{\circ~~~
\circ~~~\circ~~~\circ}&
\newcommand{\scirc}{{\scriptstyle \circ}}
\newcommand{\sbullet}{{\scriptstyle\bullet}}
$\begin{array}{*{10}{c@{\hspace{1pt}}}}
&\scirc&&\scirc&&\scirc\\[-7pt]
\sbullet&&\scirc&&\scirc&&\sbullet\\[-7pt]
&\scirc&&\scirc&&\scirc\\[-7pt]
\sbullet&&\scirc&&\scirc&&\sbullet
\end{array}$&$\mathbb{A}_1\amalg\mathbb{A}_1\amalg\mathbb{A}_1\amalg\mathbb{A}_1$\\
\hline
\end{tabular}
\]

 To summarise, there are 17 silted algebras  of type $Q$, forming 5 families:
\begin{itemize}
\item[(\romannumeral1)]tilted algebras of type $\mathbb{A}_4$: $(1)-(10)$;
\item[(\romannumeral2)]tilted algebras of type $\mathbb{A}_3\amalg\mathbb{A}_1$: $(11)-(14)$;
\item[(\romannumeral3)]tilted algebras of type $\mathbb{A}_2\amalg\mathbb{A}_2$: $(15)$;
\item[(\romannumeral4)]tilted algebras of type $\mathbb{A}_2\amalg\mathbb{A}_1\amalg\mathbb{A}_1$: $(16)$;
\item[(\romannumeral5)]tilted algebras of type $\mathbb{A}_1\amalg\mathbb{A}_1\amalg\mathbb{A}_1\amalg\mathbb{A}_1$: $(17)$.
\end{itemize}

\end{example}
%\subsection*{3.2.5 Summary}
%
%\paragraph{\indent Let $Q$ be the quiver $\stackrel{1}\circ\longrightarrow \stackrel{2}\circ \longrightarrow \cdots \longrightarrow \stackrel{n}\circ$ with $1\leq n\leq4$, and $A=KQ$. To summarize the
%computations on Section 4.2.1, Section 4.2.2, Section 4.2.3 and
%Section 4.2.4, we see that all the silted algebras of type $Q$ are tilted. Indeed, all of them are string algebras and have no consecutive-zero, so they are quasi-tilted by Theorem\ref{c-z}.
%It then follows from Theorem\ref{s} that they are tilted. In fact, they are tilted either of type $\mathbb{A}_n$ or of type $\mathbb{A}_m\coprod\mathbb{A}_{n-m}$, $1\leq m\leq {n-1}.$}

\begin{example}
Let $A$ be the path algebra of the quiver
$$\xymatrix@C=5pt{Q= \stackrel{1}\circ\ar[rr]&&\stackrel{2}\circ\ar[rr]&&
\stackrel{3}\circ&&\stackrel{4}\circ\ar[ll]}$$
The AR-quiver of  $\text{K}^{[-1.0]}(\text{proj}A)$  is

$$
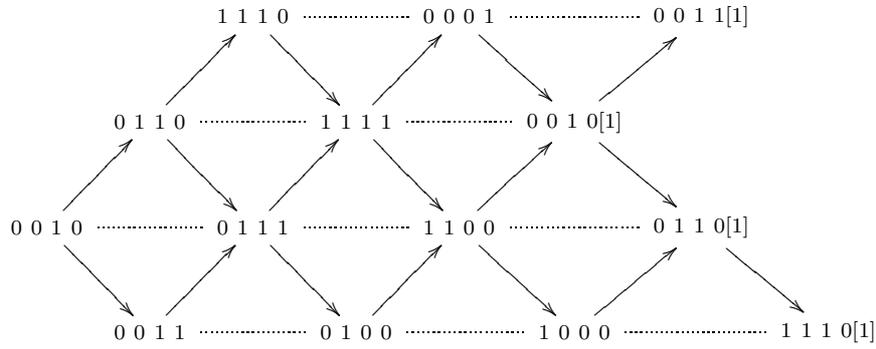
\xymatrix@C=0.1pt{&&{\begin{smallmatrix}  1&1&1&0 \end{smallmatrix}}\ar[rd]\ar@{..}[rr] &&{\begin{smallmatrix}  0&0&0&1 \end{smallmatrix}}\ar[rd]\ar@{..}[rr]
&&{\begin{smallmatrix}  0&0&1&1[1] \end{smallmatrix}}\\
& {\begin{smallmatrix}  0&1&1&0 \end{smallmatrix}}\ar@{..}[rr]\ar[ru]\ar[rd]&&
{\begin{smallmatrix}  1&1&1&1 \end{smallmatrix}}\ar[rd]\ar[ru]\ar@{..}[rr]&&{\begin{smallmatrix}  0&0&1&0[1] \end{smallmatrix}}\ar[rd]\ar[ru] \\
{\begin{smallmatrix}  0&0&1&0 \end{smallmatrix}}\ar@{..}[rr]\ar[ru]\ar[rd]&&
{\begin{smallmatrix}  0&1&1&1 \end{smallmatrix}}\ar@{..}[rr]\ar[ru]\ar[rd]&&
{\begin{smallmatrix}  1&1&0&0 \end{smallmatrix}}\ar[rd]\ar@{..}[rr]\ar[ru]&&{\begin{smallmatrix}  0&1&1&0[1] \end{smallmatrix}}\ar[rd]\\
&{\begin{smallmatrix}  0&0&1&1 \end{smallmatrix}}\ar@{..}[rr]\ar[ru]&&
{\begin{smallmatrix}  0&1&0&0 \end{smallmatrix}}\ar@{..}[rr]\ar[ru]&&
{\begin{smallmatrix}  1&0&0&0 \end{smallmatrix}}\ar@{..}[rr]\ar[ru]&&
{\begin{smallmatrix}  1&1&1&0[1] \end{smallmatrix}}
}$$

(I) Tilted algebras of type $Q$ are

\[\begin{tabular}{|c|c|c|}
 \hline
No. & tilted algebras & tilting modules\\
\hline
(1)& \xymatrix{\circ \ar[r] & \circ\ar[r] & \circ\ar[r] & \circ} &
\newcommand{\scirc}{{\scriptstyle \circ}}
\newcommand{\sbullet}{{\scriptstyle\bullet}}
$\begin{array}{*{10}{c@{\hspace{1pt}}}}
&&\sbullet&&\scirc\\[-9pt]
&\scirc&&\sbullet\\[-9pt]
\scirc&&\scirc&&\sbullet\\[-9pt]
&\scirc&&\scirc&&\sbullet\\
\end{array}  \hspace{3mm} \begin{array}{*{10}{c@{\hspace{1pt}}}}
&&\scirc&&\sbullet\\[-9pt]
&\scirc&&\sbullet\\[-9pt]
\scirc&&\sbullet&&\scirc\\[-9pt]
&\sbullet&&\scirc&&\scirc
\end{array}$\\
\hline
(2)& \xymatrix{\circ \ar[r] & \circ\ar[r] & \circ & \circ\ar[l]} &
\newcommand{\scirc}{{\scriptstyle \circ}}
\newcommand{\sbullet}{{\scriptstyle\bullet}}
$\begin{array}{*{10}{c@{\hspace{1pt}}}}
&&\sbullet&&\scirc\\[-9pt]
&\sbullet&&\scirc\\[-9pt]
\scirc&&\sbullet&&\scirc\\[-9pt]
&\scirc&&\sbullet&&\scirc
\end{array}\hspace{3mm}\begin{array}{*{10}{c@{\hspace{1pt}}}}
&&\sbullet&&\scirc\\[-9pt]
&\sbullet&&\scirc\\[-9pt]
\sbullet&&\scirc&&\scirc\\[-8pt]
&\sbullet&&\scirc&&\scirc
\end{array}\hspace{3mm}\begin{array}{*{10}{c@{\hspace{1pt}}}}
&&\scirc&&\sbullet\\[-9pt]
&\scirc&&\sbullet\\[-9pt]
\scirc&&\scirc&&\sbullet\\[-9pt]
&\scirc&&\scirc&&\sbullet
\end{array}\hspace{3mm}\begin{array}{*{10}{c@{\hspace{1pt}}}}
&&\scirc&&\sbullet\\[-9pt]
&\scirc&&\sbullet\\[-9pt]
\scirc&&\sbullet&&\scirc\\[-9pt]
&\scirc&&\sbullet&&\scirc
\end{array}$\\
\hline

(3)& \xymatrix{\circ \ar[r] & \circ & \circ\ar[l]\ar[r] & \circ} &
\newcommand{\scirc}{{\scriptstyle \circ}}
\newcommand{\sbullet}{{\scriptstyle\bullet}}
$\begin{array}{*{10}{c@{\hspace{1pt}}}}
&&\sbullet&&\scirc\\[-9pt]
&\scirc&&\sbullet\\[-9pt]
\scirc&&\sbullet&&\scirc\\[-9pt]
&\scirc&&\sbullet&&\scirc
\end{array}\hspace{3mm}\begin{array}{*{10}{c@{\hspace{1pt}}}}
&&\sbullet&&\scirc\\[-9pt]
&\sbullet&&\scirc\\[-9pt]
\scirc&&\sbullet&&\scirc\\[-9pt]
&\sbullet&&\scirc&&\scirc
\end{array}\hspace{3mm}\begin{array}{*{10}{c@{\hspace{1pt}}}}
&&\scirc&&\sbullet\\[-9pt]
&\scirc&&\sbullet\\[-9pt]
\scirc&&\scirc&&\sbullet\\[-9pt]
&\scirc&&\sbullet&&\scirc
\end{array}$\\
\hline
(4)& \xymatrix{\circ  & \circ\ar[r]\ar[l] & \circ\ar[r] & \circ} &
\newcommand{\scirc}{{\scriptstyle \circ}}
\newcommand{\sbullet}{{\scriptstyle\bullet}}
$\begin{array}{*{10}{c@{\hspace{1pt}}}}
&&\sbullet&&\scirc\\[-9pt]
&\scirc&&\sbullet\\[-9pt]
\scirc&&\scirc&&\sbullet\\[-9pt]
&\scirc&&\sbullet&&\scirc
\end{array}\hspace{3mm} \begin{array}{*{10}{c@{\hspace{1pt}}}}
&&\sbullet&&\scirc\\[-9pt]
&\scirc&&\sbullet\\[-9pt]
\scirc&&\sbullet&&\scirc\\[-9pt]
&\sbullet&&\scirc&&\scirc
\end{array}$\\
\hline
(5)&\xymatrix{\circ\ar[r]&\circ&\circ\ar[l]
&\circ\ar[l]\ar@/_/@{.}[ll]_{}}&
\newcommand{\scirc}{{\scriptstyle \circ}}
\newcommand{\sbullet}{{\scriptstyle\bullet}}
$\begin{array}{*{10}{c@{\hspace{1pt}}}}
&&\sbullet&&\scirc\\[-9pt]
&\scirc&&\scirc\\[-9pt]
\sbullet&&\scirc&&\scirc\\[-9pt]
&\sbullet&&\scirc&&\sbullet
\end{array}$\\
\hline
(6)&\makecell[c]{\xymatrix{\circ\ar[r]&\circ\ar[r]&\circ\\
&\circ\ar[u]\ar@/_/@{.}[ru]_{}
}}&
\newcommand{\scirc}{{\scriptstyle \circ}}
\newcommand{\sbullet}{{\scriptstyle\bullet}}
$\begin{array}{*{10}{c@{\hspace{1pt}}}}
&&\scirc&&\sbullet\\[-9pt]
&\scirc&&\sbullet\\[-9pt]
\scirc&&\scirc&&\scirc\\[-9pt]
&\sbullet&&\scirc&&\sbullet
\end{array}$\\
\hline
(7)&\makecell[c]{\xymatrix{\circ\ar[r]\ar@/_/@{.}[rd]_{}
&\circ\ar[d]\ar[r]&
\circ\\
&\circ
}}&
\newcommand{\scirc}{{\scriptstyle \circ}}
\newcommand{\sbullet}{{\scriptstyle\bullet}}
$\begin{array}{*{10}{c@{\hspace{1pt}}}}
&&\sbullet&&\scirc\\[-9pt]
&\scirc&&\sbullet\\[-9pt]
\scirc&&\scirc&&\scirc\\[-9pt]
&\sbullet&&\scirc&&\sbullet
\end{array}$\\
\hline

\end{tabular}
\]

%(II) Silted algebras of type $Q$ which are not tilted of type $Q$ are
(II) Silted algebras of the form $\text{End}_{\text{K}^b(\text{proj}A)}(M\oplus P[1])$ ($P\neq 0,~M\neq 0$) are
\[\begin{tabular}{|c|c|c|c|}
 \hline

No.& silted algebras & 2-term silting complexes &  tilted type\\
\hline

(8)&\xymatrix{\circ\ar[r]&\circ\ar[r]
&\circ~~~\circ}&
\newcommand{\scirc}{{\scriptstyle \circ}}
\newcommand{\sbullet}{{\scriptstyle\bullet}}
$\begin{array}{*{10}{c@{\hspace{1pt}}}}
&&\scirc&&\scirc&&\scirc\\[-8pt]
&\sbullet&&\scirc&&\scirc\\[-8pt]
\scirc&&\sbullet&&\scirc&&\scirc\\[-8pt]
&\scirc&&\sbullet&&\scirc&&\sbullet
\end{array}\hspace{3mm}\begin{array}{*{10}{c@{\hspace{1pt}}}}
&&\sbullet&&\scirc&&\sbullet\\[-8pt]
&\sbullet&&\scirc&&\scirc\\[-8pt]
\sbullet&&\scirc&&\scirc&&\scirc\\[-8pt]
&\scirc&&\scirc&&\scirc&&\scirc
\end{array}$
%\hspace{3mm} \begin{array}{*{20}{c@{\hspace{3pt}}}}
%&&\scirc&&\sbullet&&\scirc\\[-5pt]
%&\scirc&&\scirc&&\sbullet\\[-5pt]
%\scirc&&\scirc&&\scirc&&\scirc\\[-5pt]
%&\scirc&&\sbullet&&\scirc&&\sbullet
%\end{array}$
&$\mathbb{A}_3\amalg\mathbb{A}_1$\\
\hline

(9)&\xymatrix{\circ&\circ\ar[r]\ar[l]
&\circ~~~\circ}&
\newcommand{\scirc}{{\scriptstyle \circ}}
\newcommand{\sbullet}{{\scriptstyle\bullet}}
$\begin{array}{*{10}{c@{\hspace{1pt}}}}
&&\scirc&&\scirc&&\scirc\\[-8pt]
&\sbullet&&\scirc&&\scirc\\[-8pt]
\scirc&&\sbullet&&\scirc&&\scirc\\[-8pt]
&\sbullet&&\scirc&&\scirc&&\sbullet
\end{array}$&$\mathbb{A}_3\amalg\mathbb{A}_1$\\
\hline
(10)&\xymatrix{\circ\ar[r]&\circ
&\circ\ar[l]~~~\circ}&
\newcommand{\scirc}{{\scriptstyle \circ}}
\newcommand{\sbullet}{{\scriptstyle\bullet}}
$\begin{array}{*{10}{c@{\hspace{1pt}}}}
&&\scirc&&\scirc&&\scirc\\[-8pt]
&\sbullet&&\scirc&&\scirc\\[-8pt]
\sbullet&&\scirc&&\scirc&&\scirc\\[-8pt]
&\sbullet&&\scirc&&\scirc&&\sbullet
\end{array}$&$\mathbb{A}_3\amalg\mathbb{A}_1$\\
\hline

(11)&\xymatrix{\circ\ar[r]&\circ\ar[r]~~~
\circ\ar[r]&\circ}&
\newcommand{\scirc}{{\scriptstyle \circ}}
\newcommand{\sbullet}{{\scriptstyle\bullet}}
$\begin{array}{*{10}{c@{\hspace{1pt}}}}
&&\scirc&&\scirc&&\scirc\\[-8pt]
&\scirc&&\scirc&&\scirc\\[-8pt]
\sbullet&&\scirc&&\scirc&&\sbullet\\[-8pt]
&\sbullet&&\scirc&&\sbullet&&\scirc
\end{array}  \hspace{3mm}\begin{array}{*{10}{c@{\hspace{1pt}}}}
&&\scirc&&\scirc&&\scirc\\[-8pt]
&\scirc&&\scirc&&\scirc\\[-8pt]
\sbullet&&\scirc&&\scirc&&\sbullet\\[-8pt]
&\sbullet&&\scirc&&\scirc&&\sbullet
\end{array}$&$\mathbb{A}_2\amalg\mathbb{A}_2$\\
\hline
(12)&\xymatrix{\circ\ar[r]&\circ~~~
\circ~~~\circ}&
\newcommand{\scirc}{{\scriptstyle \circ}}
\newcommand{\sbullet}{{\scriptstyle\bullet}}
$\begin{array}{*{10}{c@{\hspace{1pt}}}}
&&\scirc&&\scirc&&\sbullet\\[-8pt]
&\sbullet&&\scirc&&\scirc\\[-8pt]
\sbullet&&\scirc&&\scirc&&\scirc\\[-8pt]
&\scirc&&\scirc&&\scirc&&\sbullet
\end{array}\hspace{3mm}\begin{array}{*{10}{c@{\hspace{1pt}}}}
&&\scirc&&\scirc&&\sbullet\\[-8pt]
&\scirc&&\scirc&&\scirc\\[-8pt]
\sbullet&&\scirc&&\scirc&&\sbullet\\[-8pt]
&\scirc&&\scirc&&\sbullet&&\scirc
\end{array}\hspace{3mm}\begin{array}{*{10}{c@{\hspace{1pt}}}}
&&\scirc&&\scirc&&\sbullet\\[-8pt]
&\scirc&&\scirc&&\scirc\\[-8pt]
\sbullet&&\scirc&&\scirc&&\sbullet\\[-8pt]
&\scirc&&\scirc&&\scirc&&\sbullet
\end{array}$&$\mathbb{A}_2\amalg\mathbb{A}_1\amalg\mathbb{A}_1$\\
\hline

(13)&\xymatrix{\circ\ar[r]\ar@/^/@{.}[rr]_{}&
\circ\ar[r]&\circ~~~\circ}&
\newcommand{\scirc}{{\scriptstyle \circ}}
\newcommand{\sbullet}{{\scriptstyle\bullet}}
$\begin{array}{*{10}{c@{\hspace{1pt}}}}
&&\sbullet&&\scirc&&\sbullet\\[-8pt]
&\scirc&&\scirc&&\scirc\\[-8pt]
\sbullet&&\scirc&&\scirc&&\scirc\\[-8pt]
&\scirc&&\scirc&&\sbullet&&\scirc
\end{array} \hspace{3mm}\begin{array}{*{10}{c@{\hspace{1pt}}}}
&&\scirc&&\scirc&&\sbullet\\[-8pt]
&\sbullet&&\scirc&&\scirc\\[-8pt]
\scirc&&\scirc&&\scirc&&\scirc\\[-8pt]
&\scirc&&\sbullet&&\scirc&&\sbullet
\end{array}$&$\mathbb{A}_3\amalg\mathbb{A}_1$\\
\hline
(14)&\xymatrix{\circ\ar[r]\ar@/^/@{.}[rr]_{}&
\circ\ar[r]&\circ\ar[r]&\circ}&
\newcommand{\scirc}{{\scriptstyle \circ}}
\newcommand{\sbullet}{{\scriptstyle\bullet}}
$\begin{array}{*{10}{c@{\hspace{1pt}}}}
&&\scirc&&\sbullet&&\scirc\\[-8pt]
&\scirc&&\scirc&&\scirc\\[-8pt]
\scirc&&\sbullet&&\scirc&&\scirc\\[-8pt]
&\sbullet&&\scirc&&\scirc&&\sbullet
\end{array}$&$\mathbb{A}_4$\\
\hline
(15)&\xymatrix{\circ\ar[r]&
\circ\ar[r]\ar@/^/@{.}[rr]_{}&\circ\ar[r]&\circ}&
\newcommand{\scirc}{{\scriptstyle \circ}}
\newcommand{\sbullet}{{\scriptstyle\bullet}}
$\begin{array}{*{10}{c@{\hspace{1pt}}}}
&&\scirc&&\sbullet&&\scirc\\[-8pt]
&\scirc&&\scirc&&\scirc\\[-8pt]
\scirc&&\scirc&&\scirc&&\sbullet\\[-8pt]
&\sbullet&&\scirc&&\scirc&&\sbullet
\end{array}$&$\mathbb{A}_4$\\
\hline
(5)&\xymatrix{\circ\ar[r]&\circ&\circ\ar[l]
&\circ\ar[l]\ar@/_/@{.}[ll]_{}}&
\newcommand{\scirc}{{\scriptstyle \circ}}
\newcommand{\sbullet}{{\scriptstyle\bullet}}
%$\begin{array}{*{20}{c@{\hspace{3pt}}}}
%&&\scirc&&\sbullet&&\scirc\\[-5pt]
%&\scirc&&\scirc&&\scirc\\[-5pt]
%\scirc&&\sbullet&&\scirc&&\scirc\\[-5pt]
%&\scirc&&\sbullet&&\scirc&&\sbullet
%\end{array}\hspace{3mm}
$\begin{array}{*{10}{c@{\hspace{1pt}}}}
&&\sbullet&&\scirc&&\sbullet\\[-8pt]
&\sbullet&&\scirc&&\scirc\\[-8pt]
\scirc&&\scirc&&\scirc&&\scirc\\[-8pt]
&\scirc&&\sbullet&&\scirc&&\scirc
\end{array}$&$\mathbb{A}_4$\\
\hline
(16)&\xymatrix{\circ&\circ\ar[l]&\circ
\ar[l]\ar[r]\ar@/_/@{.}[ll]_{}&\circ}&
\newcommand{\scirc}{{\scriptstyle \circ}}
\newcommand{\sbullet}{{\scriptstyle\bullet}}
$\begin{array}{*{10}{c@{\hspace{1pt}}}}
&&\scirc&&\sbullet&&\scirc\\[-8pt]
&\scirc&&\scirc&&\scirc\\[-8pt]
\scirc&&\scirc&&\scirc&&\sbullet\\[-8pt]
&\sbullet&&\scirc&&\sbullet&&\scirc
\end{array}$&$\mathbb{A}_4$\\
\hline
(6)&\makecell[c]{\xymatrix{\circ\ar[r]&\circ\ar[r]&\circ\\
&\circ\ar[u]\ar@/_/@{.}[ru]_{}
}}&
\newcommand{\scirc}{{\scriptstyle \circ}}
\newcommand{\sbullet}{{\scriptstyle\bullet}}
$\begin{array}{*{10}{c@{\hspace{1pt}}}}
&&\sbullet&&\scirc&&\sbullet\\[-8pt]
&\scirc&&\scirc&&\scirc\\[-8pt]
\scirc&&\scirc&&\sbullet&&\scirc\\[-8pt]
&\scirc&&\scirc&&\sbullet&&\scirc
\end{array}$
%\hspace{3mm}\begin{array}{*{20}{c@{\hspace{3pt}}}}
%&&\scirc&&\scirc&&\sbullet\\[-5pt]
%&\scirc&&\scirc&&\sbullet\\[-5pt]
%\scirc&&\scirc&&\scirc&&\scirc\\[-5pt]
%&\scirc&&\sbullet&&\scirc&&\sbullet
%\end{array}$
&$\mathbb{A}_4$\\
\hline
(7)&\makecell[c]{\xymatrix{\circ\ar[r]\ar@/_/@{.}[rd]_{}
&\circ\ar[d]\ar[r]&
\circ\\
&\circ
}}&
\newcommand{\scirc}{{\scriptstyle \circ}}
\newcommand{\sbullet}{{\scriptstyle\bullet}}
$\begin{array}{*{10}{c@{\hspace{1pt}}}}
&&\sbullet&&\scirc&&\sbullet\\[-8pt]
&\scirc&&\scirc&&\scirc\\[-8pt]
\scirc&&\scirc&&\sbullet&&\scirc\\[-8pt]
&\scirc&&\sbullet&&\scirc&&\scirc
\end{array}$&$\mathbb{A}_4$\\
\hline
\end{tabular}
\]

%\[\begin{tabular}{|c|c|c|c|}
% \hline
%
%No.& silted algebras & 2-term silting complexes &  tilted type\\
%\hline
%(6)&\makecell[c]{\xymatrix{\circ\ar[r]&\circ\ar[r]&\circ\\
%&\circ\ar[u]\ar@/_/@{.}[ru]_{}
%}}&
%\newcommand{\scirc}{{\scriptstyle \circ}}
%\newcommand{\sbullet}{{\scriptstyle\bullet}}
%$\begin{array}{*{20}{c@{\hspace{3pt}}}}
%&&\sbullet&&\scirc&&\sbullet\\[-5pt]
%&\scirc&&\scirc&&\scirc\\[-5pt]
%\scirc&&\scirc&&\sbullet&&\scirc\\[-5pt]
%&\scirc&&\scirc&&\sbullet&&\scirc
%\end{array}$
%%\hspace{3mm}\begin{array}{*{20}{c@{\hspace{3pt}}}}
%%&&\scirc&&\scirc&&\sbullet\\[-5pt]
%%&\scirc&&\scirc&&\sbullet\\[-5pt]
%%\scirc&&\scirc&&\scirc&&\scirc\\[-5pt]
%%&\scirc&&\sbullet&&\scirc&&\sbullet
%%\end{array}$
%&$\mathbb{A}_4$\\
%\hline
%(7)&\makecell[c]{\xymatrix{\circ\ar[r]\ar@/_/@{.}[rd]_{}
%&\circ\ar[d]\ar[r]&
%\circ\\
%&\circ
%}}&
%\newcommand{\scirc}{{\scriptstyle \circ}}
%\newcommand{\sbullet}{{\scriptstyle\bullet}}
%$\begin{array}{*{20}{c@{\hspace{3pt}}}}
%&&\sbullet&&\scirc&&\sbullet\\[-5pt]
%&\scirc&&\scirc&&\scirc\\[-5pt]
%\scirc&&\scirc&&\sbullet&&\scirc\\[-5pt]
%&\scirc&&\sbullet&&\scirc&&\scirc
%\end{array}$&$\mathbb{A}_4$\\
%\hline
%\end{tabular}
%\]

%\newcommand{\scirc}{{\scriptstyle \circ}}
%\newcommand{\sbullet}{{\scriptstyle\bullet}}
%\[\begin{array}{*{20}{c@{\hspace{3pt}}}}
%&&\scirc&&\scirc&&\scirc\\[-5pt]
%&\scirc&&\scirc&&\scirc\\[-5pt]
%\scirc&&\scirc&&\scirc&&\scirc\\[-5pt]
%&\scirc&&\scirc&&\scirc&&\scirc
%\end{array}
%\]
To summarise, there are 16 silted algebras  of type $Q$, forming 4 families:
\begin{itemize}
\item[(\romannumeral1)]tilted algebras of type $\mathbb{A}_4$: $(1)-(7)$,$(14)-(16)$;
\item[(\romannumeral2)]tilted algebras of type $\mathbb{A}_3\amalg\mathbb{A}_1$: $(8)-(10)$,$(13)$;
\item[(\romannumeral3)]tilted algebras of type $\mathbb{A}_2\amalg\mathbb{A}_2$: $(11)$;
\item[(\romannumeral4)]tilted algebras of type $\mathbb{A}_2\amalg\mathbb{A}_1\amalg\mathbb{A}_1$: $(12)$.

\end{itemize}
\end{example}

\vspace{13pt}
\subsection*{3.3 Examples of type  $\mathbb{D}$}
\subsection*{3.3.1 Type  $\mathbb{D}_4$}

\paragraph{\indent According to \cite[Theorem 1]{ONFR}, there are 50 basic 2-term silting complexes, 20 of which are tilting modules. Up to isomorphism there are four quivers of type $\mathbb{D}_4$. Due to Lemma \ref{s}, we classify silted algebras for two of them.}

\begin{example}\label{D4}
Let $A$ be the path algebra of the quiver
$$
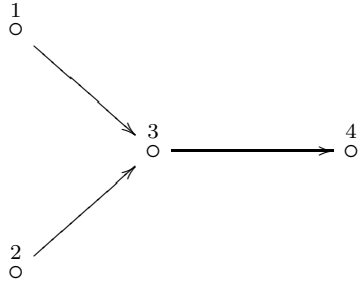
\xymatrix@C=15pt{\stackrel{1}\circ\ar[rrd]&&\\
&&\stackrel{3}\circ\ar[rrr]&&&
\stackrel{4}\circ\\
\stackrel{2}\circ\ar[rru]
}$$
 The AR-quiver of $\text{K}^{[-1.0]}(\text{proj}A)$  is
$$
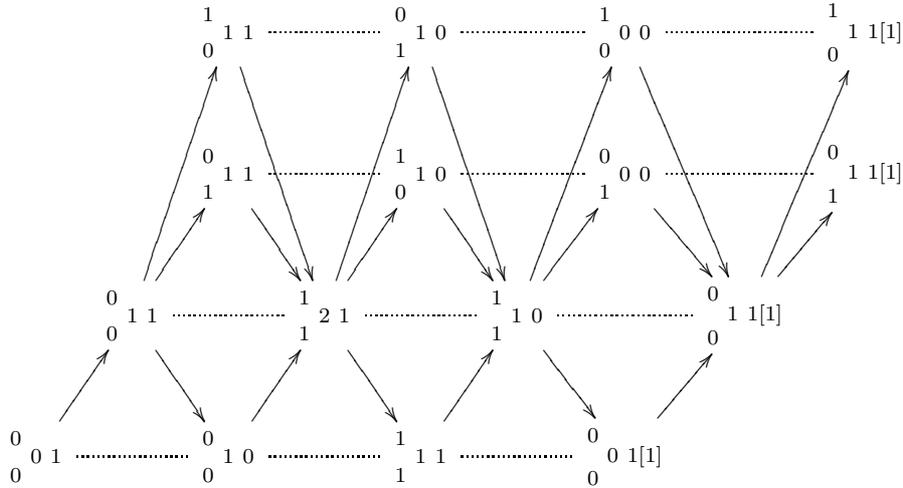
\xymatrix@C=5pt{&&{\begin{smallmatrix} 1&&\\&1&1\\0&& \end{smallmatrix}}\ar@{..}[rr]\ar[rdd]&&
{\begin{smallmatrix} 0&&\\&1&0\\1&& \end{smallmatrix}}\ar@{..}[rr]\ar[rdd]&&
{\begin{smallmatrix} 1&&\\&0&0\\0&& \end{smallmatrix}}\ar@{..}[rr]\ar[rdd]&&
{\begin{smallmatrix} 1&&\\&1&1[1]\\0&& \end{smallmatrix}}\\
&&{\begin{smallmatrix} 0&&\\&1&1\\1&& \end{smallmatrix}}\ar@{..}[rr]\ar[rd]&&
{\begin{smallmatrix} 1&&\\&1&0\\0&& \end{smallmatrix}}\ar@{..}[rr]\ar[rd]&&
{\begin{smallmatrix} 0&&\\&0&0\\1&& \end{smallmatrix}}\ar@{..}[rr]\ar[rd]&&
{\begin{smallmatrix} 0&&\\&1&1[1]\\1&& \end{smallmatrix}}\\
&{\begin{smallmatrix} 0&&\\&1&1\\0&& \end{smallmatrix}}\ar@{..}[rr]\ar[ru]\ar[ruu]\ar[rd]&&
{\begin{smallmatrix} 1&&\\&2&1\\1&& \end{smallmatrix}}\ar@{..}[rr]\ar[ru]\ar[ruu]\ar[rd]&&
{\begin{smallmatrix} 1&&\\&1&0\\1&& \end{smallmatrix}}\ar@{..}[rr]\ar[ru]\ar[ruu]\ar[rd]&&
{\begin{smallmatrix} 0&&\\&1&1[1]\\0&& \end{smallmatrix}}\ar[ru]\ar[ruu]\\
{\begin{smallmatrix} 0&&\\&0&1\\0&& \end{smallmatrix}}\ar@{..}[rr]\ar[ru]&&
{\begin{smallmatrix} 0&&\\&1&0\\0&& \end{smallmatrix}}\ar@{..}[rr]\ar[ru]&&
{\begin{smallmatrix} 1&&\\&1&1\\1&& \end{smallmatrix}}\ar@{..}[rr]\ar[ru]&&
{\begin{smallmatrix} 0&&\\&0&1[1]\\0&& \end{smallmatrix}}\ar[ru]
}$$

%(I) The  AR-quiver $\Gamma (\text{modA})$ of $\text{mod}A$ is
%$$\xymatrix@C=15pt{&&{\begin{smallmatrix} 1&&\\&1&1\\0&& \end{smallmatrix}}\ar@{..}[rr]\ar[rd]&&
%{\begin{smallmatrix} 0&&\\&1&0\\1&& \end{smallmatrix}}\ar@{..}[rr]\ar[rd]&&
%{\begin{smallmatrix} 1&&\\&0&0\\0&& \end{smallmatrix}}\\
%&{\begin{smallmatrix} 0&&\\&1&1\\0&& \end{smallmatrix}}\ar[r]\ar[ru]\ar[rd]&
%{\begin{smallmatrix} 0&&\\&1&1\\1&& \end{smallmatrix}}\ar[r]&
%{\begin{smallmatrix} 1&&\\&2&1\\1&& \end{smallmatrix}}\ar[r]\ar[ru]\ar[rd]&
%{\begin{smallmatrix} 1&&\\&1&0\\0&& \end{smallmatrix}}\ar[r]&
%{\begin{smallmatrix} 1&&\\&1&0\\1&& \end{smallmatrix}}\ar[r]\ar[ru]&
%{\begin{smallmatrix} 0&&\\&0&0\\1&& \end{smallmatrix}}\\
%{\begin{smallmatrix} 0&&\\&0&1\\0&& \end{smallmatrix}}\ar@{..}[rr]\ar[ru]&&
%{\begin{smallmatrix} 0&&\\&1&0\\0&& \end{smallmatrix}}\ar@{..}[rr]\ar[ru]&&
%{\begin{smallmatrix} 1&&\\&1&1\\1&& \end{smallmatrix}}\ar[ru]
%}$$

(I)  Tilted algebras of type $Q$ are

\newcommand{\scirc}{{\scriptstyle \circ}}
\newcommand{\sbullet}{{\scriptstyle\bullet}}

\[\begin{tabular}{|c|c|c|}
 \hline
 No. &
tilted algebras & tilting modules\\
\hline
(1)&
\makecell[c]{
\xymatrix{\circ\ar[r]&\circ\ar[r]&\circ\\&\circ\ar[u]
}} & \makecell[c]{$\begin{array}{*{8}{c@{\hspace{3pt}}}}
&&\sbullet&&\scirc&&\scirc\\[-8pt]
&&\scirc&&\sbullet&&\scirc\\[-8pt]
&\scirc&&\sbullet&&\scirc\\[-8pt]
\scirc&&\scirc&&\sbullet
\end{array}
\hspace{2mm}
\begin{array}{*{8}{c@{\hspace{3pt}}}}
&&\scirc&&\sbullet&&\scirc\\[-8pt]
&&\sbullet&&\scirc&&\scirc\\[-8pt]
&\scirc&&\sbullet&&\scirc\\[-8pt]
\scirc&&\scirc&&\sbullet
\end{array}
\hspace{2mm}
\begin{array}{*{8}{c@{\hspace{3pt}}}}
&&\sbullet&&\scirc&&\scirc\\[-8pt]
&&\sbullet&&\scirc&&\scirc\\[-8pt]
&\sbullet&&\scirc&&\scirc\\[-8pt]
\sbullet&&\scirc&&\scirc
\end{array}
\hspace{2mm}
\begin{array}{*{8}{c@{\hspace{3pt}}}}
&&\scirc&&\sbullet&&\scirc\\[-8pt]
&&\scirc&&\sbullet&&\scirc\\[-8pt]
&\scirc&&\sbullet&&\scirc\\[-8pt]
\scirc&&\sbullet&&\scirc
\end{array}
\hspace{2mm}
\begin{array}{*{8}{c@{\hspace{3pt}}}}
&&\scirc&&\scirc&&\sbullet\\[-8pt]
&&\scirc&&\scirc&&\sbullet\\[-8pt]
&\scirc&&\scirc&&\sbullet\\[-8pt]
\scirc&&\scirc&&\sbullet
\end{array}$}\\
\hline
(2)&
\makecell[c]{\xymatrix{\circ\ar[r]&\circ\ar[r]\ar[d]&\circ\\&\circ
}} & \makecell[c]{$\begin{array}{*{8}{c@{\hspace{3pt}}}}
&&\sbullet&&\scirc&&\scirc\\[-8pt]
&&\scirc&&\sbullet&&\scirc\\[-8pt]
&\scirc&&\sbullet&&\scirc\\[-8pt]
\scirc&&\sbullet&&\scirc
\end{array}
\hspace{2mm}
\begin{array}{*{8}{c@{\hspace{3pt}}}}
&&\scirc&&\sbullet&&\scirc\\[-8pt]
&&\sbullet&&\scirc&&\scirc\\[-8pt]
&\scirc&&\sbullet&&\scirc\\[-8pt]
\scirc&&\sbullet&&\scirc
\end{array}
\hspace{2mm}
\begin{array}{*{8}{c@{\hspace{3pt}}}}
&&\sbullet&&\scirc&&\scirc\\[-8pt]
&&\sbullet&&\scirc&&\scirc\\[-8pt]
&\scirc&&\sbullet&&\scirc\\[-8pt]
\scirc&&\scirc&&\sbullet
\end{array}
\hspace{2mm}
\begin{array}{*{8}{c@{\hspace{3pt}}}}
&&\scirc&&\sbullet&&\scirc\\[-8pt]
&&\scirc&&\scirc&&\sbullet\\[-8pt]
&\scirc&&\scirc&&\sbullet\\[-8pt]
\scirc&&\scirc&&\sbullet
\end{array}
\hspace{2mm}
\begin{array}{*{8}{c@{\hspace{3pt}}}}
&&\scirc&&\scirc&&\sbullet\\[-8pt]
&&\scirc&&\sbullet&&\scirc\\[-8pt]
&\scirc&&\scirc&&\sbullet\\[-8pt]
\scirc&&\scirc&&\sbullet
\end{array}$}\\
\hline
%(3)&
%\makecell[c]{\xymatrix{\circ&\circ\ar[l]\ar[d]\ar[r]&\circ\\&\circ
%}} &
%\makecell[c]{$\begin{array}{*{8}{c@{\hspace{3pt}}}}
%&&\sbullet&&\scirc&&\scirc\\[-8pt]
%&&\sbullet&&\scirc&&\scirc\\[-8pt]
%&\scirc&&\sbullet&&\scirc\\[-8pt]
%\scirc&&\sbullet&&\scirc
%\end{array}
%\hspace{3mm}
%\begin{array}{*{8}{c@{\hspace{3pt}}}}
%&&\scirc&&\sbullet&&\scirc\\[-8pt]
%&&\scirc&&\sbullet&&\scirc\\[-8pt]
%&\scirc&&\scirc&&\sbullet\\[-8pt]
%\scirc&&\scirc&&\sbullet
%\end{array}$}\\
%\hline

\end{tabular}
\]

%\[\resizebox{\textwidth}{50mm}{
\[\begin{tabular}{|c|c|c|}
 \hline
 No. &
tilted algebras & tilting modules\\
\hline
(3)&
\makecell[c]{\xymatrix{\circ&\circ\ar[l]\ar[d]\ar[r]&\circ\\&\circ
}} &
\makecell[c]{$\begin{array}{*{8}{c@{\hspace{3pt}}}}
&&\sbullet&&\scirc&&\scirc\\[-8pt]
&&\sbullet&&\scirc&&\scirc\\[-8pt]
&\scirc&&\sbullet&&\scirc\\[-8pt]
\scirc&&\sbullet&&\scirc
\end{array}
\hspace{3mm}
\begin{array}{*{8}{c@{\hspace{3pt}}}}
&&\scirc&&\sbullet&&\scirc\\[-8pt]
&&\scirc&&\sbullet&&\scirc\\[-8pt]
&\scirc&&\scirc&&\sbullet\\[-8pt]
\scirc&&\scirc&&\sbullet
\end{array}$}\\
\hline
(4)& \makecell[c]{\xymatrix{\circ\ar[r]&\circ&\circ\ar[l]\\&\circ\ar[u]
}} &
\makecell[c]{$\begin{array}{*{8}{c@{\hspace{3pt}}}}
&&\sbullet&&\scirc&&\scirc\\[-8pt]
&&\sbullet&&\scirc&&\scirc\\[-8pt]
&\sbullet&&\scirc&&\scirc\\[-8pt]
\scirc&&\sbullet&&\scirc
\end{array}
\hspace{3mm}
\begin{array}{*{8}{c@{\hspace{3pt}}}}
&&\scirc&&\sbullet&&\scirc\\[-8pt]
&&\scirc&&\sbullet&&\scirc\\[-8pt]
&\scirc&&\sbullet&&\scirc\\[-8pt]
\scirc&&\scirc&&\sbullet
\end{array}$}\\
\hline
(5)&
\makecell[c]{\xymatrix{\circ\ar[r]&\circ\\\circ\ar[u]
\ar[r]\ar@//@{.}[ru]_{}&\circ\ar[u]}}&
\makecell[c]{$\begin{array}{*{8}{c@{\hspace{3pt}}}}
&&\sbullet&&\scirc&&\sbullet\\[-8pt]
&&\scirc&&\sbullet&&\scirc\\[-8pt]
&\scirc&&\scirc&&\scirc\\[-8pt]
\scirc&&\scirc&&\sbullet
\end{array}
\hspace{3mm}
\begin{array}{*{8}{c@{\hspace{3pt}}}}
&&\scirc&&\sbullet&&\scirc\\[-8pt]
&&\sbullet&&\scirc&&\sbullet\\[-8pt]
&\scirc&&\scirc&&\scirc\\[-8pt]
\scirc&&\scirc&&\sbullet
\end{array}
\hspace{3mm}
\begin{array}{*{8}{c@{\hspace{3pt}}}}
&&\sbullet&&\scirc&&\scirc\\[-8pt]
&&\sbullet&&\scirc&&\scirc\\[-8pt]
&\scirc&&\scirc&&\scirc\\[-8pt]
\sbullet&&\scirc&&\sbullet
\end{array}$}\\
\hline
(6)&
\makecell[c]{\xymatrix{\circ\ar[r]\ar@/^/@{.}[rr]_{}&\circ\ar[r]&\circ\\
&\circ\ar[u]\ar@/_/@{.}[ru]_{}
}}& \makecell[c]{$\begin{array}{*{8}{c@{\hspace{3pt}}}}
&&\scirc&&\scirc&&\sbullet\\[-8pt]
&&\scirc&&\scirc&&\sbullet\\[-8pt]
&\scirc&&\scirc&&\scirc\\[-8pt]
\sbullet&&\scirc&&\sbullet
\end{array}$}\\
\hline
(7)&
\xymatrix{\circ\ar[r]\ar@/^/@{.}[rrr]_{}
&\circ\ar[r]&\circ\ar[r]&\circ
} & $\begin{array}{*{8}{c@{\hspace{3pt}}}}
&&\sbullet&&\scirc&&\sbullet\\[-8pt]
&&\scirc&&\scirc&&\scirc\\[-8pt]
&\scirc&&\scirc&&\scirc\\[-8pt]
\sbullet&&\scirc&&\sbullet
\end{array}
\hspace{3mm}
\begin{array}{*{8}{c@{\hspace{3pt}}}}
&&\scirc&&\scirc&&\scirc\\[-8pt]
&&\sbullet&&\scirc&&\sbullet\\[-8pt]
&\scirc&&\scirc&&\scirc\\[-8pt]
\sbullet&&\scirc&&\sbullet
\end{array}$\\
\hline
\end{tabular}
\]
%}\]

%\[\begin{tabular}{|c|c|c|}
% \hline
% No. &
%tilted algebras & tilting modules\\
%\hline
%(5)&
%\makecell[c]{\xymatrix{\circ\ar[r]&\circ\\\circ\ar[u]
%\ar[r]\ar@//@{.}[ru]_{}&\circ\ar[u]}}&
%\makecell[c]{$\begin{array}{*{8}{c@{\hspace{3pt}}}}
%&&\sbullet&&\scirc&&\sbullet\\[-8pt]
%&&\scirc&&\sbullet&&\scirc\\[-8pt]
%&\scirc&&\scirc&&\scirc\\[-8pt]
%\scirc&&\scirc&&\sbullet
%\end{array}
%\hspace{3mm}
%\begin{array}{*{8}{c@{\hspace{3pt}}}}
%&&\scirc&&\sbullet&&\scirc\\[-8pt]
%&&\sbullet&&\scirc&&\sbullet\\[-8pt]
%&\scirc&&\scirc&&\scirc\\[-8pt]
%\scirc&&\scirc&&\sbullet
%\end{array}
%\hspace{3mm}
%\begin{array}{*{8}{c@{\hspace{3pt}}}}
%&&\sbullet&&\scirc&&\scirc\\[-8pt]
%&&\sbullet&&\scirc&&\scirc\\[-8pt]
%&\scirc&&\scirc&&\scirc\\[-8pt]
%\sbullet&&\scirc&&\sbullet
%\end{array}$}\\
%\hline
%(6)&
%\makecell[c]{\xymatrix{\circ\ar[r]\ar@/^/@{.}[rr]_{}&\circ\ar[r]&\circ\\
%&\circ\ar[u]\ar@/_/@{.}[ru]_{}
%}}& \makecell[c]{$\begin{array}{*{8}{c@{\hspace{3pt}}}}
%&&\scirc&&\scirc&&\sbullet\\[-8pt]
%&&\scirc&&\scirc&&\sbullet\\[-8pt]
%&\scirc&&\scirc&&\scirc\\[-8pt]
%\sbullet&&\scirc&&\sbullet
%\end{array}$}\\
%\hline
%(7)&
%\xymatrix{\circ\ar[r]\ar@/^/@{.}[rrr]_{}
%&\circ\ar[r]&\circ\ar[r]&\circ
%} & $\begin{array}{*{8}{c@{\hspace{3pt}}}}
%&&\sbullet&&\scirc&&\sbullet\\[-8pt]
%&&\scirc&&\scirc&&\scirc\\[-8pt]
%&\scirc&&\scirc&&\scirc\\[-8pt]
%\sbullet&&\scirc&&\sbullet
%\end{array}
%\hspace{3mm}
%\begin{array}{*{8}{c@{\hspace{3pt}}}}
%&&\scirc&&\scirc&&\scirc\\[-8pt]
%&&\sbullet&&\scirc&&\sbullet\\[-8pt]
%&\scirc&&\scirc&&\scirc\\[-8pt]
%\sbullet&&\scirc&&\sbullet
%\end{array}$\\
%\hline
%\end{tabular}
%\]

\smallskip
(II) Silted algebras of the form $\text{End}_{\text{K}^b(\text{proj}A)}(M\oplus P[1])$ ($P\neq 0,~M\neq 0$) are

\[\begin{tabular}{|c|c|c|c|}
 \hline
No. &
silted algebras & 2-term silting complexes & tilted type\\
\hline
(8)&
\xymatrix{\circ\ar[r]&\circ&\circ\ar[l]~~~\circ} & $\begin{array}{*{10}{c@{\hspace{1pt}}}}
&&\scirc&&\scirc&&\scirc&&\sbullet\\[-10pt]
&&\sbullet&&\scirc&&\scirc&&\scirc\\[-10pt]
&\sbullet&&\scirc&&\scirc&&\scirc\\[-10pt]
\scirc&&\sbullet&&\scirc&&\scirc
\end{array}
\hspace{2mm}
\begin{array}{*{10}{c@{\hspace{1pt}}}}
&&\sbullet&&\scirc&&\scirc&&\scirc\\[-10pt]
&&\scirc&&\scirc&&\scirc&&\sbullet\\[-10pt]
&\sbullet&&\scirc&&\scirc&&\scirc\\[-10pt]
\scirc&&\sbullet&&\scirc&&\scirc
\end{array}
\hspace{2mm}
\begin{array}{*{10}{c@{\hspace{1pt}}}}
&&\scirc&&\scirc&&\scirc&&\sbullet\\[-10pt]
&&\scirc&&\scirc&&\scirc&&\sbullet\\[-10pt]
&\scirc&&\scirc&&\scirc&&\sbullet\\[-10pt]
\sbullet&&\scirc&&\scirc&&\scirc
\end{array}
$ & $\mathbb{A}_3\amalg\mathbb{A}_1$ \\
\hline
(9)&
\xymatrix{\circ\ar[r]&\circ\ar[r]&\circ~~~\circ} & $
\begin{array}{*{10}{c@{\hspace{1pt}}}}
&&\scirc&&\scirc&&\scirc&&\sbullet\\[-10pt]
&&\sbullet&&\scirc&&\scirc&&\scirc\\[-10pt]
&\sbullet&&\scirc&&\scirc&&\scirc\\[-10pt]
\sbullet&&\scirc&&\scirc&&\scirc
\end{array}
\hspace{2mm}
\begin{array}{*{10}{c@{\hspace{1pt}}}}
&&\sbullet&&\scirc&&\scirc&&\scirc\\[-10pt]
&&\scirc&&\scirc&&\scirc&&\sbullet\\[-10pt]
&\sbullet&&\scirc&&\scirc&&\scirc\\[-10pt]
\sbullet&&\scirc&&\scirc&&\scirc
\end{array}
\hspace{2mm}
\begin{array}{*{10}{c@{\hspace{1pt}}}}
&&\scirc&&\scirc&&\scirc&&\sbullet\\[-10pt]
&&\scirc&&\scirc&&\sbullet&&\scirc\\[-10pt]
&\scirc&&\scirc&&\scirc&&\sbullet\\[-10pt]
\sbullet&&\scirc&&\scirc&&\scirc
\end{array}
\hspace{2mm}
\begin{array}{*{10}{c@{\hspace{1pt}}}}
&&\scirc&&\scirc&&\sbullet&&\scirc\\[-10pt]
&&\scirc&&\scirc&&\scirc&&\sbullet\\[-10pt]
&\scirc&&\scirc&&\scirc&&\sbullet\\[-10pt]
\sbullet&&\scirc&&\scirc&&\scirc
\end{array} $ & $\mathbb{A}_3\amalg\mathbb{A}_1$ \\
\hline
(10)&
\xymatrix{\circ&\circ\ar[l]\ar[r]&\circ~~~\circ} & $
\begin{array}{*{10}{c@{\hspace{1pt}}}}
&&\scirc&&\scirc&&\sbullet&&\scirc\\[-10pt]
&&\scirc&&\scirc&&\sbullet&&\scirc\\[-10pt]
&\scirc&&\scirc&&\scirc&&\sbullet\\[-10pt]
\sbullet&&\scirc&&\scirc&&\scirc
\end{array}$ & $\mathbb{A}_3\amalg\mathbb{A}_1$ \\
\hline

(11)&
\xymatrix{\circ\ar[r]&\circ~~~\circ~~~\circ}&
$\begin{array}{*{10}{c@{\hspace{1pt}}}}
&&\scirc&&\scirc&&\scirc&&\sbullet\\[-10pt]
&&\scirc&&\scirc&&\scirc&&\sbullet\\[-10pt]
&\sbullet&&\scirc&&\scirc&&\scirc\\[-10pt]
\scirc&&\sbullet&&\scirc&&\scirc
\end{array}
\hspace{3mm}
\begin{array}{*{10}{c@{\hspace{1pt}}}}
&&\scirc&&\scirc&&\scirc&&\sbullet\\[-10pt]
&&\scirc&&\scirc&&\scirc&&\sbullet\\[-10pt]
&\sbullet&&\scirc&&\scirc&&\scirc\\[-10pt]
\sbullet&&\scirc&&\scirc&&\scirc
\end{array}$ &  $\mathbb{A}_2\amalg\mathbb{A}_1\amalg\mathbb{A}_1$ \\
\hline

(7)&
\xymatrix{\circ\ar[r]\ar@/^/@{.}[rrr]_{}
&\circ\ar[r]&\circ\ar[r]&\circ} & $\begin{array}{*{10}{c@{\hspace{1pt}}}}
&&\scirc&&\sbullet&&\scirc&&\sbullet\\[-10pt]
&&\sbullet&&\scirc&&\sbullet&&\scirc\\[-10pt]
&\scirc&&\scirc&&\scirc&&\scirc\\[-10pt]
\scirc&&\scirc&&\scirc&&\scirc
\end{array}
\hspace{3mm}
\begin{array}{*{10}{c@{\hspace{1pt}}}}
&&\sbullet&&\scirc&&\sbullet&&\scirc\\[-10pt]
&&\scirc&&\sbullet&&\scirc&&\sbullet\\[-10pt]
&\scirc&&\scirc&&\scirc&&\scirc\\[-10pt]
\scirc&&\scirc&&\scirc&&\scirc
\end{array}
 $ & $\mathbb{D}_4$ \\
\hline
(12)&
\makecell[c]{\xymatrix{\circ\ar[r]\ar@/^/@{.}[rr]_{}
\ar@/_/@{.}[rd]_{}&\circ\ar[r]\ar[d]&\circ\\
&\circ}}&
\makecell[c]{$\begin{array}{*{10}{c@{\hspace{1pt}}}}
&&\scirc&&\sbullet&&\scirc&&\sbullet\\[-10pt]
&&\sbullet&&\scirc&&\scirc&&\scirc\\[-10pt]
&\scirc&&\scirc&&\scirc&&\scirc\\[-10pt]
\scirc&&\sbullet&&\scirc&&\scirc
\end{array}
\hspace{3mm}
\begin{array}{*{10}{c@{\hspace{1pt}}}}
&&\sbullet&&\scirc&&\scirc&&\scirc\\[-10pt]
&&\scirc&&\sbullet&&\scirc&&\sbullet\\[-10pt]
&\scirc&&\scirc&&\scirc&&\scirc\\[-10pt]
\scirc&&\sbullet&&\scirc&&\scirc
\end{array}$} &  $\mathbb{D}_4$ \\
\hline
(13)&
\xymatrix{\circ\ar[r]\ar@/^/@{.}[rr]_{}
&\circ\ar[r]\ar@/^/@{.}[rr]_{}&\circ\ar[r]&\circ
} & $
\begin{array}{*{10}{c@{\hspace{1pt}}}}
&&\scirc&&\scirc&&\scirc&&\sbullet\\[-10pt]
&&\sbullet&&\scirc&&\sbullet&&\scirc\\[-10pt]
&\scirc&&\scirc&&\scirc&&\scirc\\[-10pt]
\sbullet&&\scirc&&\scirc&&\scirc
\end{array}
\hspace{3mm}
\begin{array}{*{10}{c@{\hspace{1pt}}}}
&&\sbullet&&\scirc&&\sbullet&&\scirc\\[-10pt]
&&\scirc&&\scirc&&\scirc&&\sbullet\\[-10pt]
&\scirc&&\scirc&&\scirc&&\scirc\\[-10pt]
\sbullet&&\scirc&&\scirc&&\scirc
\end{array}$ & strictly shod  \\
\hline
\end{tabular}
\]

%\[\begin{tabular}{|c|c|c|c|}
% \hline
%No. &
%silted algebras & 2-term silting complexes & tilted type\\
%\hline
%(12)&
%\makecell[c]{\xymatrix{\circ\ar[r]\ar@/^/@{.}[rr]_{}
%\ar@/_/@{.}[rd]_{}&\circ\ar[r]\ar[d]&\circ\\
%&\circ}}&
%\makecell[c]{$\begin{array}{*{10}{c@{\hspace{1pt}}}}
%&&\scirc&&\sbullet&&\scirc&&\sbullet\\[-10pt]
%&&\sbullet&&\scirc&&\scirc&&\scirc\\[-10pt]
%&\scirc&&\scirc&&\scirc&&\scirc\\[-10pt]
%\scirc&&\sbullet&&\scirc&&\scirc
%\end{array}
%\hspace{3mm}
%\begin{array}{*{10}{c@{\hspace{1pt}}}}
%&&\sbullet&&\scirc&&\scirc&&\scirc\\[-10pt]
%&&\scirc&&\sbullet&&\scirc&&\sbullet\\[-10pt]
%&\scirc&&\scirc&&\scirc&&\scirc\\[-10pt]
%\scirc&&\sbullet&&\scirc&&\scirc
%\end{array}$} &  $\mathbb{D}_4$ \\
%\hline
%(13)&
%\xymatrix{\circ\ar[r]\ar@/^/@{.}[rr]_{}
%&\circ\ar[r]\ar@/^/@{.}[rr]_{}&\circ\ar[r]&\circ
%} & $
%\begin{array}{*{10}{c@{\hspace{1pt}}}}
%&&\scirc&&\scirc&&\scirc&&\sbullet\\[-10pt]
%&&\sbullet&&\scirc&&\sbullet&&\scirc\\[-10pt]
%&\scirc&&\scirc&&\scirc&&\scirc\\[-10pt]
%\sbullet&&\scirc&&\scirc&&\scirc
%\end{array}
%\hspace{3mm}
%\begin{array}{*{10}{c@{\hspace{1pt}}}}
%&&\sbullet&&\scirc&&\sbullet&&\scirc\\[-10pt]
%&&\scirc&&\scirc&&\scirc&&\sbullet\\[-10pt]
%&\scirc&&\scirc&&\scirc&&\scirc\\[-10pt]
%\sbullet&&\scirc&&\scirc&&\scirc
%\end{array}$ & strictly shod  \\
%\hline
%\end{tabular}
%\]
%}\]

To summarise, there are 13 silted algebras  of type $Q$, forming 4 families:
%To summarize, let $B$ be one of  the above silted algebras. Then $B$ belongs to one of the following families:
\begin{itemize}
\item[(\romannumeral1)] tilted algebras of type $\mathbb{D}_4$: $(1)-(7)$, $(12)$;
\item[(\romannumeral2)] tilted algebras of type $\mathbb{A}_3\coprod\mathbb{A}_1$: $(8)-(10)$;
\item[(\romannumeral3)] tilted algebras of type $\mathbb{A}_2\coprod\mathbb{A}_1\coprod\mathbb{A}_1$: $(11)$;
\item[(\romannumeral4)] the strictly shod algebra
$\xymatrix@C=5pt{\circ\ar[rr]\ar@/^/@{.}[rrrr]_{}
&&\circ\ar[rr]\ar@/^/@{.}[rrrr]_{}&&\circ\ar[rr]&&\circ
}$: $(13)$.

\end{itemize}

\end{example}

\begin{example}
Let $A$ be the path algebra of the quiver
$$
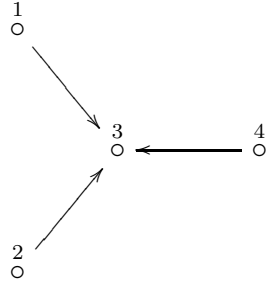
\xymatrix@C=8pt{\stackrel{1}\circ\ar[rrd]&&\\
&&\stackrel{3}\circ&&&
\stackrel{4}\circ\ar[lll]\\
\stackrel{2}\circ\ar[rru]
}$$
The AR-quiver of  $\text{K}^{[-1.0]}(\text{proj}A)$  is

$$
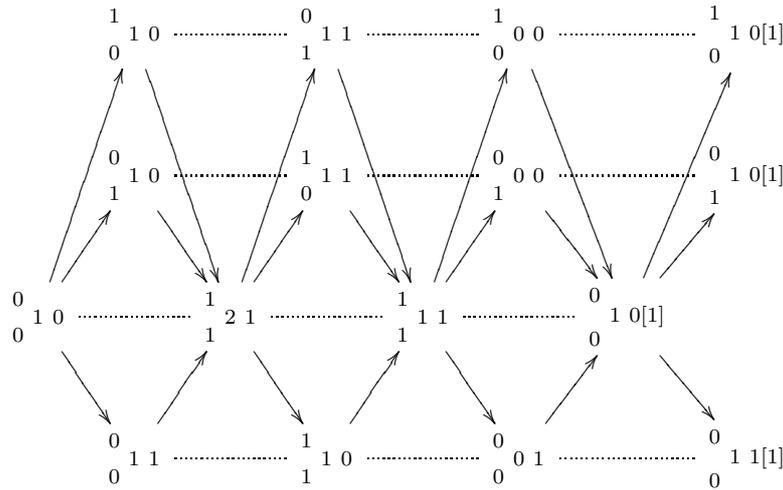
\xymatrix@C=5pt{&{\begin{smallmatrix} 1&&\\&1&0\\0&& \end{smallmatrix}}\ar@{..}[rr]\ar[rdd]&&
{\begin{smallmatrix} 0&&\\&1&1\\1&& \end{smallmatrix}}\ar@{..}[rr]\ar[rdd]&&
{\begin{smallmatrix} 1&&\\&0&0\\0&& \end{smallmatrix}}\ar@{..}[rr]\ar[rdd]&&
{\begin{smallmatrix} 1&&\\&1&0[1]\\0&& \end{smallmatrix}}\\
&{\begin{smallmatrix} 0&&\\&1&0\\1&& \end{smallmatrix}}\ar@{..}[rr]\ar[rd]&&
{\begin{smallmatrix} 1&&\\&1&1\\0&& \end{smallmatrix}}\ar@{..}[rr]\ar[rd]&&
{\begin{smallmatrix} 0&&\\&0&0\\1&& \end{smallmatrix}}\ar@{..}[rr]\ar[rd]&&
{\begin{smallmatrix} 0&&\\&1&0[1]\\1&& \end{smallmatrix}}\\
{\begin{smallmatrix} 0&&\\&1&0\\0&& \end{smallmatrix}}\ar@{..}[rr]\ar[ru]\ar[ruu]\ar[rd]&&
{\begin{smallmatrix} 1&&\\&2&1\\1&& \end{smallmatrix}}\ar@{..}[rr]\ar[ru]\ar[ruu]\ar[rd]&&
{\begin{smallmatrix} 1&&\\&1&1\\1&& \end{smallmatrix}}\ar@{..}[rr]\ar[ru]\ar[ruu]\ar[rd]&&
{\begin{smallmatrix} 0&&\\&1&0[1]\\0&& \end{smallmatrix}}\ar[ru]\ar[ruu]\ar[rd]\\
&{\begin{smallmatrix} 0&&\\&1&1\\0&& \end{smallmatrix}}\ar@{..}[rr]\ar[ru]&&
{\begin{smallmatrix} 1&&\\&1&0\\1&& \end{smallmatrix}}\ar@{..}[rr]\ar[ru]&&
{\begin{smallmatrix} 0&&\\&0&1\\0&& \end{smallmatrix}}\ar@{..}[rr]\ar[ru]&&
{\begin{smallmatrix} 0&&\\&1&1[1]\\0&& \end{smallmatrix}}
}$$

(I)  Tilted algebras of type $Q$ are
%We first apply Algorithm\ref{al1} to produce all tilting modules and compute the corresponding tilted algebras. There are the tilted algebras of type $Q$.
\[\begin{tabular}{|c|c|c|}
 \hline
No. & tilted algebras & tilting modules\\
\hline
(1)&\makecell[c]{\xymatrix{\circ\ar[r]&\circ\ar[r]&\circ\\
&\circ\ar[u]}}&
\newcommand{\scirc}{{\scriptstyle \circ}}
\newcommand{\sbullet}{{\scriptstyle\bullet}}
$\begin{array}{*{10}{c@{\hspace{1pt}}}}
&\sbullet&&\scirc&&\scirc\\[-8pt]
&\scirc&&\sbullet&&\scirc\\[-8pt]
\scirc&&\sbullet&&\scirc\\[-8pt]
&\scirc&&\sbullet&&\scirc\\[-5pt]
\end{array}  \hspace{3mm}   \begin{array}{*{10}{c@{\hspace{1pt}}}}
&\scirc&&\sbullet&&\scirc\\[-8pt]
&\sbullet&&\scirc&&\scirc\\[-8pt]
\scirc&&\sbullet&&\scirc\\[-8pt]
&\scirc&&\sbullet&&\scirc\\[-8pt]
\end{array} \hspace{3mm}  \begin{array}{*{10}{c@{\hspace{1pt}}}}
&\scirc&&\sbullet&&\scirc\\[-8pt]
&\scirc&&\sbullet&&\scirc\\[-8pt]
\scirc&&\sbullet&&\scirc\\[-8pt]
&\sbullet&&\scirc&&\scirc\\[-8pt]
\end{array} \hspace{3mm}\begin{array}{*{10}{c@{\hspace{2pt}}}}
&\scirc&&\sbullet&&\scirc\\[-8pt]
&\scirc&&\scirc&&\sbullet\\[-8pt]
\scirc&&\scirc&&\sbullet\\[-8pt]
&\scirc&&\scirc&&\sbullet\\[-5pt]
\end{array}  \hspace{3mm}   \begin{array}{*{10}{c@{\hspace{1pt}}}}
&\scirc&&\scirc&&\sbullet\\[-8pt]
&\scirc&&\sbullet&&\scirc\\[-8pt]
\scirc&&\scirc&&\sbullet\\[-8pt]
&\scirc&&\scirc&&\sbullet\\[-8pt]
\end{array} \hspace{3mm}  \begin{array}{*{10}{c@{\hspace{1pt}}}}
&\scirc&&\scirc&&\sbullet\\[-8pt]
&\scirc&&\scirc&&\sbullet\\[-8pt]
\scirc&&\scirc&&\sbullet\\[-8pt]
&\scirc&&\sbullet&&\scirc\\[-8pt]
\end{array}$\\
\hline
(2)&\makecell[c]{\xymatrix{\circ\ar[r]&\circ\ar[r]\ar[d]&\circ\\
&\circ}}&
\newcommand{\scirc}{{\scriptstyle \circ}}
\newcommand{\sbullet}{{\scriptstyle\bullet}}
$\begin{array}{*{10}{c@{\hspace{1pt}}}}
&\sbullet&&\scirc&&\scirc\\[-8pt]
&\sbullet&&\scirc&&\scirc\\[-8pt]
\scirc&&\sbullet&&\scirc\\[-8pt]
&\scirc&&\sbullet&&\scirc\\[-8pt]
\end{array} \hspace{3mm} \begin{array}{*{10}{c@{\hspace{1pt}}}}
&\sbullet&&\scirc&&\scirc\\[-8pt]
&\scirc&&\sbullet&&\scirc\\[-8pt]
\scirc&&\sbullet&&\scirc\\[-8pt]
&\sbullet&&\scirc&&\scirc\\[-8pt]
\end{array}  \hspace{3mm} \begin{array}{*{10}{c@{\hspace{1pt}}}}
&\scirc&&\sbullet&&\scirc\\[-8pt]
&\sbullet&&\scirc&&\scirc\\[-8pt]
\scirc&&\sbullet&&\scirc\\[-8pt]
&\sbullet&&\scirc&&\scirc\\[-8pt]
\end{array}\hspace{3mm}\begin{array}{*{10}{c@{\hspace{1pt}}}}
&\scirc&&\sbullet&&\scirc\\[-8pt]
&\scirc&&\sbullet&&\scirc\\[-8pt]
\scirc&&\scirc&&\sbullet\\[-8pt]
&\scirc&&\scirc&&\sbullet\\[-8pt]
\end{array} \hspace{3mm} \begin{array}{*{10}{c@{\hspace{1pt}}}}
&\scirc&&\sbullet&&\scirc\\[-8pt]
&\scirc&&\scirc&&\sbullet\\[-8pt]
\scirc&&\scirc&&\sbullet\\[-8pt]
&\scirc&&\sbullet&&\scirc\\[-8pt]
\end{array}  \hspace{3mm} \begin{array}{*{10}{c@{\hspace{1pt}}}}
&\scirc&&\scirc&&\sbullet\\[-8pt]
&\scirc&&\sbullet&&\scirc\\[-8pt]
\scirc&&\scirc&&\sbullet\\[-8pt]
&\scirc&&\sbullet&&\scirc\\[-8pt]
\end{array}$\\
\hline
(3)&\makecell[c]{\xymatrix{\circ&\circ\ar[l]\ar[d]\ar[r]
&\circ\\&\circ}}&
\newcommand{\scirc}{{\scriptstyle \circ}}
\newcommand{\sbullet}{{\scriptstyle\bullet}}
$\begin{array}{*{10}{c@{\hspace{1pt}}}}
&\sbullet&&\scirc&&\scirc\\[-8pt]
&\sbullet&&\scirc&&\scirc\\[-8pt]
\scirc&&\sbullet&&\scirc\\[-8pt]
&\sbullet&&\scirc&&\scirc\\[-5pt]
\end{array}\hspace{5mm}\begin{array}{*{10}{c@{\hspace{1pt}}}}
&\scirc&&\sbullet&&\scirc\\[-8pt]
&\scirc&&\sbullet&&\scirc\\[-8pt]
\scirc&&\scirc&&\sbullet\\[-8pt]
&\scirc&&\sbullet&&\scirc\\[-8pt]
\end{array}$\\
\hline
%(4)&\makecell[c]{\xymatrix{\circ\ar[r]&\circ&\circ\ar[l]\\
%&\circ\ar[u]}}&
%\newcommand{\scirc}{{\scriptstyle \circ}}
%\newcommand{\sbullet}{{\scriptstyle\bullet}}
%$\begin{array}{*{10}{c@{\hspace{1pt}}}}
%&\sbullet&&\scirc&&\scirc\\[-8pt]
%&\sbullet&&\scirc&&\scirc\\[-8pt]
%\sbullet&&\scirc&&\scirc\\[-8pt]
%&\sbullet&&\scirc&&\scirc\\[-5pt]
%\end{array}\hspace{5mm}\begin{array}{*{10}{c@{\hspace{1pt}}}}
%&\scirc&&\sbullet&&\scirc\\[-8pt]
%&\scirc&&\sbullet&&\scirc\\[-8pt]
%\scirc&&\sbullet&&\scirc\\[-8pt]
%&\scirc&&\sbullet&&\scirc\\[-5pt]
%\end{array}\hspace{5mm}\begin{array}{*{10}{c@{\hspace{1pt}}}}
%&\scirc&&\scirc&&\sbullet\\[-8pt]
%&\scirc&&\scirc&&\sbullet\\[-8pt]
%\scirc&&\scirc&&\sbullet\\[-8pt]
%&\scirc&&\scirc&&\sbullet\\[-8pt]
%\end{array}$\\
%\hline

\end{tabular}
\]

\[\begin{tabular}{|c|c|c|}
 \hline
No. & tilted algebras & tilting modules\\
\hline

%(3)&\makecell[c]{\xymatrix{\circ&\circ\ar[l]\ar[d]\ar[r]
%&\circ\\&\circ}}&
%\newcommand{\scirc}{{\scriptstyle \circ}}
%\newcommand{\sbullet}{{\scriptstyle\bullet}}
%$\begin{array}{*{10}{c@{\hspace{1pt}}}}
%&\sbullet&&\scirc&&\scirc\\[-8pt]
%&\sbullet&&\scirc&&\scirc\\[-8pt]
%\scirc&&\sbullet&&\scirc\\[-8pt]
%&\sbullet&&\scirc&&\scirc\\[-5pt]
%\end{array}\hspace{5mm}\begin{array}{*{10}{c@{\hspace{1pt}}}}
%&\scirc&&\sbullet&&\scirc\\[-8pt]
%&\scirc&&\sbullet&&\scirc\\[-8pt]
%\scirc&&\scirc&&\sbullet\\[-8pt]
%&\scirc&&\sbullet&&\scirc\\[-8pt]
%\end{array}$\\
%\hline
%(4)&\makecell[c]{\xymatrix{\circ\ar[r]&\circ&\circ\ar[l]\\
%&\circ\ar[u]}}&
%\newcommand{\scirc}{{\scriptstyle \circ}}
%\newcommand{\sbullet}{{\scriptstyle\bullet}}
%$\begin{array}{*{10}{c@{\hspace{1pt}}}}
%&\sbullet&&\scirc&&\scirc\\[-8pt]
%&\sbullet&&\scirc&&\scirc\\[-8pt]
%\sbullet&&\scirc&&\scirc\\[-8pt]
%&\sbullet&&\scirc&&\scirc\\[-5pt]
%\end{array}\hspace{5mm}\begin{array}{*{10}{c@{\hspace{1pt}}}}
%&\scirc&&\sbullet&&\scirc\\[-8pt]
%&\scirc&&\sbullet&&\scirc\\[-8pt]
%\scirc&&\sbullet&&\scirc\\[-8pt]
%&\scirc&&\sbullet&&\scirc\\[-5pt]
%\end{array}\hspace{5mm}\begin{array}{*{10}{c@{\hspace{1pt}}}}
%&\scirc&&\scirc&&\sbullet\\[-8pt]
%&\scirc&&\scirc&&\sbullet\\[-8pt]
%\scirc&&\scirc&&\sbullet\\[-8pt]
%&\scirc&&\scirc&&\sbullet\\[-8pt]
%\end{array}$\\
%\hline
(4)&\makecell[c]{\xymatrix{\circ\ar[r]&\circ&\circ\ar[l]\\
&\circ\ar[u]}}&
\newcommand{\scirc}{{\scriptstyle \circ}}
\newcommand{\sbullet}{{\scriptstyle\bullet}}
$\begin{array}{*{10}{c@{\hspace{1pt}}}}
&\sbullet&&\scirc&&\scirc\\[-8pt]
&\sbullet&&\scirc&&\scirc\\[-8pt]
\sbullet&&\scirc&&\scirc\\[-8pt]
&\sbullet&&\scirc&&\scirc\\[-5pt]
\end{array}\hspace{5mm}\begin{array}{*{10}{c@{\hspace{1pt}}}}
&\scirc&&\sbullet&&\scirc\\[-8pt]
&\scirc&&\sbullet&&\scirc\\[-8pt]
\scirc&&\sbullet&&\scirc\\[-8pt]
&\scirc&&\sbullet&&\scirc\\[-5pt]
\end{array}\hspace{5mm}\begin{array}{*{10}{c@{\hspace{1pt}}}}
&\scirc&&\scirc&&\sbullet\\[-8pt]
&\scirc&&\scirc&&\sbullet\\[-8pt]
\scirc&&\scirc&&\sbullet\\[-8pt]
&\scirc&&\scirc&&\sbullet\\[-8pt]
\end{array}$\\
\hline
(5)&\makecell[c]{\xymatrix{\circ\ar[r]
&\circ\\\circ\ar[u]\ar[r]\ar@//@{.}[ru]_{}&\circ\ar[u]}}&
\newcommand{\scirc}{{\scriptstyle \circ}}
\newcommand{\sbullet}{{\scriptstyle\bullet}}
$\begin{array}{*{10}{c@{\hspace{1pt}}}}
&\sbullet&&\scirc&&\sbullet\\[-8pt]
&\scirc&&\sbullet&&\scirc\\[-8pt]
\scirc&&\scirc&&\scirc\\[-8pt]
&\scirc&&\sbullet&&\scirc
\end{array} \hspace{5mm}   \begin{array}{*{10}{c@{\hspace{1pt}}}}
&\scirc&&\sbullet&&\scirc\\[-8pt]
&\sbullet&&\scirc&&\sbullet\\[-8pt]
\scirc&&\scirc&&\scirc\\[-8pt]
&\scirc&&\sbullet&&\scirc
\end{array}  \hspace{5mm}  \begin{array}{*{10}{c@{\hspace{1pt}}}}
&\scirc&&\sbullet&&\scirc\\[-8pt]
&\scirc&&\sbullet&&\scirc\\[-8pt]
\scirc&&\scirc&&\scirc\\[-8pt]
&\sbullet&&\scirc&&\sbullet\\[-8pt]
\end{array} $\\
\hline
\end{tabular}
\]

%(II) Silted algebras of type $Q$ which are not tilted of type $Q$ are
\smallskip
(II) Silted algebras of the form $\text{End}_{\text{K}^b(\text{proj}A)}(M\oplus P[1])$ ($P\neq 0,~M\neq 0$) are

\newcommand{\tabincell}[2]{\begin{tabular}{@{}#1@{}}#2\end{tabular}}
\begin{table}[htbp]
\begin{tabular}{|c|c|c|c|}
 \hline
No.&silted algebras & 2-term silting complexes &  tilted type\\
\hline

(6)&\makecell[c]{\xymatrix{\circ\ar[r]\ar@/^/@{.}[rr]_{}
\ar@/_/@{.}[rd]_{}&\circ\ar[d]\ar[r]&\circ\\
&\circ
}}&
\newcommand{\scirc}{{\scriptstyle \circ}}
\newcommand{\sbullet}{{\scriptstyle\bullet}}
$\begin{array}{*{12}{c@{\hspace{1pt}}}}
&\scirc&&\sbullet&&\scirc&&\sbullet\\[-6pt]
&\sbullet&&\scirc&&\scirc&&\scirc\\[-6pt]
\scirc&&\scirc&&\scirc&&\scirc\\[-6pt]
&\sbullet&&\scirc&&\scirc&&\scirc
\end{array}\hspace{3mm} \begin{array}{*{12}{c@{\hspace{1pt}}}}
&\sbullet&&\scirc&&\scirc&&\scirc\\[-6pt]
&\scirc&&\sbullet&&\scirc&&\sbullet\\[-6pt]
\scirc&&\scirc&&\scirc&&\scirc\\[-6pt]
&\sbullet&&\scirc&&\scirc&&\scirc
\end{array} \hspace{3mm} \begin{array}{*{12}{c@{\hspace{1pt}}}}
&\sbullet&&\scirc&&\scirc&&\scirc\\[-6pt]
&\sbullet&&\scirc&&\scirc&&\scirc\\[-6pt]
\scirc&&\scirc&&\scirc&&\scirc\\[-6pt]
&\scirc&&\sbullet&&\scirc&&\sbullet
\end{array}$& $\mathbb{D}_4$\\
\hline
(7)&\makecell[c]{\xymatrix{\circ\ar[r]\ar@/^/@{.}[rr]_{}
&\circ\ar[r]&\circ\\
&\circ\ar[u]\ar@/_/@{.}[ru]_{}
}}&\newcommand{\scirc}{{\scriptstyle \circ}}
\newcommand{\sbullet}{{\scriptstyle\bullet}}
$\begin{array}{*{12}{c@{\hspace{1pt}}}}
&\scirc&&\scirc&&\scirc&&\sbullet\\[-6pt]
&\scirc&&\scirc&&\scirc&&\sbullet\\[-6pt]
\scirc&&\scirc&&\scirc&&\scirc\\[-6pt]
&\sbullet&&\scirc&&\sbullet&&\scirc
\end{array}  \hspace{3mm} \begin{array}{*{12}{c@{\hspace{1pt}}}}
&\scirc&&\scirc&&\scirc&&\sbullet\\[-6pt]
&\sbullet&&\scirc&&\sbullet&&\scirc\\[-6pt]
\scirc&&\scirc&&\scirc&&\scirc\\[-6pt]
&\scirc&&\scirc&&\scirc&&\sbullet
\end{array} \hspace{3mm} \begin{array}{*{12}{c@{\hspace{1pt}}}}
&\sbullet&&\scirc&&\sbullet&&\scirc\\[-6pt]
&\scirc&&\scirc&&\scirc&&\sbullet\\[-6pt]
\scirc&&\scirc&&\scirc&&\scirc\\[-6pt]
&\scirc&&\scirc&&\scirc&&\sbullet
\end{array}$&$\mathbb{D}_4$\\
\hline
(8)&\xymatrix{\circ\ar[r]&\circ
&\circ\ar[l]~~~\circ}&
\newcommand{\scirc}{{\scriptstyle \circ}}
\newcommand{\sbullet}{{\scriptstyle\bullet}}
$\begin{array}{*{12}{c@{\hspace{1pt}}}}
&\scirc&&\scirc&&\scirc&&\sbullet\\[-6pt]
&\sbullet&&\scirc&&\scirc&&\scirc\\[-6pt]
\sbullet&&\scirc&&\scirc&&\scirc\\[-6pt]
&\sbullet&&\scirc&&\scirc&&\scirc
\end{array}  \hspace{3mm}  \begin{array}{*{12}{c@{\hspace{1pt}}}}
&\sbullet&&\scirc&&\scirc&&\scirc\\[-6pt]
&\scirc&&\scirc&&\scirc&&\sbullet\\[-6pt]
\sbullet&&\scirc&&\scirc&&\scirc\\[-6pt]
&\sbullet&&\scirc&&\scirc&&\scirc
\end{array} \hspace{3mm}  \begin{array}{*{12}{c@{\hspace{1pt}}}}
&\sbullet&&\scirc&&\scirc&&\scirc\\[-6pt]
&\sbullet&&\scirc&&\scirc&&\scirc\\[-6pt]
\sbullet&&\scirc&&\scirc&&\scirc\\[-6pt]
&\scirc&&\scirc&&\scirc&&\sbullet
\end{array}$& $\mathbb{A}_3\amalg\mathbb{A}_1$\\
\hline
%(9)&\xymatrix{\circ\ar[r]\ar@/^/@{.}[rrr]_{}
%&\circ\ar[r]&\circ\ar[r]&\circ
%}&
%\newcommand{\scirc}{{\scriptstyle \circ}}
%\newcommand{\sbullet}{{\scriptstyle\bullet}}
%$\begin{array}{*{10}{c@{\hspace{1pt}}}}
%&\scirc&&\sbullet&&\scirc&&\sbullet\\[-10pt]
%&\sbullet&&\scirc&&\sbullet&&\scirc\\[-10pt]
%\scirc&&\scirc&&\scirc&&\scirc\\[-10pt]
%&\scirc&&\scirc&&\scirc&&\scirc
%\end{array} \hspace{1mm} \begin{array}{*{10}{c@{\hspace{1pt}}}}
%&\scirc&&\sbullet&&\scirc&&\sbullet\\[-10pt]
%&\scirc&&\scirc&&\scirc&&\scirc\\[-10pt]
%\scirc&&\scirc&&\scirc&&\scirc\\[-10pt]
%&\sbullet&&\scirc&&\sbullet&&\scirc
%\end{array} \hspace{1mm}  \begin{array}{*{10}{c@{\hspace{1pt}}}}
%&\sbullet&&\scirc&&\sbullet&&\scirc\\[-10pt]
%&\scirc&&\sbullet&&\scirc&&\sbullet\\[-10pt]
%\scirc&&\scirc&&\scirc&&\scirc\\[-10pt]
%&\scirc&&\scirc&&\scirc&&\scirc
%\end{array}\hspace{1mm} \begin{array}{*{10}{c@{\hspace{1pt}}}}
%&\scirc&&\scirc&&\scirc&&\scirc\\[-10pt]
%&\scirc&&\sbullet&&\scirc&&\sbullet\\[-10pt]
%\scirc&&\scirc&&\scirc&&\scirc\\[-10pt]
%&\sbullet&&\scirc&&\sbullet&&\scirc
%\end{array} \hspace{1mm} \begin{array}{*{10}{c@{\hspace{1pt}}}}
%&\sbullet&&\scirc&&\sbullet&&\scirc\\[-10pt]
%&\scirc&&\scirc&&\scirc&&\scirc\\[-10pt]
%\scirc&&\scirc&&\scirc&&\scirc\\[-10pt]
%&\scirc&&\sbullet&&\scirc&&\sbullet
%\end{array} \hspace{1mm} \begin{array}{*{10}{c@{\hspace{1pt}}}}
%&\scirc&&\scirc&&\scirc&&\scirc\\[-10pt]
%&\sbullet&&\scirc&&\sbullet&&\scirc\\[-10pt]
%\scirc&&\scirc&&\scirc&&\scirc\\[-10pt]
%&\scirc&&\sbullet&&\scirc&&\sbullet
%\end{array}$& $\mathbb{D}_4$\\
(9)&\xymatrix{\circ\ar[r]\ar@/^/@{.}[rrr]_{}
&\circ\ar[r]&\circ\ar[r]&\circ
}& \tabincell{c}{\newcommand{\scirc}{{\scriptstyle \circ}}
\newcommand{\sbullet}{{\scriptstyle\bullet}}
$\begin{array}{*{10}{c@{\hspace{1pt}}}}
&\scirc&&\sbullet&&\scirc&&\sbullet\\[-5pt]
&\sbullet&&\scirc&&\sbullet&&\scirc\\[-5pt]
\scirc&&\scirc&&\scirc&&\scirc\\[-5pt]
&\scirc&&\scirc&&\scirc&&\scirc
\end{array}$\hspace{1mm}  $\begin{array}{*{10}{c@{\hspace{1pt}}}}
&\scirc&&\sbullet&&\scirc&&\sbullet\\[-5pt]
&\scirc&&\scirc&&\scirc&&\scirc\\[-5pt]
\scirc&&\scirc&&\scirc&&\scirc\\[-5pt]
&\sbullet&&\scirc&&\sbullet&&\scirc
\end{array}$ \hspace{1mm}  $\begin{array}{*{10}{c@{\hspace{1pt}}}}
&\sbullet&&\scirc&&\sbullet&&\scirc\\[-5pt]
&\scirc&&\sbullet&&\scirc&&\sbullet\\[-5pt]
\scirc&&\scirc&&\scirc&&\scirc\\[-5pt]
&\scirc&&\scirc&&\scirc&&\scirc
\end{array}$\\
\newcommand{\scirc}{{\scriptstyle \circ}}
\newcommand{\sbullet}{{\scriptstyle\bullet}}
$\begin{array}{*{10}{c@{\hspace{1pt}}}}
&\scirc&&\scirc&&\scirc&&\scirc\\[-5pt]
&\scirc&&\sbullet&&\scirc&&\sbullet\\[-5pt]
\scirc&&\scirc&&\scirc&&\scirc\\[-5pt]
&\sbullet&&\scirc&&\sbullet&&\scirc
\end{array}$\hspace{1mm} $\begin{array}{*{10}{c@{\hspace{1pt}}}}
&\sbullet&&\scirc&&\sbullet&&\scirc\\[-5pt]
&\scirc&&\scirc&&\scirc&&\scirc\\[-5pt]
\scirc&&\scirc&&\scirc&&\scirc\\[-5pt]
&\scirc&&\sbullet&&\scirc&&\sbullet
\end{array}$\hspace{1mm} $\begin{array}{*{10}{c@{\hspace{1pt}}}}
&\scirc&&\scirc&&\scirc&&\scirc\\[-5pt]
&\sbullet&&\scirc&&\sbullet&&\scirc\\[-5pt]
\scirc&&\scirc&&\scirc&&\scirc\\[-5pt]
&\scirc&&\sbullet&&\scirc&&\sbullet
\end{array}$}&
 $\mathbb{D}_4$ \\
\hline
(10)&\xymatrix{\circ\ar[r]&\circ~~~
\circ~~~\circ}&
\newcommand{\scirc}{{\scriptstyle \circ}}
\newcommand{\sbullet}{{\scriptstyle\bullet}}
$\begin{array}{*{12}{c@{\hspace{1pt}}}}
&\scirc&&\scirc&&\scirc&&\sbullet\\[-6pt]
&\scirc&&\scirc&&\scirc&&\sbullet\\[-6pt]
\sbullet&&\scirc&&\scirc&&\scirc\\[-6pt]
&\sbullet&&\scirc&&\scirc&&\scirc
\end{array} \hspace{3mm} \begin{array}{*{12}{c@{\hspace{1pt}}}}
&\scirc&&\scirc&&\scirc&&\sbullet\\[-6pt]
&\sbullet&&\scirc&&\scirc&&\scirc\\[-6pt]
\sbullet&&\scirc&&\scirc&&\scirc\\[-6pt]
&\scirc&&\scirc&&\scirc&&\sbullet
\end{array}  \hspace{3mm}  \begin{array}{*{12}{c@{\hspace{1pt}}}}
&\sbullet&&\scirc&&\scirc&&\scirc\\[-6pt]
&\scirc&&\scirc&&\scirc&&\sbullet\\[-6pt]
\sbullet&&\scirc&&\scirc&&\scirc\\[-6pt]
&\scirc&&\scirc&&\scirc&&\sbullet
\end{array}$&$\mathbb{A}_2\amalg\mathbb{A}_1\amalg\mathbb{A}_1$\\
\hline

(11)&\xymatrix{\circ~~~
\circ~~~\circ~~~\circ}&
\newcommand{\scirc}{{\scriptstyle \circ}}
\newcommand{\sbullet}{{\scriptstyle\bullet}}
$\begin{array}{*{12}{c@{\hspace{1pt}}}}
&\scirc&&\scirc&&\scirc&&\sbullet\\[-6pt]
&\scirc&&\scirc&&\scirc&&\sbullet\\[-6pt]
\sbullet&&\scirc&&\scirc&&\scirc\\[-6pt]
&\scirc&&\scirc&&\scirc&&\sbullet
\end{array}$&$\mathbb{A}_1\amalg\mathbb{A}_1\amalg\mathbb{A}_1\amalg\mathbb{A}_1$\\
\hline
\end{tabular}
\end{table}

%\[\begin{tabular}{|c|c|c|c|}
% \hline
%No.&silted algebras & 2-term silting complexes &  tilted type\\
%\hline
%
%(10)&\xymatrix{\circ\ar[r]&\circ~~~
%\circ~~~\circ}&
%\newcommand{\scirc}{{\scriptstyle \circ}}
%\newcommand{\sbullet}{{\scriptstyle\bullet}}
%$\begin{array}{*{20}{c@{\hspace{3pt}}}}
%&\scirc&&\scirc&&\scirc&&\sbullet\\[-5pt]
%&\scirc&&\scirc&&\scirc&&\sbullet\\[-5pt]
%\sbullet&&\scirc&&\scirc&&\scirc\\[-5pt]
%&\sbullet&&\scirc&&\scirc&&\scirc
%\end{array} \hspace{3mm} \begin{array}{*{20}{c@{\hspace{3pt}}}}
%&\scirc&&\scirc&&\scirc&&\sbullet\\[-5pt]
%&\sbullet&&\scirc&&\scirc&&\scirc\\[-5pt]
%\sbullet&&\scirc&&\scirc&&\scirc\\[-5pt]
%&\scirc&&\scirc&&\scirc&&\sbullet
%\end{array}  \hspace{3mm}  \begin{array}{*{20}{c@{\hspace{3pt}}}}
%&\sbullet&&\scirc&&\scirc&&\scirc\\[-5pt]
%&\scirc&&\scirc&&\scirc&&\sbullet\\[-5pt]
%\sbullet&&\scirc&&\scirc&&\scirc\\[-5pt]
%&\scirc&&\scirc&&\scirc&&\sbullet
%\end{array}$&$\mathbb{A}_2\amalg\mathbb{A}_1\amalg\mathbb{A}_1$\\
%\hline
%
%
%(11)&\xymatrix{\circ~~~
%\circ~~~\circ~~~\circ}&
%\newcommand{\scirc}{{\scriptstyle \circ}}
%\newcommand{\sbullet}{{\scriptstyle\bullet}}
%$\begin{array}{*{20}{c@{\hspace{3pt}}}}
%&\scirc&&\scirc&&\scirc&&\sbullet\\[-5pt]
%&\scirc&&\scirc&&\scirc&&\sbullet\\[-5pt]
%\sbullet&&\scirc&&\scirc&&\scirc\\[-5pt]
%&\scirc&&\scirc&&\scirc&&\sbullet
%\end{array}$&$\mathbb{A}_1\amalg\mathbb{A}_1\amalg\mathbb{A}_1\amalg\mathbb{A}_1$\\
%\hline
%
%\end{tabular}
%\]

To summarise, there are 11 silted algebras  of type $Q$, forming 4 families:
\begin{itemize}
\item[(\romannumeral1)] tilted algebras of type $\mathbb{D}_4$: $(1)-(7)$, $(9)$;
\item[(\romannumeral2)] tilted algebras of type $\mathbb{A}_3\coprod\mathbb{A}_1$: $(8)$;
\item[(\romannumeral3)] tilted algebras of type $\mathbb{A}_2\coprod\mathbb{A}_1\coprod\mathbb{A}_1$: $(10)$;
\item[(\romannumeral4)] $\mathbb{A}_1\coprod\mathbb{A}_1\coprod\mathbb{A}_1\coprod\mathbb{A}_1$:$(11)$.

\end{itemize}
%\newcommand{\scirc}{{\scriptstyle \circ}}
%\newcommand{\sbullet}{{\scriptstyle\bullet}}
%\[\begin{array}{*{20}{c@{\hspace{3pt}}}}
%&\scirc&&\scirc&&\scirc&&\scirc\\[-5pt]
%&\scirc&&\scirc&&\scirc&&\scirc\\[-5pt]
%\scirc&&\scirc&&\scirc&&\scirc\\[-5pt]
%&\scirc&&\scirc&&\scirc&&\scirc
%\end{array}
%\]
\end{example}

\vspace{13pt}
\subsection*{3.3.2 Type  $\mathbb{D}_5$}

\paragraph{\indent Let $A$ be the path algebra of the quiver}
$$
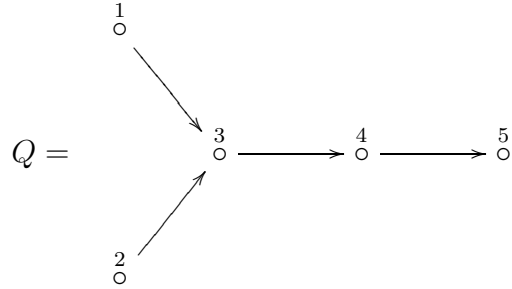
\xymatrix@C=8pt{&\stackrel{1}\circ\ar[rrd]&&\\
Q=&&&\stackrel{3}\circ\ar[rrr]&&&
\stackrel{4}\circ\ar[rrr]&&&\stackrel{5}\circ\\
&\stackrel{2}\circ\ar[rru]
}$$
According to \cite[Theorem 1]{ONFR}, there are 182 basic 2-term silting complexes, 77 of which are tilting modules. % By removing the right border we obtain the AR quiver of $\text{mod}A$.
The AR-quiver of  $\text{K}^{[-1.0]}(\text{proj}A)$  is
$$
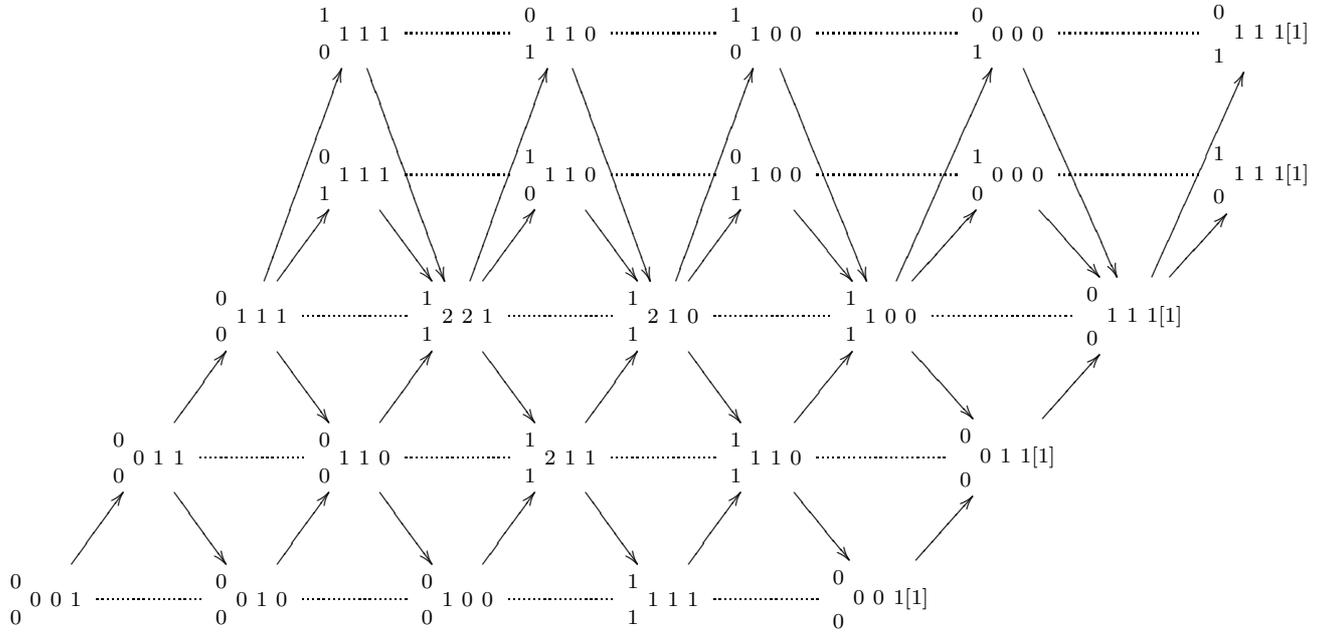
\xymatrix@C=0.0001pt{&&&{\begin{smallmatrix} 1&&&\\&1&1&1\\0&&& \end{smallmatrix}}\ar@{..}[rr]\ar[rdd]&&
{\begin{smallmatrix} 0&&&\\&1&1&0\\1&&& \end{smallmatrix}}\ar@{..}[rr]\ar[rdd]&&
{\begin{smallmatrix} 1&&&\\&1&0&0\\0&&& \end{smallmatrix}}\ar@{..}[rr]\ar[rdd]&&
{\begin{smallmatrix} 0&&&\\&0&0&0\\1&&&
\end{smallmatrix}}\ar@{..}[rr]\ar[rdd]&&
{\begin{smallmatrix} 0&&&\\&1&1&1[1]\\1&&&
\end{smallmatrix}}\\
&&&{\begin{smallmatrix}0&&&\\&1&1&1\\1&&& \end{smallmatrix}}\ar@{..}[rr]\ar[rd]&&
{\begin{smallmatrix}1&&&\\&1&1&0\\0&&& \end{smallmatrix}}\ar@{..}[rr]\ar[rd]&&
{\begin{smallmatrix}0&&&\\&1&0&0\\1&&& \end{smallmatrix}}\ar@{..}[rr]\ar[rd]&&
{\begin{smallmatrix}1&&&\\&0&0&0\\0&&&
\end{smallmatrix}}\ar@{..}[rr]\ar[rd]&&
{\begin{smallmatrix}1&&&\\&1&1&1[1]\\0&&&
\end{smallmatrix}}\\
&&{\begin{smallmatrix}0&&&\\&1&1&1\\0&&& \end{smallmatrix}}\ar@{..}[rr]\ar[rd]\ar[ru]\ar[ruu]&&
{\begin{smallmatrix}1&&&\\&2&2&1\\1&&& \end{smallmatrix}}\ar@{..}[rr]\ar[rd]\ar[ru]\ar[ruu]&&
{\begin{smallmatrix}1&&&\\&2&1&0\\1&&& \end{smallmatrix}}\ar@{..}[rr]\ar[rd]\ar[ru]\ar[ruu]&&
{\begin{smallmatrix}1&&&\\&1&0&0\\1&&& \end{smallmatrix}}\ar[ru]\ar[ruu]\ar[rd]\ar@{..}[rr]&&
{\begin{smallmatrix}0&&&\\&1&1&1[1]\\0&&& \end{smallmatrix}}\ar[ruu]\ar[ru]\\
&{\begin{smallmatrix}0&&&\\&0&1&1\\0&&& \end{smallmatrix}}\ar@{..}[rr]\ar[rd]\ar[ru]&&
{\begin{smallmatrix}0&&&\\&1&1&0\\0&&& \end{smallmatrix}}\ar@{..}[rr]\ar[rd]\ar[ru]&&
{\begin{smallmatrix}1&&&\\&2&1&1\\1&&& \end{smallmatrix}}\ar@{..}[rr]\ar[rd]\ar[ru]&&
{\begin{smallmatrix}1&&&\\&1&1&0\\1&&& \end{smallmatrix}}\ar@{..}[rr]\ar[ru]\ar[rd]&&
{\begin{smallmatrix}0&&&\\&0&1&1[1]\\0&&& \end{smallmatrix}}\ar[ru]\\
{\begin{smallmatrix}0&&&\\&0&0&1\\0&&& \end{smallmatrix}}\ar@{..}[rr]\ar[ru]&&
{\begin{smallmatrix}0&&&\\&0&1&0\\0&&& \end{smallmatrix}}\ar@{..}[rr]\ar[ru]&&
{\begin{smallmatrix}0&&&\\&1&0&0\\0&&& \end{smallmatrix}}\ar@{..}[rr]\ar[ru]&&
{\begin{smallmatrix}1&&&\\&1&1&1\\1&&& \end{smallmatrix}}\ar[ru]\ar@{..}[rr]&&
{\begin{smallmatrix}0&&&\\&0&0&1[1]\\0&&& \end{smallmatrix}}\ar[ru]
}$$

(I)
Tilted algebras of type $Q$ are:

$
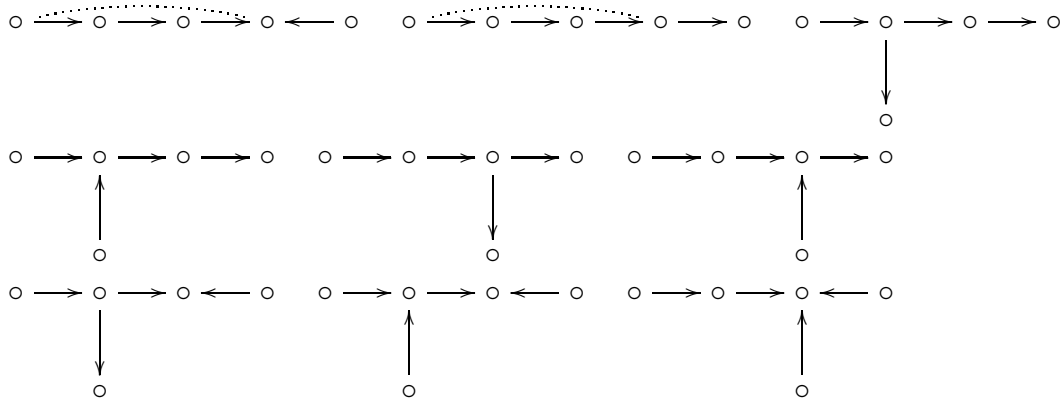
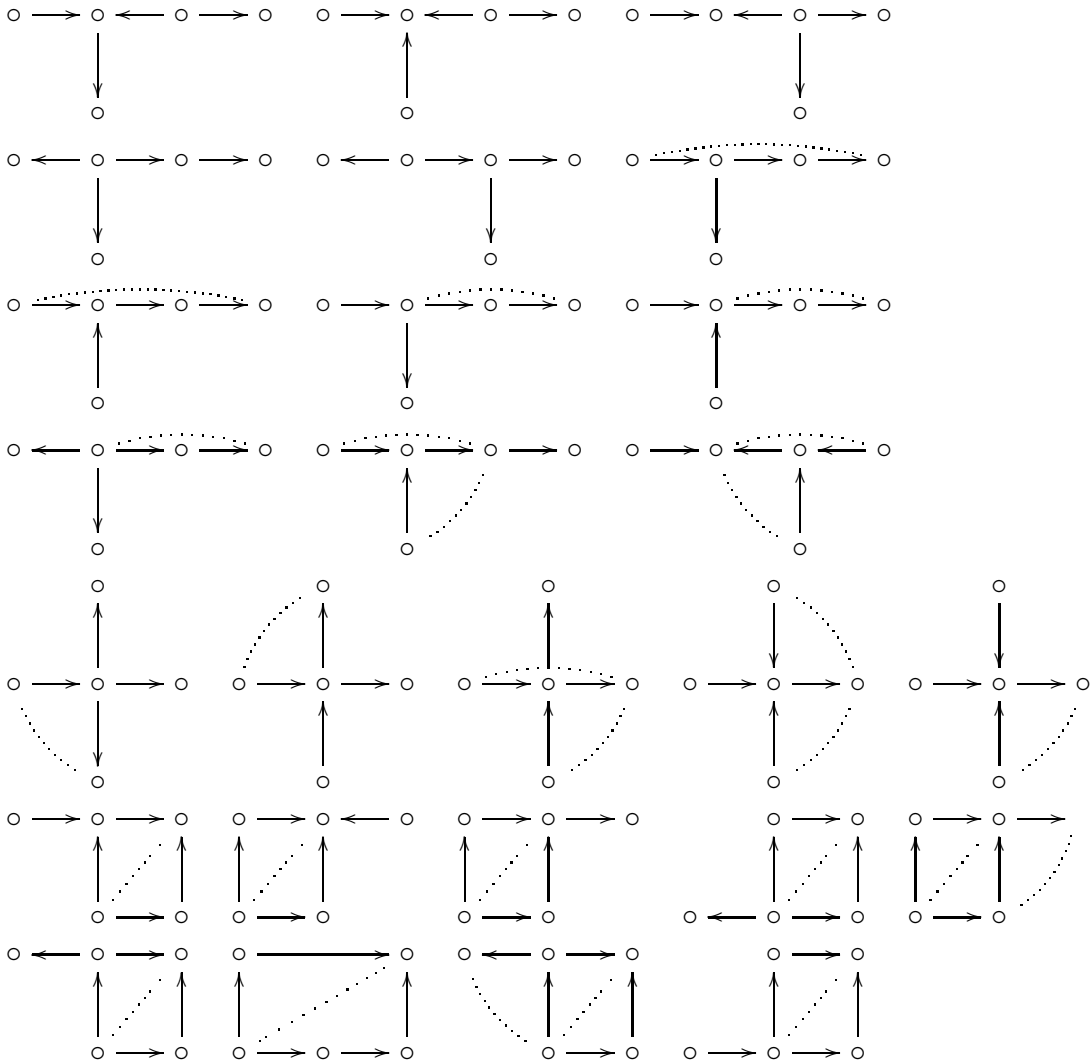
\xymatrix@C=5pt{\circ\ar[rr]\ar@/^/@{.}[rrrrrr]&&
\circ\ar[rr]&&\circ\ar[rr]&&\circ&&\circ\ar[ll]
}$~~~\text{}~~~
$\xymatrix@C=5pt{\circ\ar[rr]\ar@/^/@{.}[rrrrrr]_{}&&
\circ\ar[rr]&&\circ\ar[rr]&&
\circ\ar[rr]&&\circ
}$~~~\text{}~~~
$\xymatrix@C=5pt{\circ\ar[rr]&&\circ\ar[rr]\ar[d]&&
\circ\ar[rr]&&\circ\\
&&\circ
}$

$\xymatrix@C=5pt{\circ\ar[rr]&&\circ\ar[rr]&&\circ
\ar[rr]&&\circ\\
&&\circ\ar[u]
}$~~~\text{}~~~
$\xymatrix@C=5pt{\circ\ar[rr]&&\circ\ar[rr]&&
\circ\ar[rr]\ar[d]&&\circ\\
&&&&\circ
}$~~~\text{}~~~
$\xymatrix@C=5pt{\circ\ar[rr]&&\circ\ar[rr]&&
\circ\ar[rr]&&\circ\\
&&&&\circ\ar[u]
}$

$\xymatrix@C=5pt{\circ\ar[rr]&&\circ\ar[rr]\ar[d]&&
\circ&&\circ\ar[ll]\\
&&\circ
}$~~~\text{}~~~
$\xymatrix@C=5pt{\circ\ar[rr]&&\circ\ar[rr]&&\circ
&&\circ\ar[ll]\\
&&\circ\ar[u]
}$~~~\text{}~~~
$\xymatrix@C=5pt{\circ\ar[rr]&&\circ\ar[rr]&&
\circ&&\circ\ar[ll]\\
&&&&\circ\ar[u]
}$

$\xymatrix@C=5pt{\circ\ar[rr]&&\circ\ar[d]&&
\circ\ar[rr]\ar[ll]&&\circ\\
&&\circ
}$~~~\text{}~~~
$\xymatrix@C=5pt{\circ\ar[rr]&&\circ&&\circ
\ar[rr]\ar[ll]&&\circ\\
&&\circ\ar[u]
}$~~~\text{}~~~
$\xymatrix@C=5pt{\circ\ar[rr]&&\circ&&\circ
\ar[ll]\ar[rr]\ar[d]&&\circ\\
&&&&\circ
}$

$\xymatrix@C=5pt{\circ&&\circ\ar[ll]\ar[rr]\ar[d]
&&\circ\ar[rr]&&\circ\\
&&\circ
}$~~~\text{}~~~
$\xymatrix@C=5pt{\circ&&\circ\ar[ll]\ar[rr]&&
\circ\ar[rr]\ar[d]&&\circ\\
&&&&\circ
}$~~~\text{}~~~
$\xymatrix@C=5pt{\circ\ar[rr]\ar@/^/@{.}[rrrrrr]_{}&&
\circ\ar[rr]\ar[d]&&
\circ\ar[rr]&&\circ\\
&&\circ
}$

$\xymatrix@C=5pt{\circ\ar[rr]\ar@/^/@{.}[rrrrrr]_{}&&
\circ\ar[rr]&&
\circ\ar[rr]&&\circ\\
&&\circ\ar[u]
}$~~~\text{}~~~
$\xymatrix@C=5pt{\circ\ar[rr]&&\circ\ar[rr]
\ar[d]\ar@/^/@{.}[rrrr]_{}&&
\circ\ar[rr]&&\circ\\
&&\circ
}$~~~\text{}~~~
$\xymatrix@C=5pt{\circ\ar[rr]&&
\circ\ar[rr]\ar@/^/@{.}[rrrr]_{}&&
\circ\ar[rr]&&\circ\\
&&\circ\ar[u]
}$

$\xymatrix@C=5pt{\circ&&\circ\ar[rr]
\ar[ll]\ar[d]\ar@/^/@{.}[rrrr]_{}&&
\circ\ar[rr]&&\circ\\
&&\circ
}$~~~\text{}~~~
$\xymatrix@C=5pt{\circ\ar[rr]\ar@/^/@{.}[rrrr]_{}&&
\circ\ar[rr]&&
\circ\ar[rr]&&\circ\\
&&\circ\ar[u]\ar@/_/@{.}[rru]_{}
}$~~~\text{}~~~
$\xymatrix@C=5pt{\circ\ar[rr]&&\circ&&
\circ\ar[ll]&&\circ\ar[ll]\ar@/_/@{.}[llll]_{}\\
&&&&\circ\ar[u]\ar@/^/@{.}[llu]_{}
}$

$\xymatrix@C=5pt{&&\circ\\
\circ\ar[rr]\ar@/_/@{.}[rrd]_{}&&\circ\ar[rr]\ar[u]
\ar[d]&&\circ\\
&&\circ
}$~~~\text{}~~~
$\xymatrix@C=5pt{&&\circ\\
\circ\ar[rr]\ar@/^/@{.}[rru]_{}&&\circ\ar[u]\ar[rr]
&&\circ\\
&&\circ\ar[u]
}$~~~\text{}~~~
$\xymatrix@C=5pt{&&\circ\\
\circ\ar[rr]\ar@/^/@{.}[rrrr]_{}&&\circ\ar[u]\ar[rr]
&&\circ\\
&&\circ\ar[u]\ar@/_/@{.}[rru]_{}
}$~~~\text{}~~~
$\xymatrix@C=5pt{&&\circ\ar[d]\ar@/^/@{.}[rrd]_{}\\
\circ\ar[rr]&&\circ\ar[rr]&&\circ\\
&&\circ\ar[u]\ar@/_/@{.}[rru]_{}
}$~~~\text{}~~~
$\xymatrix@C=5pt{&&\circ\ar[d]\\
\circ\ar[rr]&&\circ\ar[rr]&&\circ\\
&&\circ\ar[u]\ar@/_/@{.}[rru]_{}
}$

$\xymatrix@C=5pt{\circ\ar[rr]&&\circ\ar[rr]&&\circ\\
&&\circ\ar[rr]\ar[u]\ar@//@{.}[rru]_{}&&\circ\ar[u]
}$~~~\text{}~~~
$\xymatrix@C=5pt{\circ\ar[rr]&&\circ&&\circ\ar[ll]\\
\circ\ar[u]\ar[rr]\ar@//@{.}[rru]_{}&&\circ\ar[u]
}$
~~~\text{}~~~
$\xymatrix@C=5pt{\circ\ar[rr]&&\circ\ar[rr]&&\circ\\
\circ\ar[u]\ar[rr]\ar@//@{.}[rru]_{}&&\circ\ar[u]
}$~~\text{}~~~
$\xymatrix@C=5pt{&&\circ\ar[rr]&&\circ\\
\circ&&\circ\ar[ll]\ar[u]\ar[rr]\ar@//@{.}[rru]_{}&&\circ\ar[u]
}$~~\text{}~~~
$\xymatrix@C=5pt{\circ\ar[rr]&&\circ\ar[rr]&&\\
\circ\ar[u]\ar[rr]\ar@//@{.}[rru]_{}
&&\circ\ar[u]\ar@/_/@{.}[rru]_{}}$

$\xymatrix@C=5pt{\circ&&\circ\ar[ll]\ar[rr]&&\circ\\
&&\circ\ar[u]\ar[rr]\ar@//@{.}[rru]_{}&&\circ\ar[u]
}$~~~\text{}~~~
$\xymatrix@C=5pt{\circ\ar[rrrr]&&&&\circ\\
\circ\ar[u]\ar[rr]\ar@//@{.}[rrrru]_{}&&\circ\ar[rr]&&\circ\ar[u]
}$
~~~\text{}~~~
$\xymatrix@C=5pt{\circ&&\circ\ar[ll]\ar[rr]&&\circ\\
&&\circ\ar[u]\ar[rr]\ar@/^/@{.}[llu]_{}\ar@//@{.}[rru]_{}&&\circ\ar[u]
}$~~~\text{}~~~
$\xymatrix@C=5pt{&&\circ\ar[rr]&&\circ\\
\circ\ar[rr]&&\circ\ar[u]\ar[rr]
\ar@//@{.}[rru]_{}&&\circ\ar[u]}$

\smallskip
(II) Silted algebras of the form $\text{End}_{\text{K}^b(\text{proj}A)}(M\oplus P[1])$ ($P\neq 0,~M\neq 0$) are

\newcommand{\scirc}{{\scriptstyle \circ}}
\newcommand{\sbullet}{{\scriptstyle\bullet}}
\vspace{-20pt}
\[
\begin{tabular}{|c|c|c|c|}
 \hline
No. &
silted algebras & 2-term silting complexes & tilted type\\
\hline
(1)&
\xymatrix@C=3pt{\circ\ar[rr]\ar@/^/@{.}[rrrrrrrr]_{}&&
\circ\ar[rr]&&\circ\ar[rr]&&
\circ\ar[rr]&&\circ
} & $\begin{array}{*{12}{c@{\hspace{0.3pt}}}}
&&&\scirc&&\sbullet&&\scirc&&\sbullet&&\scirc\\[-5pt]
&&&\sbullet&&\scirc&&\sbullet&&\scirc&&\sbullet\\[-5pt]
&&\scirc&&\scirc&&\scirc&&\scirc&&\scirc\\[-5pt]
&\scirc&&\scirc&&\scirc&&\scirc&&\scirc\\[-5pt]
\scirc&&\scirc&&\scirc&&\scirc&&\scirc
\end{array}
\hspace{0.3mm}
\begin{array}{*{12}{c@{\hspace{0.3pt}}}}
&&&\sbullet&&\scirc&&\sbullet&&\scirc&&\sbullet\\[-5pt]
&&&\scirc&&\sbullet&&\scirc&&\sbullet&&\scirc\\[-5pt]
&&\scirc&&\scirc&&\scirc&&\scirc&&\scirc\\[-5pt]
&\scirc&&\scirc&&\scirc&&\scirc&&\scirc\\[-5pt]
\scirc&&\scirc&&\scirc&&\scirc&&\scirc
\end{array} $ & $\mathbb{D}_5$ \\
\hline
(2)&
\xymatrix@C=3pt{\circ\ar[rr]\ar@/^/@{.}[rrrrrr]_{}&&
\circ\ar[rr]&&\circ\ar[rr]\ar@/^/@{.}[rrrr]_{}&&
\circ\ar[rr]&&\circ
} &
$\begin{array}{*{12}{c@{\hspace{0.3pt}}}}
&&&\scirc&&\scirc&&\scirc&&\sbullet&&\scirc\\[-5pt]
&&&\sbullet&&\scirc&&\sbullet&&\scirc&&\sbullet\\[-5pt]
&&\scirc&&\scirc&&\scirc&&\scirc&&\scirc\\[-5pt]
&\scirc&&\scirc&&\scirc&&\scirc&&\scirc\\[-5pt]
\sbullet&&\scirc&&\scirc&&\scirc&&\scirc
\end{array}
\hspace{0.3mm}
\begin{array}{*{12}{c@{\hspace{0.3pt}}}}
&&&\sbullet&&\scirc&&\sbullet&&\scirc&&\sbullet\\[-5pt]
&&&\scirc&&\scirc&&\scirc&&\sbullet&&\scirc\\[-5pt]
&&\scirc&&\scirc&&\scirc&&\scirc&&\scirc\\[-5pt]
&\scirc&&\scirc&&\scirc&&\scirc&&\scirc\\[-5pt]
\sbullet&&\scirc&&\scirc&&\scirc&&\scirc
\end{array}$ & strictly shod\\
\hline
(3)&
\xymatrix@C=3pt{\circ\ar[rr]\ar@/^/@{.}[rrrr]&&
\circ\ar[rr]\ar@/^/@{.}[rrrr]&&\circ\ar[rr]&&
\circ&&\circ\ar[ll]
} &
$\begin{array}{*{12}{c@{\hspace{0.3pt}}}}
&&&\scirc&&\scirc&&\scirc&&\sbullet&&\scirc\\[-5pt]
&&&\sbullet&&\scirc&&\scirc&&\scirc&&\sbullet\\[-5pt]
&&\scirc&&\scirc&&\scirc&&\scirc&&\scirc\\[-5pt]
&\sbullet&&\scirc&&\scirc&&\scirc&&\scirc\\[-5pt]
\scirc&&\sbullet&&\scirc&&\scirc&&\scirc
\end{array}
\hspace{0.3mm}
\begin{array}{*{12}{c@{\hspace{0.3pt}}}}
&&&\sbullet&&\scirc&&\scirc&&\scirc&&\sbullet\\[-5pt]
&&&\scirc&&\scirc&&\scirc&&\sbullet&&\scirc\\[-5pt]
&&\scirc&&\scirc&&\scirc&&\scirc&&\scirc\\[-5pt]
&\sbullet&&\scirc&&\scirc&&\scirc&&\scirc\\[-5pt]
\scirc&&\sbullet&&\scirc&&\scirc&&\scirc
\end{array}$ & strictly shod\\
\hline
%(4)&
%\xymatrix@C=3pt{\circ\ar[rr]\ar@/^/@{.}[rrrr]&&
%\circ\ar[rr]\ar@/^/@{.}[rrrr]&&\circ\ar[rr]&&
%\circ\ar[rr]&&\circ
%}&$\begin{array}{*{12}{c@{\hspace{0.3pt}}}}
%&&&\scirc&&\scirc&&\scirc&&\sbullet&&\scirc\\[-5pt]
%&&&\sbullet&&\scirc&&\scirc&&\scirc&&\sbullet\\[-5pt]
%&&\scirc&&\scirc&&\scirc&&\scirc&&\scirc\\[-5pt]
%&\sbullet&&\scirc&&\scirc&&\scirc&&\scirc\\[-5pt]
%\sbullet&&\scirc&&\scirc&&\scirc&&\scirc
%\end{array}
%\hspace{0.3mm}
%\begin{array}{*{12}{c@{\hspace{0.3pt}}}}
%&&&\sbullet&&\scirc&&\scirc&&\scirc&&\sbullet\\[-5pt]
%&&&\scirc&&\scirc&&\scirc&&\sbullet&&\scirc\\[-5pt]
%&&\scirc&&\scirc&&\scirc&&\scirc&&\scirc\\[-5pt]
%&\sbullet&&\scirc&&\scirc&&\scirc&&\scirc\\[-5pt]
%\sbullet&&\scirc&&\scirc&&\scirc&&\scirc
%\end{array}$&  strictly shod\\
%\hline
\end{tabular}
\]

\[
\begin{tabular}{|c|c|c|c|}
 \hline
No. &
silted algebras & 2-term silting complexes & tilted type\\
\hline
(4)&
\xymatrix@C=3pt{\circ\ar[rr]\ar@/^/@{.}[rrrr]&&
\circ\ar[rr]\ar@/^/@{.}[rrrr]&&\circ\ar[rr]&&
\circ\ar[rr]&&\circ
}&$\begin{array}{*{12}{c@{\hspace{0.3pt}}}}
&&&\scirc&&\scirc&&\scirc&&\sbullet&&\scirc\\[-5pt]
&&&\sbullet&&\scirc&&\scirc&&\scirc&&\sbullet\\[-5pt]
&&\scirc&&\scirc&&\scirc&&\scirc&&\scirc\\[-5pt]
&\sbullet&&\scirc&&\scirc&&\scirc&&\scirc\\[-5pt]
\sbullet&&\scirc&&\scirc&&\scirc&&\scirc
\end{array}
\hspace{0.3mm}
\begin{array}{*{12}{c@{\hspace{0.3pt}}}}
&&&\sbullet&&\scirc&&\scirc&&\scirc&&\sbullet\\[-5pt]
&&&\scirc&&\scirc&&\scirc&&\sbullet&&\scirc\\[-5pt]
&&\scirc&&\scirc&&\scirc&&\scirc&&\scirc\\[-5pt]
&\sbullet&&\scirc&&\scirc&&\scirc&&\scirc\\[-5pt]
\sbullet&&\scirc&&\scirc&&\scirc&&\scirc
\end{array}$&  strictly shod\\
\hline
(5)&
\xymatrix@C=3pt{\circ\ar[rr]\ar@/^/@{.}[rrrrrr]_{}&&
\circ\ar[rr]&&\circ\ar[rr]&&
\circ~~~\circ
}&
$\begin{array}{*{12}{c@{\hspace{0.3pt}}}}
&&&\scirc&&\scirc&&\scirc&&\scirc&&\scirc\\[-5pt]
&&&\scirc&&\scirc&&\sbullet&&\scirc&&\sbullet\\[-5pt]
&&\scirc&&\scirc&&\scirc&&\scirc&&\scirc\\[-5pt]
&\scirc&&\scirc&&\scirc&&\scirc&&\sbullet\\[-5pt]
\sbullet&&\scirc&&\sbullet&&\scirc&&\scirc
\end{array}
\hspace{0.3mm}
\begin{array}{*{12}{c@{\hspace{0.3pt}}}}
&&&\scirc&&\scirc&&\sbullet&&\scirc&&\sbullet\\[-5pt]
&&&\scirc&&\scirc&&\scirc&&\scirc&&\scirc\\[-5pt]
&&\scirc&&\scirc&&\scirc&&\scirc&&\scirc\\[-5pt]
&\scirc&&\scirc&&\scirc&&\scirc&&\sbullet\\[-5pt]
\sbullet&&\scirc&&\sbullet&&\scirc&&\scirc
\end{array}$& $\mathbb{D}_4\amalg\mathbb{A}_1$ \\
\hline
(6)&
\xymatrix@C=3pt{\circ\ar[rr]&&\circ\ar[rr]&&
\circ\ar[rr]&&\circ~~~\circ
}&
$\begin{array}{*{12}{c@{\hspace{0.3pt}}}}
&&&\scirc&&\scirc&&\scirc&&\scirc&&\scirc\\[-5pt]
&&&\sbullet&&\scirc&&\scirc&&\scirc&&\sbullet\\[-5pt]
&&\sbullet&&\scirc&&\scirc&&\scirc&&\scirc\\[-5pt]
&\sbullet&&\scirc&&\scirc&&\scirc&&\scirc\\[-5pt]
\sbullet&&\scirc&&\scirc&&\scirc&&\scirc
\end{array}
\hspace{0.3mm}
\begin{array}{*{12}{c@{\hspace{0.3pt}}}}
&&&\sbullet&&\scirc&&\scirc&&\scirc&&\sbullet\\[-5pt]
&&&\scirc&&\scirc&&\scirc&&\scirc&&\scirc\\[-5pt]
&&\sbullet&&\scirc&&\scirc&&\scirc&&\scirc\\[-5pt]
&\sbullet&&\scirc&&\scirc&&\scirc&&\scirc\\[-5pt]
\sbullet&&\scirc&&\scirc&&\scirc&&\scirc
\end{array}$& $\mathbb{A}_4\amalg\mathbb{A}_1$\\
\hline
(7)&
\xymatrix@C=3pt{\circ\ar[rr]&&\circ\ar[rr]&&
\circ&&\circ\ar[ll]~~~\circ
}&
$\begin{array}{*{12}{c@{\hspace{0.3pt}}}}
&&&\scirc&&\scirc&&\scirc&&\scirc&&\scirc\\[-5pt]
&&&\sbullet&&\scirc&&\scirc&&\scirc&&\sbullet\\[-5pt]
&&\sbullet&&\scirc&&\scirc&&\scirc&&\scirc\\[-5pt]
&\scirc&&\sbullet&&\scirc&&\scirc&&\scirc\\[-5pt]
\scirc&&\scirc&&\sbullet&&\scirc&&\scirc
\end{array}
\hspace{0.3mm}
\begin{array}{*{12}{c@{\hspace{0.3pt}}}}
&&&\scirc&&\scirc&&\scirc&&\scirc&&\scirc\\[-5pt]
&&&\sbullet&&\scirc&&\scirc&&\scirc&&\sbullet\\[-5pt]
&&\sbullet&&\scirc&&\scirc&&\scirc&&\scirc\\[-5pt]
&\sbullet&&\scirc&&\scirc&&\scirc&&\scirc\\[-5pt]
\scirc&&\sbullet&&\scirc&&\scirc&&\scirc
\end{array}
\hspace{0.3mm}
\begin{array}{*{12}{c@{\hspace{0.3pt}}}}
&&&\sbullet&&\scirc&&\scirc&&\scirc&&\sbullet\\[-5pt]
&&&\scirc&&\scirc&&\scirc&&\scirc&&\scirc\\[-5pt]
&&\sbullet&&\scirc&&\scirc&&\scirc&&\scirc\\[-5pt]
&\scirc&&\sbullet&&\scirc&&\scirc&&\scirc\\[-5pt]
\scirc&&\scirc&&\sbullet&&\scirc&&\scirc
\end{array}
\hspace{0.3mm}
\begin{array}{*{12}{c@{\hspace{0.3pt}}}}
&&&\sbullet&&\scirc&&\scirc&&\scirc&&\sbullet\\[-5pt]
&&&\scirc&&\scirc&&\scirc&&\scirc&&\scirc\\[-5pt]
&&\sbullet&&\scirc&&\scirc&&\scirc&&\scirc\\[-5pt]
&\sbullet&&\scirc&&\scirc&&\scirc&&\scirc\\[-5pt]
\scirc&&\sbullet&&\scirc&&\scirc&&\scirc
\end{array}$&$\mathbb{A}_4\amalg\mathbb{A}_1$\\
\hline
(8)&
\xymatrix@C=5pt{\circ\ar[rr]&&\circ&&\circ\ar[ll]
\ar[rr]&&\circ~~~\circ
}&
$\begin{array}{*{12}{c@{\hspace{0.3pt}}}}
&&&\scirc&&\scirc&&\scirc&&\scirc&&\scirc\\[-5pt]
&&&\sbullet&&\scirc&&\scirc&&\scirc&&\sbullet\\[-5pt]
&&\sbullet&&\scirc&&\scirc&&\scirc&&\scirc\\[-5pt]
&\scirc&&\sbullet&&\scirc&&\scirc&&\scirc\\[-5pt]
\scirc&&\sbullet&&\scirc&&\scirc&&\scirc
\end{array}
\hspace{0.3mm}
\begin{array}{*{12}{c@{\hspace{0.3pt}}}}
&&&\sbullet&&\scirc&&\scirc&&\scirc&&\sbullet\\[-5pt]
&&&\scirc&&\scirc&&\scirc&&\scirc&&\scirc\\[-5pt]
&&\sbullet&&\scirc&&\scirc&&\scirc&&\scirc\\[-5pt]
&\scirc&&\sbullet&&\scirc&&\scirc&&\scirc\\[-5pt]
\scirc&&\sbullet&&\scirc&&\scirc&&\scirc
\end{array}$&$\mathbb{A}_4\amalg\mathbb{A}_1$\\
\hline
(9)&
\xymatrix@C=5pt{\circ\ar[rr]&&\circ~~~
\circ\ar[rr]&&\circ\ar[rr]&&\circ
}&
$\begin{array}{*{12}{c@{\hspace{0.3pt}}}}
&&&\scirc&&\scirc&&\scirc&&\sbullet&&\scirc\\[-5pt]
&&&\scirc&&\scirc&&\scirc&&\scirc&&\sbullet\\[-5pt]
&&\scirc&&\scirc&&\scirc&&\scirc&&\sbullet\\[-5pt]
&\sbullet&&\scirc&&\scirc&&\scirc&&\scirc\\[-5pt]
\scirc&&\sbullet&&\scirc&&\scirc&&\scirc
\end{array}
\hspace{0.3mm}
\begin{array}{*{12}{c@{\hspace{0.3pt}}}}
&&&\scirc&&\scirc&&\scirc&&\sbullet&&\scirc\\[-5pt]
&&&\scirc&&\scirc&&\scirc&&\scirc&&\sbullet\\[-5pt]
&&\scirc&&\scirc&&\scirc&&\scirc&&\sbullet\\[-5pt]
&\sbullet&&\scirc&&\scirc&&\scirc&&\scirc\\[-5pt]
\sbullet&&\scirc&&\scirc&&\scirc&&\scirc
\end{array}
\hspace{0.3mm}
\begin{array}{*{12}{c@{\hspace{0.3pt}}}}
&&&\scirc&&\scirc&&\scirc&&\scirc&&\sbullet\\[-5pt]
&&&\scirc&&\scirc&&\scirc&&\sbullet&&\scirc\\[-5pt]
&&\scirc&&\scirc&&\scirc&&\scirc&&\sbullet\\[-5pt]
&\sbullet&&\scirc&&\scirc&&\scirc&&\scirc\\[-5pt]
\scirc&&\sbullet&&\scirc&&\scirc&&\scirc
\end{array}
\hspace{0.3mm}
\begin{array}{*{12}{c@{\hspace{0.3pt}}}}
&&&\scirc&&\scirc&&\scirc&&\scirc&&\sbullet\\[-5pt]
&&&\scirc&&\scirc&&\scirc&&\sbullet&&\scirc\\[-5pt]
&&\scirc&&\scirc&&\scirc&&\scirc&&\sbullet\\[-5pt]
&\sbullet&&\scirc&&\scirc&&\scirc&&\scirc\\[-5pt]
\sbullet&&\scirc&&\scirc&&\scirc&&\scirc
\end{array}$&$\mathbb{A}_3\amalg\mathbb{A}_2$\\
\hline
(10)&
\xymatrix@C=5pt{\circ\ar[rr]&&\circ~~~
\circ\ar[rr]&&\circ&&\circ\ar[ll]
}&
$\begin{array}{*{12}{c@{\hspace{0.3pt}}}}
&&&\scirc&&\scirc&&\scirc&&\scirc&&\sbullet\\[-5pt]
&&&\scirc&&\scirc&&\scirc&&\scirc&&\sbullet\\[-5pt]
&&\scirc&&\scirc&&\scirc&&\scirc&&\sbullet\\[-5pt]
&\sbullet&&\scirc&&\scirc&&\scirc&&\scirc\\[-5pt]
\scirc&&\sbullet&&\scirc&&\scirc&&\scirc
\end{array}
\hspace{0.3mm}
\begin{array}{*{12}{c@{\hspace{0.3pt}}}}
&&&\scirc&&\scirc&&\scirc&&\scirc&&\sbullet\\[-5pt]
&&&\scirc&&\scirc&&\scirc&&\scirc&&\sbullet\\[-5pt]
&&\scirc&&\scirc&&\scirc&&\scirc&&\sbullet\\[-5pt]
&\sbullet&&\scirc&&\scirc&&\scirc&&\scirc\\[-5pt]
\sbullet&&\scirc&&\scirc&&\scirc&&\scirc
\end{array}$&$\mathbb{A}_3\amalg\mathbb{A}_2$\\
\hline
(11)&
\xymatrix@C=5pt{\circ\ar[rr]&&\circ~~~
\circ&&\circ\ar[ll]\ar[rr]&&\circ
}&
$\begin{array}{*{12}{c@{\hspace{0.3pt}}}}
&&&\scirc&&\scirc&&\scirc&&\sbullet&&\scirc\\[-5pt]
&&&\scirc&&\scirc&&\scirc&&\sbullet&&\scirc\\[-5pt]
&&\scirc&&\scirc&&\scirc&&\scirc&&\sbullet\\[-5pt]
&\sbullet&&\scirc&&\scirc&&\scirc&&\scirc\\[-5pt]
\scirc&&\sbullet&&\scirc&&\scirc&&\scirc
\end{array}
\hspace{0.3mm}
\begin{array}{*{12}{c@{\hspace{0.3pt}}}}
&&&\scirc&&\scirc&&\scirc&&\sbullet&&\scirc\\[-5pt]
&&&\scirc&&\scirc&&\scirc&&\sbullet&&\scirc\\[-5pt]
&&\scirc&&\scirc&&\scirc&&\scirc&&\sbullet\\[-5pt]
&\sbullet&&\scirc&&\scirc&&\scirc&&\scirc\\[-5pt]
\sbullet&&\scirc&&\scirc&&\scirc&&\scirc
\end{array}$&$\mathbb{A}_3\amalg\mathbb{A}_2$\\
\hline
(12)&
\xymatrix@C=5pt{\circ\ar[rr]&&\circ\ar[rr]
&&\circ~~~\circ~~~\circ
}&
$\begin{array}{*{12}{c@{\hspace{0.3pt}}}}
&&&\scirc&&\scirc&&\scirc&&\scirc&&\sbullet\\[-5pt]
&&&\scirc&&\scirc&&\scirc&&\scirc&&\sbullet\\[-5pt]
&&\sbullet&&\scirc&&\scirc&&\scirc&&\scirc\\[-5pt]
&\scirc&&\sbullet&&\scirc&&\scirc&&\scirc\\[-5pt]
\scirc&&\scirc&&\sbullet&&\scirc&&\scirc
\end{array}
\hspace{0.3mm}
\begin{array}{*{12}{c@{\hspace{0.3pt}}}}
&&&\scirc&&\scirc&&\scirc&&\scirc&&\sbullet\\[-5pt]
&&&\scirc&&\scirc&&\scirc&&\scirc&&\sbullet\\[-5pt]
&&\sbullet&&\scirc&&\scirc&&\scirc&&\scirc\\[-5pt]
&\sbullet&&\scirc&&\scirc&&\scirc&&\scirc\\[-5pt]
\sbullet&&\scirc&&\scirc&&\scirc&&\scirc
\end{array}$&$\mathbb{A}_3\amalg\mathbb{A}_1
\amalg\mathbb{A}_1$\\
\hline
(13)&
\xymatrix@C=5pt{\circ&&\circ\ar[ll]\ar[rr]&&\circ
~~~\circ~~~\circ
}&
$\begin{array}{*{12}{c@{\hspace{0.3pt}}}}
&&&\scirc&&\scirc&&\scirc&&\scirc&&\sbullet\\[-5pt]
&&&\scirc&&\scirc&&\scirc&&\scirc&&\sbullet\\[-5pt]
&&\sbullet&&\scirc&&\scirc&&\scirc&&\scirc\\[-5pt]
&\scirc&&\sbullet&&\scirc&&\scirc&&\scirc\\[-5pt]
\scirc&&\sbullet&&\scirc&&\scirc&&\scirc
\end{array}$&$\mathbb{A}_3\amalg\mathbb{A}_1
\amalg\mathbb{A}_1$\\
\hline
(14)&
\xymatrix@C=5pt{\circ\ar[rr]&&\circ&&\circ\ar[ll]
~~~\circ~~~\circ
}&
$\begin{array}{*{12}{c@{\hspace{0.3pt}}}}
&&&\scirc&&\scirc&&\scirc&&\scirc&&\sbullet\\[-5pt]
&&&\scirc&&\scirc&&\scirc&&\scirc&&\sbullet\\[-5pt]
&&\sbullet&&\scirc&&\scirc&&\scirc&&\scirc\\[-5pt]
&\sbullet&&\scirc&&\scirc&&\scirc&&\scirc\\[-5pt]
\scirc&&\sbullet&&\scirc&&\scirc&&\scirc
\end{array}$&$\mathbb{A}_3\amalg\mathbb{A}_1
\amalg\mathbb{A}_1$\\
\hline
%(15)&
%\xymatrix@C=5pt{\circ\ar[rr]\ar@/^/@{.}[rrrr]_{}&&
%\circ\ar[rr]&&\circ&~~~\circ~~~\circ
%}&
%$\begin{array}{*{12}{c@{\hspace{0.3pt}}}}
%&&&\scirc&&\scirc&&\scirc&&\scirc&&\sbullet\\[-5pt]
%&&&\scirc&&\scirc&&\scirc&&\scirc&&\sbullet\\[-5pt]
%&&\sbullet&&\scirc&&\scirc&&\scirc&&\scirc\\[-5pt]
%&\scirc&&\scirc&&\scirc&&\scirc&&\scirc\\[-5pt]
%\sbullet&&\scirc&&\sbullet&&\scirc&&\scirc
%\end{array}$&$\mathbb{A}_3\amalg\mathbb{A}_1
%\amalg\mathbb{A}_1$\\
%\hline
\end{tabular}
\]

\[\begin{tabular}{|c|c|c|c|}
 \hline
No. &
silted algebras & 2-term silting complexes & tilted type\\
\hline
(15)&
\xymatrix@C=5pt{\circ\ar[rr]\ar@/^/@{.}[rrrr]_{}&&
\circ\ar[rr]&&\circ&~~~\circ~~~\circ
}&
$\begin{array}{*{12}{c@{\hspace{0.3pt}}}}
&&&\scirc&&\scirc&&\scirc&&\scirc&&\sbullet\\[-5pt]
&&&\scirc&&\scirc&&\scirc&&\scirc&&\sbullet\\[-5pt]
&&\sbullet&&\scirc&&\scirc&&\scirc&&\scirc\\[-5pt]
&\scirc&&\scirc&&\scirc&&\scirc&&\scirc\\[-5pt]
\sbullet&&\scirc&&\sbullet&&\scirc&&\scirc
\end{array}$&$\mathbb{A}_3\amalg\mathbb{A}_1
\amalg\mathbb{A}_1$\\
\hline
(16)&
\makecell[c]{\xymatrix@C=5pt{\circ\ar[rr]\ar@/_/@{.}[rrd]_{}
\ar@/^/@{.}[rrrrrr]_{}&&
\circ\ar[rr]\ar[d]&&
\circ\ar[rr]&&\circ\\
&&\circ
}}&
\makecell[c]{$\begin{array}{*{12}{c@{\hspace{0.3pt}}}}
&&&\scirc&&\sbullet&&\scirc&&\scirc&&\scirc\\[-5pt]
&&&\sbullet&&\scirc&&\sbullet&&\scirc&&\sbullet\\[-5pt]
&&\scirc&&\scirc&&\scirc&&\scirc&&\scirc\\[-5pt]
&\scirc&&\scirc&&\scirc&&\scirc&&\scirc\\[-5pt]
\scirc&&\scirc&&\sbullet&&\scirc&&\scirc
\end{array}
\hspace{0.3mm}
\begin{array}{*{12}{c@{\hspace{0.3pt}}}}
&&&\sbullet&&\scirc&&\sbullet&&\scirc&&\sbullet\\[-5pt]
&&&\scirc&&\sbullet&&\scirc&&\scirc&&\scirc\\[-5pt]
&&\scirc&&\scirc&&\scirc&&\scirc&&\scirc\\[-5pt]
&\scirc&&\scirc&&\scirc&&\scirc&&\scirc\\[-5pt]
\scirc&&\scirc&&\sbullet&&\scirc&&\scirc
\end{array}$} & $\mathbb{D}_5$ \\
\hline
(17)&
\makecell[c]{\xymatrix@C=5pt{\circ\ar[rr]
\ar@/^/@{.}[rrrrrr]_{}&&
\circ\ar[rr]\ar@/_/@{.}[rrd]_{}&&
\circ\ar[rr]\ar[d]&&\circ\\
&&&&\circ
}}&
\makecell[c]{$\begin{array}{*{12}{c@{\hspace{0.3pt}}}}
&&&\scirc&&\sbullet&&\scirc&&\sbullet&&\scirc\\[-5pt]
&&&\sbullet&&\scirc&&\scirc&&\scirc&&\sbullet\\[-5pt]
&&\scirc&&\scirc&&\scirc&&\scirc&&\scirc\\[-5pt]
&\scirc&&\scirc&&\scirc&&\scirc&&\scirc\\[-5pt]
\scirc&&\sbullet&&\scirc&&\scirc&&\scirc
\end{array}
\hspace{0.3mm}
\begin{array}{*{12}{c@{\hspace{0.3pt}}}}
&&&\sbullet&&\scirc&&\scirc&&\scirc&&\sbullet\\[-5pt]
&&&\scirc&&\sbullet&&\scirc&&\sbullet&&\scirc\\[-5pt]
&&\scirc&&\scirc&&\scirc&&\scirc&&\scirc\\[-5pt]
&\scirc&&\scirc&&\scirc&&\scirc&&\scirc\\[-5pt]
\scirc&&\sbullet&&\scirc&&\scirc&&\scirc
\end{array}$} & $\mathbb{D}_5$ \\
\hline
(18)&
\makecell[c]{\xymatrix@C=5pt{\circ\ar[rr]&&\circ&&
\circ\ar[ll]\ar[rr]&&\circ\\
&&&&\circ\ar[u]\ar@/_/@{.}[rru]_{}
\ar@/^/@{.}[llu]_{}
}}&
\makecell[c]{$\begin{array}{*{12}{c@{\hspace{0.3pt}}}}
&&&\scirc&&\sbullet&&\scirc&&\scirc&&\scirc\\[-5pt]
&&&\sbullet&&\scirc&&\scirc&&\scirc&&\sbullet\\[-5pt]
&&\scirc&&\scirc&&\scirc&&\scirc&&\scirc\\[-5pt]
&\scirc&&\sbullet&&\scirc&&\scirc&&\scirc\\[-5pt]
\scirc&&\scirc&&\sbullet&&\scirc&&\scirc
\end{array}
\hspace{0.3mm}
\begin{array}{*{12}{c@{\hspace{0.3pt}}}}
&&&\sbullet&&\scirc&&\scirc&&\scirc&&\sbullet\\[-5pt]
&&&\scirc&&\sbullet&&\scirc&&\scirc&&\scirc\\[-5pt]
&&\scirc&&\scirc&&\scirc&&\scirc&&\scirc\\[-5pt]
&\scirc&&\sbullet&&\scirc&&\scirc&&\scirc\\[-5pt]
\scirc&&\scirc&&\sbullet&&\scirc&&\scirc
\end{array}$}& $\mathbb{D}_5$ \\
\hline
(19)&
\makecell[c]{\xymatrix@C=5pt{\circ&&\circ\ar[ll]&&
\circ\ar[ll]\ar[rr]&&\circ\\
&&&&\circ\ar[u]\ar@/_/@{.}[rru]_{}
\ar@/^/@{.}[llu]_{}
}}&
$\begin{array}{*{12}{c@{\hspace{0.3pt}}}}
&&&\scirc&&\sbullet&&\scirc&&\scirc&&\scirc\\[-5pt]
&&&\sbullet&&\scirc&&\scirc&&\scirc&&\sbullet\\[-5pt]
&&\scirc&&\scirc&&\scirc&&\scirc&&\scirc\\[-5pt]
&\scirc&&\sbullet&&\scirc&&\scirc&&\scirc\\[-5pt]
\scirc&&\sbullet&&\scirc&&\scirc&&\scirc
\end{array}
\hspace{0.3mm}
\begin{array}{*{12}{c@{\hspace{0.3pt}}}}
&&&\sbullet&&\scirc&&\scirc&&\scirc&&\sbullet\\[-5pt]
&&&\scirc&&\sbullet&&\scirc&&\scirc&&\scirc\\[-5pt]
&&\scirc&&\scirc&&\scirc&&\scirc&&\scirc\\[-5pt]
&\scirc&&\sbullet&&\scirc&&\scirc&&\scirc\\[-5pt]
\scirc&&\sbullet&&\scirc&&\scirc&&\scirc
\end{array}$& $\mathbb{D}_5$ \\
\hline
(20)&
\makecell[c]{\xymatrix@C=5pt{\circ\ar[rr]\ar@/^/@{.}[rrrr]_{}
\ar@/_/@{.}[rrd]_{}&&\circ\ar[rr]
\ar[d]\ar@/^/@{.}[rrrr]_{}&&
\circ\ar[rr]&&\circ\\
&&\circ
}}&
$\begin{array}{*{12}{c@{\hspace{0.3pt}}}}
&&&\scirc&&\scirc&&\scirc&&\scirc&&\scirc\\[-5pt]
&&&\sbullet&&\scirc&&\sbullet&&\scirc&&\sbullet\\[-5pt]
&&\scirc&&\scirc&&\scirc&&\scirc&&\scirc\\[-5pt]
&\scirc&&\scirc&&\scirc&&\scirc&&\scirc\\[-5pt]
\sbullet&&\scirc&&\sbullet&&\scirc&&\scirc
\end{array}
\hspace{0.3mm}
\begin{array}{*{12}{c@{\hspace{0.3pt}}}}
&&&\sbullet&&\scirc&&\sbullet&&\scirc&&\sbullet\\[-5pt]
&&&\scirc&&\scirc&&\scirc&&\scirc&&\scirc\\[-5pt]
&&\scirc&&\scirc&&\scirc&&\scirc&&\scirc\\[-5pt]
&\scirc&&\scirc&&\scirc&&\scirc&&\scirc\\[-5pt]
\sbullet&&\scirc&&\sbullet&&\scirc&&\scirc
\end{array}$ & strictly shod\\
\hline
(21)&
\makecell[c]{\xymatrix@C=5pt{\circ\ar[rr]&&\circ\ar[rr]
\ar[d]&&\circ~~~\circ\\
&&\circ
}}&
$\begin{array}{*{12}{c@{\hspace{0.3pt}}}}
&&&\scirc&&\scirc&&\sbullet&&\scirc&&\scirc\\[-5pt]
&&&\scirc&&\scirc&&\sbullet&&\scirc&&\scirc\\[-5pt]
&&\scirc&&\scirc&&\scirc&&\sbullet&&\scirc\\[-5pt]
&\scirc&&\scirc&&\scirc&&\scirc&&\sbullet\\[-5pt]
\sbullet&&\scirc&&\scirc&&\scirc&&\scirc
\end{array}
\hspace{0.3mm}
\begin{array}{*{12}{c@{\hspace{0.3pt}}}}
&&&\scirc&&\scirc&&\scirc&&\sbullet&&\scirc\\[-5pt]
&&&\scirc&&\scirc&&\scirc&&\scirc&&\sbullet\\[-5pt]
&&\scirc&&\scirc&&\scirc&&\scirc&&\sbullet\\[-5pt]
&\scirc&&\scirc&&\scirc&&\scirc&&\sbullet\\[-5pt]
\sbullet&&\scirc&&\scirc&&\scirc&&\scirc
\end{array}
\hspace{0.3mm}
\begin{array}{*{12}{c@{\hspace{0.3pt}}}}
&&&\scirc&&\scirc&&\scirc&&\scirc&&\sbullet\\[-5pt]
&&&\scirc&&\scirc&&\scirc&&\sbullet&&\scirc\\[-5pt]
&&\scirc&&\scirc&&\scirc&&\scirc&&\sbullet\\[-5pt]
&\scirc&&\scirc&&\scirc&&\scirc&&\sbullet\\[-5pt]
\sbullet&&\scirc&&\scirc&&\scirc&&\scirc
\end{array}$&$\mathbb{D}_4
\amalg\mathbb{A}_1$\\
\hline
(22)&
\makecell[c]{\xymatrix@C=5pt{\circ\ar[rr]&&\circ\ar[rr]
&&\circ~~~\circ\\
&&\circ\ar[u]
}}&
$\begin{array}{*{12}{c@{\hspace{0.3pt}}}}
&&&\scirc&&\scirc&&\sbullet&&\scirc&&\scirc\\[-5pt]
&&&\scirc&&\scirc&&\scirc&&\sbullet&&\scirc\\[-5pt]
&&\scirc&&\scirc&&\scirc&&\sbullet&&\scirc\\[-5pt]
&\scirc&&\scirc&&\scirc&&\scirc&&\sbullet\\[-5pt]
\sbullet&&\scirc&&\scirc&&\scirc&&\scirc
\end{array}
\hspace{0.3mm}
\begin{array}{*{12}{c@{\hspace{0.3pt}}}}
&&&\scirc&&\scirc&&\scirc&&\sbullet&&\scirc\\[-5pt]
&&&\scirc&&\scirc&&\sbullet&&\scirc&&\scirc\\[-5pt]
&&\scirc&&\scirc&&\scirc&&\sbullet&&\scirc\\[-5pt]
&\scirc&&\scirc&&\scirc&&\scirc&&\sbullet\\[-5pt]
\sbullet&&\scirc&&\scirc&&\scirc&&\scirc
\end{array}
\hspace{0.3mm}
\begin{array}{*{12}{c@{\hspace{0.3pt}}}}
&&&\scirc&&\scirc&&\scirc&&\scirc&&\sbullet\\[-5pt]
&&&\scirc&&\scirc&&\scirc&&\scirc&&\sbullet\\[-5pt]
&&\scirc&&\scirc&&\scirc&&\scirc&&\sbullet\\[-5pt]
&\scirc&&\scirc&&\scirc&&\scirc&&\sbullet\\[-5pt]
\sbullet&&\scirc&&\scirc&&\scirc&&\scirc
\end{array}$&$\mathbb{D}_4
\amalg\mathbb{A}_1$\\
\hline
(23)&
\makecell[c]{\xymatrix@C=5pt{\circ\ar[rr]&&\circ&&\circ\ar[ll]
~~~\circ\\
&&\circ\ar[u]
}}&
$\begin{array}{*{12}{c@{\hspace{0.3pt}}}}
&&&\scirc&&\scirc&&\scirc&&\sbullet&&\scirc\\[-5pt]
&&&\scirc&&\scirc&&\scirc&&\sbullet&&\scirc\\[-5pt]
&&\scirc&&\scirc&&\scirc&&\sbullet&&\scirc\\[-5pt]
&\scirc&&\scirc&&\scirc&&\scirc&&\sbullet\\[-5pt]
\sbullet&&\scirc&&\scirc&&\scirc&&\scirc
\end{array}$&$\mathbb{D}_4
\amalg\mathbb{A}_1$\\
\hline
(24)&
\makecell[c]{\xymatrix@C=5pt{\circ&&\circ\ar[ll]\ar[rr]\ar[d]&&
\circ~~~\circ\\
&&\circ
}}&
$\begin{array}{*{12}{c@{\hspace{0.3pt}}}}
&&&\scirc&&\scirc&&\scirc&&\sbullet&&\scirc\\[-5pt]
&&&\scirc&&\scirc&&\scirc&&\sbullet&&\scirc\\[-5pt]
&&\scirc&&\scirc&&\scirc&&\scirc&&\sbullet\\[-5pt]
&\scirc&&\scirc&&\scirc&&\scirc&&\sbullet\\[-5pt]
\sbullet&&\scirc&&\scirc&&\scirc&&\scirc
\end{array}$&$\mathbb{D}_4
\amalg\mathbb{A}_1$\\
\hline
(25)&
\makecell[c]{\xymatrix@C=5pt{\circ\ar[rr]&&\circ\ar[rr]
&&\circ&~~~\circ\\
&&\circ\ar[u]\ar@/_/@{.}[rru]_{}
}}&
$\begin{array}{*{12}{c@{\hspace{0.3pt}}}}
&&&\scirc&&\scirc&&\scirc&&\scirc&&\scirc\\[-5pt]
&&&\sbullet&&\scirc&&\scirc&&\scirc&&\sbullet\\[-5pt]
&&\sbullet&&\scirc&&\scirc&&\scirc&&\scirc\\[-5pt]
&\scirc&&\scirc&&\scirc&&\scirc&&\scirc\\[-5pt]
\sbullet&&\scirc&&\sbullet&&\scirc&&\scirc
\end{array}
\hspace{0.3mm}
\begin{array}{*{12}{c@{\hspace{0.3pt}}}}
&&&\sbullet&&\scirc&&\scirc&&\scirc&&\sbullet\\[-5pt]
&&&\scirc&&\scirc&&\scirc&&\scirc&&\scirc\\[-5pt]
&&\sbullet&&\scirc&&\scirc&&\scirc&&\scirc\\[-5pt]
&\scirc&&\scirc&&\scirc&&\scirc&&\scirc\\[-5pt]
\sbullet&&\scirc&&\sbullet&&\scirc&&\scirc
\end{array}$&$\mathbb{D}_4
\amalg\mathbb{A}_1$\\
\hline
%(26)&
%\makecell[c]{\xymatrix@C=5pt{\circ\ar[rr]\ar@/^/@{.}[rrrr]_{}
%&&\circ\ar[rr]
%&&\circ&~~~\circ\\
%&&\circ\ar[u]\ar@/_/@{.}[rru]_{}
%}}&
%$\begin{array}{*{12}{c@{\hspace{0.3pt}}}}
%&&&\scirc&&\scirc&&\scirc&&\scirc&&\sbullet\\[-5pt]
%&&&\scirc&&\scirc&&\scirc&&\scirc&&\sbullet\\[-5pt]
%&&\scirc&&\scirc&&\scirc&&\scirc&&\scirc\\[-5pt]
%&\scirc&&\scirc&&\scirc&&\scirc&&\sbullet\\[-5pt]
%\sbullet&&\scirc&&\sbullet&&\scirc&&\scirc
%\end{array}$&$\mathbb{D}_4
%\amalg\mathbb{A}_1$\\
%\hline
\end{tabular}
\]

\[\begin{tabular}{|c|c|c|c|}
 \hline
No. &
silted algebras & 2-term silting complexes & tilted type\\
\hline

(26)&
\makecell[c]{\xymatrix@C=5pt{\circ\ar[rr]\ar@/^/@{.}[rrrr]_{}
&&\circ\ar[rr]
&&\circ&~~~\circ\\
&&\circ\ar[u]\ar@/_/@{.}[rru]_{}
}}&
$\begin{array}{*{12}{c@{\hspace{0.3pt}}}}
&&&\scirc&&\scirc&&\scirc&&\scirc&&\sbullet\\[-5pt]
&&&\scirc&&\scirc&&\scirc&&\scirc&&\sbullet\\[-5pt]
&&\scirc&&\scirc&&\scirc&&\scirc&&\scirc\\[-5pt]
&\scirc&&\scirc&&\scirc&&\scirc&&\sbullet\\[-5pt]
\sbullet&&\scirc&&\sbullet&&\scirc&&\scirc
\end{array}$&$\mathbb{D}_4
\amalg\mathbb{A}_1$\\
\hline
(27)&
\makecell[c]{\xymatrix@C=5pt{\circ\ar[rr]&&\circ&~~~\circ\\
\circ\ar[u]\ar[rr]\ar@//@{.}[rru]_{}&&\circ\ar[u]
}}&
$\begin{array}{*{12}{c@{\hspace{0.3pt}}}}
&&&\scirc&&\scirc&&\sbullet&&\scirc&&\scirc\\[-5pt]
&&&\scirc&&\scirc&&\sbullet&&\scirc&&\scirc\\[-5pt]
&&\scirc&&\scirc&&\scirc&&\scirc&&\scirc\\[-5pt]
&\scirc&&\scirc&&\scirc&&\scirc&&\sbullet\\[-5pt]
\sbullet&&\scirc&&\sbullet&&\scirc&&\scirc
\end{array}
\hspace{0.3mm}
\begin{array}{*{12}{c@{\hspace{0.3pt}}}}
&&&\scirc&&\scirc&&\scirc&&\sbullet&&\scirc\\[-5pt]
&&&\scirc&&\scirc&&\sbullet&&\scirc&&\sbullet\\[-5pt]
&&\scirc&&\scirc&&\scirc&&\scirc&&\scirc\\[-5pt]
&\scirc&&\scirc&&\scirc&&\scirc&&\sbullet\\[-5pt]
\sbullet&&\scirc&&\scirc&&\scirc&&\scirc
\end{array}
\hspace{0.3mm}
\begin{array}{*{12}{c@{\hspace{0.3pt}}}}
&&&\scirc&&\scirc&&\sbullet&&\scirc&&\sbullet\\[-5pt]
&&&\scirc&&\scirc&&\scirc&&\sbullet&&\scirc\\[-5pt]
&&\scirc&&\scirc&&\scirc&&\scirc&&\scirc\\[-5pt]
&\scirc&&\scirc&&\scirc&&\scirc&&\sbullet\\[-5pt]
\sbullet&&\scirc&&\scirc&&\scirc&&\scirc
\end{array}$&$\mathbb{D}_4
\amalg\mathbb{A}_1$\\
\hline
\end{tabular}
\]

To summarise, there are 62 silted algebras of type $Q$, forming 6 families:
\begin{itemize}
\item[(\romannumeral1)] tilted algebras of type $\mathbb{D}_5$:
$\rm(I)$, $(1)$,
 $(16)-(19)$;

\item[(\romannumeral2)] tilted algebras of type $\mathbb{D}_4\amalg\mathbb{A}_1$: $(5)$, $(21)-(26)$, $(27)$;

\item[(\romannumeral3)] tilted algebras of type $\mathbb{A}_4\amalg\mathbb{A}_1$: $(6)$, $(7)$, $(8)$;

\item[(\romannumeral4)] tilted algebras of type $\mathbb{A}_3\amalg\mathbb{A}_2$: $(9)$, $(10)$, $(11)$;

\item[(\romannumeral5)] tilted algebras of type $\mathbb{A}_3\amalg\mathbb{A}_1\amalg\mathbb{A}_1$: $(12)-(15)$;

\item[(\romannumeral6)]  strictly shod algebras:
    $(2)$, $(3)$, $(4)$, $(20)$.
\end{itemize}

%\cite{Soergel11}

%\bibliographystyle{amsplain}
%\bibliography{stanYang.bib}

\begin{thebibliography}{1}

\bibitem{AI} T. Aihara, O. Iyama. Silting mutation in triangulated categories. J. Lond. Math. Soc., 2012, 85(3): 633-668.
\bibitem{AIR} T. Adachi, O. Iyama, I. Reiten. $\tau $-tilting theory. Compos. Math., 2014, 150(3): 415-452.
%\bibitem{APR} M. Auslander, M. I. Platzeck, I. Reiten. Coxeter functors without diagrams. Trans. Amer. Math. Soc., 1979, 250: 1-46.
\bibitem{ASS} I. Assem, A. Skowronski, D. Simson. Elements of the Representation Theory of Associative Algebras: Volume 1: Techniques of Representation Theory. 2006.

%\bibitem{B} K. Bongartz. Tilted algebras. Lecture Notes in Math.No.903, Springer-Verlag, Berlin, Heidelberg, New York, 1981, pp.26-38.
%\bibitem{B} K. Bongartz. Tilted algebras. Lecture Notes in Math., Springer-Verlag, Berlin, Heidelberg, New York. 1981, 903: 26-38.

%\bibitem{BGP}   I. N. Bernstein, I. M. Gelfand, V. A. Ponomarev. Coxeter functors and Gabriel¡¯s theorem. Russian Math. Surveys., 1973, 28: 19-33.

%\bibitem{BR} M. C. Butler, C. M. Ringel. Auslander-reiten sequences with few middle terms and applications to string algebrass. Comm. Algebra., 1987, 15(1-2): 145-179.
\bibitem{BRT} A. B. Buan, I. Reiten,  H. Thomas. Three kinds of mutation. J. Algebra, 2011, 339(1): 97-113.

%\bibitem{BT} J. B¨¦elanger, C. Tosar. Shod string algebras. Comm. Algebra., 2005, 33(8): 2465-2487.
\bibitem{BZ}  A. B. Buan, Y. Zhou. Silted algebras. Adv. Math., 2016, 303: 859-887.
\bibitem{BZ_2016} A. B. Buan, Y. Zhou. A silting theorem. J. Pure Appl. Algebra,  2016, 220(7): 2748-2770.
\bibitem{C} F. U. Coelho. Shod algebras. IME-USP, 2001,5(1):25-61.
\bibitem{CL} F. U. Coelho, M. Lanzilotta. Algebras with small homological dimensions. Manuscripta Math., 1999, 100(1): 1-11.
%\bibitem{H} D. Happel. Quasitilted algebras. In CMS Conf. Proc., 1998, 23: 55-82.
%\bibitem{H_1998}  D. Happel. Triangulated categories in the representation of finite dimensional algebras. London Math.soc.lect.note Ser., 1988.

%\bibitem{HL} F. Huard, S. Liu. Tilted string algebras. J. Pure Appl. Algebra., 2000, 153(2): 151-164.
\bibitem{HR} D. Happel, C. M. Ringel. Construction of tilted algebras. In Representations of algebras. Springer, Berlin, Heidelberg. 1981, 125-144.
%\bibitem{HR-1982}  D. Happel, C. M. Ringel. Tilted algebras. Trans. Amer. Math. soc., 1982, 274(2): 399-443.
%\bibitem{HRS} D. Happel, I. Reiten, S. O. Smal\o.  Tilting in abelian categories and quasitilted algebras. Amer. Math. soc., 1996.

%\bibitem{kN} B. Keller, P. Nicol¨¢s. Cluster hearts and cluster tilting objects. Work in preparation. 2011.
\bibitem{kN} B. Keller, P. Nicol'as. Cluster hearts and cluster tilting objects, work in preparation. Talk notes based on this work are available at http://www.iaz.uni-stuttgart.de/LstAGeoAlg/activities/t-workshop/Nicolas Notes.pdf.

\bibitem{kv} B. Keller, D. Vossieck. Aisles in derived categories. Bull. Soc. Math. Belg. S¨¦r. A., 1998, 40(2): 239-253.


\bibitem{kY} S. Koenig, D. Yang. Silting objects, simple-minded collections, t-structures and co-t-structures for finite-dimensional algebras. Doc. Math., 2014, 19(1): 403-438.
\bibitem{ONFR} M. A. A. Obaid, S. K. Nauman, W. M. Fakieh, C. M. Ringel. The numbers of support-tilting modules for a Dynkin algebra. J. Integer Seq., 2014, 18(10).
%\bibitem{RS} I. Reiten, A. Skowronski. Characterizations of algebras with small homological dimensions. Adv. Math., 2003, 179(1): 122-154.

\end{thebibliography}

%\end{document}

\def\cprime{$'$}
\providecommand{\bysame}{\leavevmode\hbox to3em{\hrulefill}\thinspace}
\providecommand{\MR}{\relax\ifhmode\unskip\space\fi MR }
\providecommand{\MRhref}[2]{%
  \href{http://www.ams.org/mathscinet-getitem?mr=#1}{#2}
}
\providecommand{\href}[2]{#2}

\end{document}